\documentclass[12pt,a4paper]{article}
\usepackage{color}
\usepackage[colorlinks=true]{hyperref}
\usepackage[margin=1in]{geometry}
\usepackage{fullpage,mathrsfs,graphicx,framed,color,amssymb,amsmath,amsthm}
\usepackage[boxed, noline, ruled, linesnumbered]{algorithm2e}
\usepackage{epsfig, graphicx}
\usepackage{latexsym,amsfonts,amsbsy,amssymb}
\usepackage{amsmath,amsthm}
\usepackage{enumerate}
\usepackage{cleveref}
\usepackage{tikz}
\usepackage{stmaryrd}
\usetikzlibrary{positioning,calc,decorations.pathreplacing,arrows.meta}
\usepackage{siunitx}
\usepackage{array}
\usepackage{diagbox} 
\usepackage{bm} 
\usepackage{longtable}
\usepackage{caption}
\usepackage[backend=biber, style=numeric]{biblatex}

\newtheorem{thm}{Theorem}
\newtheorem{Def}{Definition}
\newtheorem{cor}{Corollary}
\newtheorem{lem}{Lemma}

\newtheorem{problem}{Problem}
\numberwithin{thm}{section}
\numberwithin{Def}{section}
\numberwithin{cor}{section}
\numberwithin{lem}{section}
\numberwithin{exm}{section}
\numberwithin{rem}{section}
\numberwithin{equation}{section} 

\usepackage{changes} 
\makeatletter 

\@addtoreset{equation}{section}
\makeatother 
\allowdisplaybreaks 

\newtheorem{remark}{Remark}[section]

\allowdisplaybreaks 

\begin{document}
\title{Continuous Finite Element Method For Maxwell Eigenvalue Problems With 
Regular Decomposition Technique\footnote{This work was supported by 
the Strategic Priority Research Program of the Chinese Academy of 
Sciences (XDB0640000, XDB0640300),  National Key Laboratory of Computational Physics 
(No. 6142A05230501), National Natural Science Foundations of 
China (NSFC 1233000214), National Center for Mathematics and Interdisciplinary Science, CAS.}}
\author{Feiyi Liao\footnote{SKLMS, Academy of Mathematics and Systems Science,
Chinese Academy of Sciences, No.55, Zhongguancun Donglu, 
Beijing 100190, China, and School of
Mathematical Sciences, University of Chinese Academy
of Sciences, Beijing, 100049, China (liaofeiyi@amss.ac.cn).},\ \ \
Haochen Liu\footnote{SKLMS, Academy of Mathematics and Systems Science,
Chinese Academy of Sciences, No.55, Zhongguancun Donglu, 
Beijing 100190, China, and School of
Mathematical Sciences, University of Chinese Academy
of Sciences, Beijing, 100049, China (liuhaochen@lsec.cc.ac.cn).}\ \ \
and \ \ Hehu Xie\footnote{SKLMS, Academy of Mathematics and Systems Science,
Chinese Academy of Sciences, No.55, Zhongguancun Donglu, Beijing 100190, 
China, and School of Mathematical Sciences, University of Chinese Academy
of Sciences, Beijing, 100049, China (hhxie@lsec.cc.ac.cn).}}
\date{}
\maketitle
\begin{abstract}
With the regular decomposition technique, 
we decompose the space $\mathbf{H}_0^s(\mathbf{curl}; \Omega)$ 
into the sum of a vector potential space and the gradient of a scalar space, both possessing higher regularity. 
Based on this new high order regular decomposition, 
a novel numerical method using standard high order 
Lagrange finite elements is designed for solving Maxwell 
eigenvalue problems. 
Specifically, the full convergence orders of the eigenpair approximations are proved for the proposed numerical method. 
Finally, numerical examples are provided to validate the proposed scheme and confirm the theoretical convergence results.
\end{abstract}
\section{Introduction}\label{Section_Intro}
Since their formulation in 1873,
Maxwell's equations and the corresponding eigenvalue problems have been fundamental to describing electromagnetic fields \cite{key22}.
In the realm of numerical simulation, the finite element method has become a standard tool for practical electromagnetic applications.
N\'{e}d\'{e}lec introduced the first and second families of edge elements in 1980 \cite{key24} and 1986 \cite{key25}, respectively.
Both families are conforming elements in  $\mathbf{H}(\mathbf{curl};\Omega)$ 
and their degrees of freedom are associated with edges,
which makes the lowest-order elements straightforward to construct.\cite{key28}.
However, for higher-order edge elements,
constructing systematic representations of basis functions corresponding to their degrees of freedom remains a challenge.
Most implementations of high-order edge elements do not conform to the standard finite element framework \cite{key33,key34},
as the basis functions often fail to satisfy the duality relation with the degrees 
of freedom defined for N\'{e}d\'{e}lec elements.
Consequently, developing higher-order discrete elements that rely solely on point-based degrees of freedom remains a compelling research objective.

Motivated by the ease of implementation in standard FEM packages and their excellent approximation properties for $\mathrm{H}^1$ functions, researchers have long attempted to discretize $\mathbf{H}(\mathbf{curl};\Omega)$ using continuous ($\mathrm{C}^0$) finite elements, such as conventional Lagrange elements.
However, a significant drawback is that when the solution does not belong to $\mathrm{H}^1(\Omega)$—which is often the case for Maxwell systems on domains with concave corners or re-entrant edges \cite{key3}—discrete problems based on continuous finite elements frequently converge to incorrect solutions \cite{key35,key5}.
To mitigate this, a commonly employed strategy is to add a weighted divergence stabilization term to the variational formulation, which can yield a correct $\mathrm{L}^2$ approximation when appropriate weighting coefficients and the corresponding $\mathrm{C}^0$ element function space are selected \cite{key36}.
Some special composite triangulations of the domain have also been adopted, such as Alfeld mesh, Powell-Sabin mesh and Worsey-Farin mesh \cite{key36,key39,key40}.
Additionally, Duan et al. proposed a method incorporating a projection operator into the penalty term \cite{key37}. 
Bonito and Guermond considered a mixed method controlling the divergence of the discrete solution 
in fractional Sobolev 
space $\mathrm{H}^{-\alpha}$ with $\alpha\in(1/2,1)$ \cite{key41}. 
While effective for specific problems, these methods are often cumbersome to implement.  
Special meshes are also impractical for large-scale computations due to numerous constraints on the triangulation \cite{key42}.

The aim of this paper is to design a robust continuous finite element method for 
the Maxwell eigenvalue problem utilizing the regular decomposition technique.
The core concept involves using a vector potential lifting operator to isolate the portion of the solution lacking $\mathbf{H}^1$ regularity into the gradient of the $\mathrm{H}^1(\Omega)$ scalar space \cite{key3}.
In particular, Pasciak and Zhao established an $\mathrm{L}^2$-stable regular 
decomposition \eqref{order-1-decomposion}, which represents functions in $\mathbf{H}_0(\mathbf{curl};\Omega)$ as the sum of $\mathbf{u}\in\mathbf{H}_0^{1}(\Omega)$ and the gradient of $p\in\mathrm{H}_0^{1}(\Omega)$ \cite{key23}.
The discrete scheme based on this decomposition employs two continuous finite element spaces to discretize the $\mathbf{H}^1$ vector space and the $\mathrm{H}^1$ scalar space.
Unfortunately, our numerical experiments in Section \ref{Section_Example_Cube} reveal that even when high-order Lagrange finite elements are used for smoothing problems,
this discrete scheme with homogeneous Dirichlet boundary conditions fails to achieve full convergence orders for eigenvalue approximations, as shown in Figure \ref{fig:cube H01}. 
This limitation may be attributed to inherent regularity
constraints of the $\mathbf{H}_0^1$-conforming potential functions in the regular decomposition.

Building on the regularity results for 
the de Rham complex on bounded Lipschitz domains established by Costabel and McIntosh \cite{key16},
we propose a novel high order regular decomposition for  $\mathbf{H}_0^s(\mathbf{curl};\Omega)$ with $s>1/2$.   
Specifically, by relaxing the boundary condition of the vector potential 
to a tangential trace vanishing condition, we demonstrate that any function 
$\boldsymbol{\xi}\in\mathbf{H}_0^s(\mathbf{curl};\Omega)$, whose components 
and curl possess regularity of order $s>1/2$, can be decomposed into 
a vector function $\mathbf{u}\in\mathbf{H}^{s+1}(\Omega)$ 
and the gradient of a scalar function $p\in\mathrm{H}^{s+1}(\Omega)$.
This decomposition allows for the direct discretization of the Maxwell eigenvalue problem using standard conforming Lagrange finite elements on regular meshes, leading 
to the corresponding numerical scheme. 

Based on the spectral convergence theory of compact operators \cite{key19}, a rigorous convergence analysis is provided in this paper.
We prove that our discrete scheme is free of spurious zero modes and exhibits spectral convergence when the vector and scalar potential spaces are discretized using Lagrange finite elements of order $k$ and $k+1$ with corresponding boundary conditions, respectively.
Furthermore, 
we also establish the optimal convergence rates for the eigenpair approximations 
in the $\mathbf{H}(\mathbf{curl};\Omega)$-norm. 
These rates are determined by the regularity of the exact eigenfunctions and the polynomial 
degree of the Lagrange elements used.
Numerical experiments on three representative polyhedrons  with varying concavity validate the efficiency of the proposed scheme and confirm the theoretical convergence rates.
Unlike existing approaches requiring specialized composite meshes or complex stabilization,
our numerical scheme is formally concise, attains the theoretical optimal convergence rates using only standard nodal Lagrange elements on regular tetrahedral meshes, and can be easily implemented within existing standard finite element packages.

The remainder of this paper is organized as follows.
Section \ref{Section_Pre} introduces the regular decomposition and preliminary spectral approximation results. 
As an application of regular decomposition, the discretization scheme using continuous finite elements is designed in Section \ref{Section_Problem}.
Sections \ref{Section_Operator} and \ref{Section_Error_Estimates} are devoted to presenting the spectral convergence analysis and error estimates for eigenpair approximations.
Finally, numerical examples are provided in Section \ref{Section_Numerical}
to validate the corresponding theoretical results. 

\section{Preliminaries}\label{Section_Pre}
Unless otherwise specified,
let $\Omega\subset\mathbb{R}^3$ be an open bounded contractible polyhedron with Lipschitz boundary $\partial\Omega$,
which means it has the homotopy
type of a one-point space \cite[P. 366]{key11}.  
Consequently, any domain that is topologically homeomorphic to the unit ball is contractible.

Assuming that the medium in $\Omega$ is homogeneous and isotropic, 
the electric permittivity $\varepsilon$ and the magnetic permeability $\mu$ are both positive constants in $\Omega$.
In such case, we can assume without loss of generality that $\varepsilon = \mu = 1$.
Let $\mathbf E$ represent the electric field.  
In this paper, we are concerned with the following Maxwell eigenvalue problem: 
Find $\lambda\in\mathbb R$ and $\mathbf{0}\neq \mathbf E:\Omega\to \mathbb{R}^3$ such that 
\begin{eqnarray}
\label{maxwell-eigenvalue-problem}
\left\{
\begin{array}{rll}
\nabla\times\nabla\times\mathbf{E}&=\lambda\mathbf{E},\quad&\mathrm{in}~\Omega, \\
\nabla\cdot\mathbf{E}&=\mathbf0,\quad&\mathrm{in}~\Omega, \\
\mathbf{E}\times\mathbf n &=\mathbf0,\quad&\mathrm{on}~\partial\Omega,
\end{array}
\right.
\end{eqnarray}
where $\mathbf n$ denotes 
the unit outward normal vector on $\partial\Omega$.

First, we introduce the functional settings which will be used here.  
Bold characters will denote vector valued functions and their corresponding functional spaces.
Given a three-dimensional domain $D$, for any $p\geq1$, $\mathrm{L}^p(D)$ 
denotes the classical Lebesgue function space 
and $\mathbf{L}^p(D) := [\mathrm{L}^p(D)]^3$.
For any nonnegative $s$, $\mathrm{H}^s(D)$ denotes the standard 
Sobolev space and $\mathbf{H}^s(D) := [\mathrm{H}^s(D)]^3$. 
For the analysis in this paper, 
we introduce the following classical Hilbert spaces 
\begin{align*}
\mathrm{H}_{0}^{1}(\Omega)&:=\left\{q\in\mathrm{H}^{1}(\Omega):
q=0\mathrm{~on~}\partial\Omega\right\},\\
\mathbf{H}(\mathrm{div},\Omega)&:=\left\{\mathbf{v}\in\mathbf{L}^2(\Omega):
\nabla\cdot\mathbf{v}\in\mathrm{L}^2(\Omega)\right\},\\
\mathbf{H}_0(\mathrm{div},\Omega)&:=\{\mathbf{v}\in\mathbf{H}(\mathrm{div},\Omega):
\mathbf{v}\cdot\mathbf{n}=0\mathrm{~on~}\partial\Omega\},\\
\mathbf{H}(\mathrm{div}^{0},\Omega)&:=\{\mathbf{v}\in\mathbf{H}(\mathrm{div},\Omega):
\nabla\cdot\mathbf{v}=0\},\\
\mathbf{H}_{0}(\mathrm{div}^{0},\Omega)&:
=\mathbf{H}_0(\mathrm{div},\Omega)\cap\mathbf{H}(\mathrm{div}^0,\Omega),\\
\mathbf{H}(\mathbf{curl},\Omega)&:=\left\{\mathbf{v}\in\mathbf{L}^2(\Omega):
\nabla\times\mathbf{v}\in\mathbf{L}^2(\Omega)\right\},\\
\mathbf{H}_0(\mathbf{curl},\Omega)&:=\{\mathbf{v}\in\mathbf{H}(\mathbf{curl},\Omega):
\mathbf{v}\times\mathbf{n}=\mathbf{0}\mathrm{~on~}\partial\Omega\},\\
\mathbf{H}(\mathbf{curl}^0,\Omega)&:=\{\mathbf{v}\in\mathbf{H}(\mathbf{curl},\Omega):
\nabla\times\mathbf{v}=\mathbf{0}\},\\
\mathbf{H}_0(\mathbf{curl}^0,\Omega)&:
=\mathbf{H}_0(\mathbf{curl},\Omega)\cap\mathbf{H}(\mathbf{curl}^0,\Omega),\\
\mathbf{H}^s(\mathbf{curl},\Omega)&:=\left\{\mathbf{v}\in\mathbf{H}^s(\Omega):
\nabla\times\mathbf{v}\in\mathbf{H}^s(\Omega)\right\},\\
\mathbf{H}_0^s(\mathbf{curl};\Omega)&:
=\mathbf{H}_0(\mathbf{curl},\Omega)\cap\mathbf{H}^s(\mathbf{curl},\Omega), 
\end{align*}
with the corresponding norm definitions 
\begin{align*}
\|\mathbf{v}\|^2_{\mathbf{H}(\mathrm{div},\Omega)}&=
\|\mathbf{v}\|_{\mathbf{L}^2(\Omega)}^2+\|\nabla\cdot\mathbf{v}\|_{\mathrm{L}^2(\Omega)}^2,\\
\|\mathbf{v}\|^2_{\mathbf{H}(\mathbf{curl},\Omega)}&=
\|\mathbf{v}\|_{\mathbf{L}^2(\Omega)}^2+\|\nabla\times\mathbf{v}\|_{\mathbf{L}^2(\Omega)}^2,\\
\|\mathbf{v}\|^2_{\mathbf{H}^s(\mathbf{curl};\Omega)}&=
\|\mathbf{v}\|_{\mathbf{H}^s(\Omega)}^2+\|\nabla\times\mathbf{v}\|_{\mathbf{H}^s(\Omega)}^2.
\end{align*}
For the sake of simplicity, the inner product of 
both $\mathrm{L}^2(\Omega)$ and $\mathbf{L}^2(\Omega)$ is denoted as $(\cdot, \cdot)$ 
Besides, the analysis for the Maxwell problems always need following space 
\begin{align*}
X(\Omega):=\mathbf{H}(\mathbf{curl};\Omega)\cap \mathbf{H}(\mathrm{div};\Omega),
\end{align*}
and its subspaces 
\begin{align*}
X_N(\Omega):=\mathbf{H}_0(\mathbf{curl};\Omega)
\cap \mathbf{H}(\mathrm{div};\Omega),\quad
X_T(\Omega):=\mathbf{H}(\mathbf{curl};\Omega)\cap \mathbf{H}_0(\mathrm{div};\Omega).
\end{align*}
The space $X(\Omega)$ is endowed with the following norm 
\begin{align*}
\|\mathbf{v}\|^{2}_{X(\Omega)}=\|\mathbf{v}\|_{\mathbf{L}^{2}(\Omega)}^2
+\|\nabla\times\mathbf{v}\|_{\mathbf{L}^{2}(\Omega)}^2
+\|\nabla\cdot\mathbf{v}\|_{\mathrm{L}^{2}(\Omega)}^2.
\end{align*}
In the paper, we will repeatedly use the following embedding theorem, 
which depends largely on the geometry of $\Omega$.

\begin{lem}(\cite[Section 2]{key3})\label{X_N-X_T-embedded-theorem}
Let $\Omega\subset\mathbb{R}^3$ be a bounded Lipschitz domain. 
Then the space $X_N(\Omega)$ and $X_T(\Omega)$ 
are compactly embedded into $\mathbf{L}^2(\Omega)$.
In particular, when $\Omega$ is convex, 
$X_N(\Omega)$ and $X_T(\Omega)$ are continuously embedded into $\mathbf{H}^1(\Omega)$.
When $\Omega$ is a polyhedron, there exists a constant $s>1/2$, 
depending only on the geometry of $\Omega$, 
such that $X_N(\Omega)$ and $X_T(\Omega)$ are continuously 
embedded into $\mathbf{H}^s(\Omega)$.
More precisely, $s$ depends on the apertures along the edges 
and the solid angles at the vertices.
\end{lem}
Next, we present the procedure of obtaining the regular decomposition 
for the space $\mathbf{H}_0^s(\mathbf{curl};\Omega)$.
As a well-known result, for the special case of $s=0$, there exists a curl-preserving 
operator from $\mathbf{H}_0(\mathbf{curl};\Omega)$ onto $\mathbf{H}_0^1(\Omega)$.
\begin{lem}(\cite[Section 2]{key23})\label{thm:Lipschitz-contractible-order_1-decomposion}
Let $\Omega$ be a bounded contractible Lipschitz domain in $\mathbb{R}^3$. Then there exist 
continuous linear operators $\mathsf{P}_0: \mathbf{H}_0(\mathbf{curl};\Omega)\mapsto \mathbf{H}_0^1(\Omega)$ 
and $\mathsf{Q}_0$: $\mathbf{H}_0(\mathbf{curl};\Omega)$ $\mapsto$  $\mathrm{H}_0^1(\Omega)$ such that
\begin{align}
\mathsf{P}_0+\mathbf{grad}\circ\mathsf{Q}_0=\mathsf{Id},
\end{align}
where $\mathsf{Id}$ denotes the identity mapping on $\mathbf{H}_0(\mathbf{curl};\Omega)$.
\end{lem}
Lemma \ref{thm:Lipschitz-contractible-order_1-decomposion} means that  
the space $\mathbf{H}_0(\mathbf{curl};\Omega)$ has following regular decomposition 
\begin{align}\label{order-1-decomposion}
\mathbf{H}_0(\mathbf{curl};\Omega)=\mathbf{H}_0^1(\Omega)+\nabla (\mathrm{H}_0^1(\Omega)).
\end{align} 
To consider the regular decomposition of the space $\mathbf{H}_0^s(\mathbf{curl};\Omega)$ 
with $s>0$, we begin by employing the potential lifting theory of arbitrary 
order 
from \cite{key16} and the de Rham complexes  
from \cite{key8}. 
\begin{lem}(\cite[Corollary 4.7]{key16})\label{forall-order-derham}
Let $\Omega$ be a bounded Lipschitz domain in $\mathbb{R}^3$ 
and $s,t>0$.
If $\mathbf{u}\in \mathbf{H}^s(\Omega)$ satisfies
$\mathbf{u}=\nabla\times \mathbf{v}$ for some $\mathbf{v}\in \mathbf{H}^t(\Omega)$, then there exists
$\mathbf{w}\in \mathbf{H}^{s+1}(\Omega)$ 
such that $\mathbf{u}=\nabla\times \mathbf{w}$,
and there is a constant $C$ independent of $\mathbf{u}$ and $\mathbf{v}$ with
\begin{align}
\|\mathbf{w}\|_{\mathbf{H}^{s+1}(\Omega)}
\leq C(\|\mathbf{u}\|_{\mathbf{H}^s(\Omega)}+\|\mathbf{v}\|_{\mathbf{H}^{t}(\Omega)}).
\end{align}
\end{lem}

\begin{lem}(\cite[Section 2]{key8})\label{thm:exact_de_rham_zero_bc}
Let $\Omega$ be a bounded Lipschitz domain in $\mathbb{R}^3$.
The following two de Rham complexes hold 
\begin{align}
\label{eq:de_rham_sequence}
&\mathrm{H}^1(\Omega) \xrightarrow{\mathbf{grad}} \mathbf{H}(\mathbf{curl};\Omega) 
\xrightarrow{\mathbf{curl}} \mathbf{H}(\mathrm{div};\Omega) \xrightarrow{\mathrm{div}} \mathrm{L}^2(\Omega), \\
\label{eq:de_rham_sequence_zero_bc}
&\mathrm{H}_0^1(\Omega) \xrightarrow{\mathbf{grad}} \mathbf{H}_0(\mathbf{curl};\Omega) 
\xrightarrow{\mathbf{curl}} \mathbf{H}_0(\mathrm{div};\Omega) \xrightarrow{\mathrm{div}} \mathrm{L}^2(\Omega).
\end{align}
When $\Omega$ is also contractible, the two de Rham complexes are exact 
in the sense of 
\begin{equation}
\mathrm{Im}(\mathbf{grad}) = \ker(\mathbf{curl}),  \quad \mathrm{Im}(\mathbf{curl}) = \ker(\mathrm{div}).
\end{equation}
\end{lem}
As an application of Lemmas \ref{forall-order-derham} and \ref{thm:exact_de_rham_zero_bc}, we can naturally obtain the following regular decomposition theorem of the space $\mathbf{H}^s(\mathbf{curl};\Omega)$ with $s\geq0$.

\begin{thm}
\label{thm:Lipschitz-contractible-order_s+1-decomposion}
Let $\Omega$ be a bounded contractible Lipschitz domain in $\mathbb{R}^3$ and $s\geq0$. 
Then there exist continuous linear operators $\mathsf{M}:\mathbf{H}^s(\mathbf{curl};\Omega)\mapsto \mathbf{H}^{s+1}(\Omega)$ and $\mathsf{N}:\mathbf{H}^s(\mathbf{curl};\Omega)\mapsto \mathrm{H}^{s+1}(\Omega)$ such that 
\begin{align}
\mathsf{M}+\mathbf{grad}\circ\mathsf{N}=\mathsf{Id},
\end{align}
where $\mathsf{Id}$ denotes the identity mapping on $\mathbf{H}^s(\mathbf{curl};\Omega)$.
Furthermore, for any $\boldsymbol{\xi}\in \mathbf{H}^s(\mathbf{curl};\Omega)$, 
the following inequalities hold 
\begin{align}\label{eq:Lipschitz-contractible-order_s+1-decomposion-estimate}
\|\mathsf{M}\boldsymbol{\xi}\|_{\mathbf{H}^{s+1}(\Omega)}
\leq C_1\|\boldsymbol{\xi}\|_{\mathbf{H}^s(\mathbf{curl};\Omega)}, \quad
\|\mathsf{N}\boldsymbol{\xi}\|_{\mathrm{H}^{s+1}(\Omega)}
\leq C_2\|\boldsymbol{\xi}\|_{\mathbf{H}^s(\mathbf{curl};\Omega)},
\end{align}
where $C_1$ and $C_2$ are constants depending on the domain $\Omega$.
\end{thm}

\begin{proof}
For any function $\boldsymbol{\xi}\in \mathbf{H}^s(\mathbf{curl};\Omega)$, 
let $\boldsymbol{\eta}=\nabla\times \boldsymbol{\xi}\in \mathbf{H}^s(\Omega)\cap \mathbf{H}(\mathrm{div}^0;\Omega)$. 
By Lemma \ref{forall-order-derham}, there exists $\mathbf{w}\in \mathbf{H}^{s+1}(\Omega)$ 
and a constant $C_1$, depending only on $\Omega$, such that $\nabla\times \mathbf{w}=\boldsymbol{\eta}$, 
and the following inequalities hold
\begin{align}\label{eq:Lipschitz-contractible-order_s+1-decomposion-estimate1}
\|\mathbf{w}\|_{\mathbf{H}^{s+1}(\Omega)}
\leq C_1(\|\boldsymbol{\eta}\|_{\mathbf{H}^s(\Omega)}+\|\boldsymbol{\xi}\|_{\mathbf{H}^s(\Omega)})
\leq \sqrt{2}C_1\|\boldsymbol{\xi}\|_{\mathbf{H}^s(\mathbf{curl};\Omega)}.
\end{align}
Then, let us define the operator $\mathsf{M}: \mathbf{H}^s(\mathbf{curl};\Omega)\mapsto 
\mathbf{H}^{s+1}(\Omega)$ as 
$\mathsf{M}\boldsymbol{\xi}=\mathbf{w}$.
From \eqref{eq:Lipschitz-contractible-order_s+1-decomposion-estimate1}, $\mathsf{M}$ is a bounded linear operator 
and satisfies $\nabla\times\mathsf{M}\boldsymbol{\xi}=\nabla\times \boldsymbol{\xi}$.

Since $\Omega$ is a bounded contractible Lipschitz domain in $\mathbb{R}^3$, and $\boldsymbol{\xi}
-\mathsf{M}\boldsymbol{\xi}\in \mathbf{H}(\mathbf{curl}^0;\Omega)$, it follows from Lemma \ref{thm:exact_de_rham_zero_bc} 
that there exists $p\in \mathrm{H}^1(\Omega)$ such that $\boldsymbol{\xi}-\mathsf{M}\boldsymbol{\xi}=\nabla p$.
Moreover, since $\boldsymbol{\xi}-\mathsf{M}\boldsymbol{\xi}\in \mathbf{H}^s(\Omega)$, we have $p\in \mathrm{H}^{s+1}(\Omega)$.
Similarly, we can define  the operator $\mathsf{N}: \mathbf{H}_0^s(\mathbf{curl};\Omega) 
\mapsto \mathrm{H}^{s+1}(\Omega)\cap \mathrm{H}_0^1(\Omega)$ as $\mathsf{N}\boldsymbol{\xi}=p$.
It is straightforward to show the operator $\mathsf{N}$ is linear. 
Furthermore, by the equivalence of the norm $\|\cdot\|_{\mathrm{H}^1(\Omega)}$ 
and the seminorm $|\cdot|_{\mathrm{H}^1(\Omega)}$ on $\mathrm{H}_0^1(\Omega)$, there is a constant $C$ such that $\mathsf{N}$ is bounded 
in the following sense 
\begin{align}
\label{eq:Lipschitz-contractible-order_s+1-decomposion-estimate2}
\|\mathsf{N}\boldsymbol{\xi}\|_{\mathrm{H}^{s+1}(\Omega)}
&\leq C \|\nabla(\mathsf{N}\boldsymbol{\xi})\|_{\mathbf{H}^s(\Omega)}
=C \|\boldsymbol{\xi}-\mathsf{M}\boldsymbol{\xi}\|_{\mathbf{H}^s(\Omega)}\nonumber\\
&\leq C\|\boldsymbol{\xi}\|_{\mathbf{H}^s(\Omega)}+C\|\mathsf{M}\boldsymbol{\xi}\|_{\mathbf{H}^s(\Omega)}
\leq C_2\|\boldsymbol{\xi}\|_{\mathbf{H}^s(\mathbf{curl};\Omega)},
\end{align}
where $C_2=C+CC_1$.
Combining \eqref{eq:Lipschitz-contractible-order_s+1-decomposion-estimate1} 
and \eqref{eq:Lipschitz-contractible-order_s+1-decomposion-estimate2} yields 
the estimate \eqref{eq:Lipschitz-contractible-order_s+1-decomposion-estimate}. 
The operator relation $\mathsf{M}+\mathbf{grad}\circ\mathsf{N}=\mathsf{Id}$
follows naturally from the definitions of $\mathsf{M}$ and $\mathsf{N}$.
\end{proof}

By utilizing Theorem \ref{thm:Lipschitz-contractible-order_s+1-decomposion}, $\mathbf{H}^s(\mathbf{curl};\Omega),s\geq0$ can be decomposed into
\begin{align}
\label{eq:Hcurl-order-s-decomposion}
\mathbf{H}^s(\mathbf{curl};\Omega)=\mathbf{H}^{s+1}(\Omega)+\nabla (\mathrm{H}^{s+1}(\Omega)),
\end{align}
where $s\geq 0$.
When tangential zero boundary condition is imposed on the Lipschitz polyhedron $\Omega$, we do not intend to consider the sum space in \eqref{eq:Hcurl-order-s-decomposion}, but rather aim to treat the boundary conditions for each potential function space separately.
Therefore, we introduce the lowest order trace theory for $\mathrm{H}^{s}(\Omega)$ on polyhedra.
For more information, please refer to \cite{key46}.
The faces of $\Omega$ are denoted by $f_j,1\leq j\leq J$, the edges by $e_\ell,1\leq \ell\leq L$,  and the vertices by $\bm{a}_i,1\leq i\leq I$.
Besides, $\mathbf{n}_j$ stands for the unit outward normal vector on each $f_j$.
Arbitrarily fix a vertex $\bm{a}$ of $\Omega$. For any two different faces $f_j$ and $f_{j^\prime}$ sharing $\bm{a}$, $1\leq j<j^\prime\leq J$, the intersection of the two planes containing $f_j$ and $f_{j^\prime}$ respectively is a line $d_{jj^\prime}$.
Let $\bm{\sigma}_{jj^\prime}$ stand for a unit vector in $d_{jj^\prime}$.
The set of all $\bm{\sigma}_{jj^\prime},1\leq j<j^\prime\leq J$ is denoted by $\Sigma_{\bm{a}}$.
For all $\bm{\sigma}$ in $\Sigma_{\bm{a}}$, we introduce the set $J(\bm{\sigma})$ of indices $j, 1\leq j\leq J$, such that $\bm{\sigma}$ is tangential to $f_j$.
For any $j$ in $J(\bm{\sigma})$, $\bm{\tau}_j$ stands for the unit vector tangential to $f_j$ which is directly orthogonal to $\bm{\sigma}$.

Given $s>1/2$, there exists a trace operator $\gamma^{(1)}$ which is continuous from $\mathrm{H}^{s}(\Omega)$ into 
$\mathbb{H}^{s(1)}(\partial\Omega):=\prod_{j=1}^{J}\mathrm{H}^{s-1/2}(f_j)$ \cite[Corollary 4.3]{key46}.
The trace operator $\gamma^{(1)}$ is not surjective, and the elements in its image space satisfy some compatibility conditions, whose geometric interpretation is reflected in the continuity at edges and vertices.
Assuming that $s-1/2$ is not an integer.
Then, a $J$-tuple $(g^1,\ldots,g^J)$ in $\mathbb{H}^{s+1(1)}(\partial\Omega)$ is the image of a function in $\mathrm{H}^{s+1}(\Omega)$ by the trace mapping $\gamma^{(1)}$ if and only if \cite[Theorem 6.1]{key46} \\
(i)\ For all $\ell,1\leq \ell\leq L$, suppose $f_{j_-(\ell)}$ and $f_{j_+(\ell)}$ are two different faces sharing $e_\ell$, 
it holds that 
\begin{align}
\label{trace-condition1}
g^{j_-(\ell)}(\bm{x})=g^{j_+(\ell)}(\bm{x}) \quad\mathrm{for}\ \mathrm{a.e.}\ \bm{x}\in e_\ell,
\end{align}
(ii)\ For all integers $n, 0\leq n<s-1/2$ and $m, 0\leq m<s-1/2-n$, for all vertex $\bm{a}$ 
and $\bm{\sigma}$ in $\Sigma_{\bm{a}}$, for any differential operators $\mathcal{L}_j=c_j\partial_{\bm{\tau}_j}^n, j\in J(\bm{\sigma})$ 
such that $\sum_{j\in J(\bm{\sigma})}\mathcal{L}_j=0$, the following conditions hold: 
\begin{align}
\label{trace-condition2}
\partial_{\bm{\sigma}}^m\sum_{j\in J(\bm{\sigma})}(\mathcal{L}_jg^j)(\bm{a})=0.
\end{align}
The space $\widetilde{\mathbb{H}}^{s+1(1)}(\partial\Omega)$ of all functions in $\mathbb{H}^{s+1(1)}(\partial\Omega)$ satisfying all compatibility conditions \eqref{trace-condition1} and \eqref{trace-condition2} is closed in $\mathbb{H}^{s+1(1)}(\partial\Omega)$.
Moreover, the operator $\gamma^{(1)}$ is continuous from $\mathrm{H}^{s+1}(\Omega)$ onto $\widetilde{\mathbb{H}}^{s+1(1)}(\partial\Omega)$ and admits a continuous inverse from $\widetilde{\mathbb{H}}^{s+1(1)}(\partial\Omega)$ into $\mathrm{H}^{s+1}(\Omega)$ \cite[Corollary 6.2]{key46}.

Combining Theorem \ref{thm:Lipschitz-contractible-order_s+1-decomposion} with the trace theory on polyhedra, we obtain the following regular decomposition theorem of the space $\mathbf{H}_0^s(\mathbf{curl};\Omega)$, which guides the convergence analysis for the numerical method in this paper.

\begin{thm}
\label{thm:Lipschitz-contractible-order_s+1-decomposion-zeroboundary}
Let $\Omega$ be a bounded contractible Lipschitz polyhedron in $\mathbb{R}^3$ and $s>1/2$ such that $s-1/2$ is not an integer. 
Then there exist continuous linear operators $\mathsf{P}:\mathbf{H}_0^s(\mathbf{curl};\Omega)\mapsto \mathbf{H}^{s+1}(\Omega)
\cap \mathbf{H}_0(\mathbf{curl};\Omega)$ and $\mathsf{Q}:\mathbf{H}_0^s(\mathbf{curl};\Omega)\mapsto \mathrm{H}^{s+1}(\Omega)
\cap \mathrm{H}_0^1(\Omega)$ such that 
\begin{align}
\mathsf{P}+\mathbf{grad}\circ\mathsf{Q}=\mathsf{Id},
\end{align}
where $\mathsf{Id}$ denotes the identity mapping 
on $\mathbf{H}_0^s(\mathbf{curl};\Omega)$.
Furthermore, for any $\boldsymbol{\xi}\in \mathbf{H}_0^s(\mathbf{curl};\Omega)$, 
the following inequalities hold 
\begin{align}\label{eq:Lipschitz-contractible-order_s+1-decomposion-estimate-zeroboundary}
\|\mathsf{P}\boldsymbol{\xi}\|_{\mathbf{H}^{s+1}(\Omega)}
\leq C_1\|\boldsymbol{\xi}\|_{\mathbf{H}^s(\mathbf{curl};\Omega)}, \quad
\|\mathsf{Q}\boldsymbol{\xi}\|_{\mathrm{H}^{s+1}(\Omega)}
\leq C_2\|\boldsymbol{\xi}\|_{\mathbf{H}^s(\mathbf{curl};\Omega)},
\end{align}
where $C_1$ and $C_2$ are constants depending on the polyhedron $\Omega$.
\end{thm}

\begin{proof}
For any $\boldsymbol{\xi}\in \mathbf{H}_0^s(\mathbf{curl};\Omega)$, there exists a decomposition $\boldsymbol{\xi}=\mathbf{u}+\nabla p$ by Theorem \ref{thm:Lipschitz-contractible-order_s+1-decomposion}, where $\mathbf{u}=\mathsf{M}\boldsymbol{\xi}\in\mathbf{H}^{s+1}(\Omega)$ and $p=\mathsf{N}\boldsymbol{\xi}\in\mathrm{H}^{s+1}(\Omega)$.  
The image of $p$ under the trace operator $\gamma^{(1)}$ is denoted as the $J$-tuple $(p^1,\ldots,p^J)$, where $p^j$ is 
the trace of $p$ on the face $f_j$, $1\leq j\leq J$. 
Since $\mathrm{H}^{s+1}(\Omega)$ is continuously embedded into $\mathrm{C}(\overline{\Omega})$ when $s>1/2$ \cite[Theorem 3.5]{key1}, it follows that $p|_{f_j}=p^j$, $1\leq j\leq J$.
Similar conclusions also hold for $\mathbf{u}$, which implies that $\mathbf{u}|_{f_j}\in \mathbf{H}^{s+1/2}(f_j)$.

On each face $f_j,1\leq j\leq J$, $\boldsymbol{\xi}\in \mathbf{H}_0(\mathbf{curl};\Omega)$ yields that $(\mathbf{u}+\nabla p)|_{f_j}\times \mathbf{n}_j=\mathbf{0}$, and consequently, the surface gradient of $p^j$ on $f_j$ is
\begin{align}
\label{eq:surface gradient-pj}
\nabla_{f_j}(p^j)=\nabla_{f_j}(p|_{f_j})=(\mathbf{n}_j\times(\nabla p)|_{f_j})\times\mathbf{n}_j
=-(\mathbf{n}_j\times\mathbf{u}|_{f_j})\times\mathbf{n}_j\in \mathbf{H}^{s+1/2}(f_j),
\end{align}
which implies that $p^j\in \mathrm{H}^{s+3/2}(f_j)$ \cite[Section 3.4]{key1}.
Since $(p^1,\ldots,p^J)=\gamma^{(1)}p$, condition \eqref{trace-condition1} naturally holds.
Furthermore, for all integers $1\leq n<s+1/2$ and $0\leq m<s+1/2-n$, for each vertex $\bm{a}$ and each vector $\bm{\sigma}\in\Sigma_{\bm{a}}$, and for any differential operators $\mathcal{L}_j=c_j\partial_{\bm{\tau}_j}^n, j\in J(\bm{\sigma})$ satisfying 
$\sum_{j\in J(\bm{\sigma})}\mathcal{L}_j=0$, it follows from \eqref{eq:surface gradient-pj} that
\begin{align}
&\partial_{\bm{\sigma}}^m\sum_{j\in J(\bm{\sigma})}(\mathcal{L}_jp^j)(\bm{a})
=\partial_{\bm{\sigma}}^m\sum_{j\in J(\bm{\sigma})}(c_j\partial_{\bm{\tau}_j}^np^j)(\bm{a})
=\partial_{\bm{\sigma}}^m\sum_{j\in J(\bm{\sigma})}(c_j\partial_{\bm{\tau}_j}^{n-1}(\nabla_{f_j}(p^j)\cdot\bm{\tau}_j))(\bm{a}) \nonumber\\
=&\partial_{\bm{\sigma}}^m\sum_{j\in J(\bm{\sigma})}(c_j\partial_{\bm{\tau}_j}^{n-1}(\nabla p\cdot\bm{\tau}_j))(\bm{a})
=\partial_{\bm{\sigma}}^m\sum_{j\in J(\bm{\sigma})}(c_j\partial_{\bm{\tau}_j}^{n}p)(\bm{a})
=\partial_{\bm{\sigma}}^m\sum_{j\in J(\bm{\sigma})}(\mathcal{L}_jp)(\bm{a})=0.
\end{align}
A similar equality holds for $0\leq n\leq s+1/2$ and $1\leq m < s+1/2-n$, 
since the vector $\bm{\sigma}$ is also tangent to $f_j$. 
Finally, \eqref{trace-condition2} is trivial for  $n = m = 0$.
Therefore, the $J$-tuple $(p^1,\ldots,p^J)\in\widetilde{\mathbb{H}}^{s+2(1)}(\partial\Omega)$ is the image of some function $\phi\in\mathrm{H}^{s+2}(\Omega)$ by the trace  operator $\gamma^{(1)}$.

Let us define the operators $\mathsf{P}$ and $\mathsf{Q}$ as 
$\mathsf{P}\boldsymbol{\xi}=\mathbf{u}+\nabla\phi, \mathsf{Q}\boldsymbol{\xi}=\nabla(p-\phi)$, where $\mathbf{u}+\nabla\phi\in\mathbf{H}^{s+1}(\Omega)$ and $p-\phi\in\mathrm{H}^{s+1}(\Omega)$.
Evidently, $\mathsf{P}+\mathbf{grad}\circ\mathsf{Q}=\mathsf{Id}$.
Besides, on each face $f_j,1\leq j\leq J$, $(p-\phi)|_{f_j}=p^j-p^j=0$ yields that $p-\phi\in\mathrm{H}_0^1(\Omega)$.
Consequently, $\mathbf{u}+\nabla\phi=\boldsymbol{\xi}-\nabla(p-\phi)\in\mathbf{H}_0(\mathbf{curl};\Omega)$.
Furthermore, from \eqref{eq:surface gradient-pj} and the continuity of the trace operator $\gamma^{(1)}$ and its inverse, it follows that $\phi$ is bounded 
in the following sense 
\begin{align}
\label{eq:bounded-phi}
&\|\phi\|_{\mathrm{H}^{s+2}(\Omega)}
\lesssim \|\gamma^{(1)}\phi\|_{\mathbb{H}^{s+2(1)}(\partial\Omega)}
\lesssim \sum_{j=1}^{J}\|p^j\|_{\mathrm{H}^{s+3/2}(f_j)}
\lesssim \sum_{j=1}^{J}\left(\|p^j\|_{\mathrm{L}^{2}(f_j)}+\|\nabla_{f_j}p^j\|_{\mathbf{H}^{s+1/2}(f_j)} \right) \nonumber\\
&\leq \sum_{j=1}^{J}\left(\|p\|_{\mathrm{L}^{2}(f_j)}+\|\mathbf{u}\|_{\mathbf{H}^{s+1/2}(f_j)} \right)
\lesssim \|p\|_{\mathrm{L}^{2}(\Omega)}+\|\mathbf{u}\|_{\mathbf{H}^{s+1}(\Omega)}
\lesssim \|\boldsymbol{\xi}\|_{\mathbf{H}^s(\mathbf{curl};\Omega)},
\end{align}
where the notation $\lesssim$ indicates that the left-hand side is bounded above by a positive constant multiple of the right-hand side, with the constant being independent of the functions.
Combining \eqref{eq:Lipschitz-contractible-order_s+1-decomposion-estimate} 
and \eqref{eq:bounded-phi} immediately yields 
the estimate \eqref{eq:Lipschitz-contractible-order_s+1-decomposion-estimate-zeroboundary}. 
\end{proof}

According to Theorem \ref{thm:Lipschitz-contractible-order_s+1-decomposion-zeroboundary}, one can naturally decompose $\mathbf{H}_0^s(\mathbf{curl};\Omega)$ into
\begin{align}
\label{eq:H_0curl-order-s-decomposion}
\mathbf{H}_0^s(\mathbf{curl};\Omega)=\mathbf{H}^{s+1}(\Omega)\cap \mathbf{H}_0(\mathbf{curl};\Omega)
+\nabla (\mathrm{H}^{s+1}(\Omega)\cap \mathrm{H}_0^1(\Omega))
\end{align}
when $s-1/2>0$ is not an integer.
This is a new regular decomposition proposed in this paper, 
and it should be noted that the boundary condition of the vector potential function differs from those in decomposition \eqref{order-1-decomposion}.
The significance of this distinction lies in the fact that the regular decomposition \eqref{eq:H_0curl-order-s-decomposion} can be applied to discrete schemes with Lagrange type elements of arbitrary order, which is capable of achieving corresponding high order convergence rates.
If $s-1/2$ is a positive integer, to avoid more complex discussions, we consider a sufficiently small positive number $\delta$ and regard the functions in $\mathbf{H}_0^s(\mathbf{curl};\Omega)$ as being decomposed within $\mathbf{H}_0^{s-\delta}(\mathbf{curl};\Omega)$.
Due to the arbitrariness of $\delta$, the subsequent discussion will directly employ \eqref{eq:H_0curl-order-s-decomposion} to represent the regular decomposition of $\mathbf{H}_0^s(\mathbf{curl};\Omega)$, which will not affect the convergence order of the interpolation error.
Moreover, from the decomposition \eqref{order-1-decomposion}, it can be seen that \eqref{eq:H_0curl-order-s-decomposion} also holds when $s=0$.

At the end of this section, we recall the spectral approximation properties
of compact operators as presented in \cite{key19}. 
Let $T:H\to H$ be a self-adjoint compact operator on a Hilbert space $H$. 
Denote by $\sigma(T)$ and $\rho(T)$ the spectrum and resolvent set of $T$, respectively.
According to the Riesz–Schauder theory for compact operators, $\sigma(T)$ 
is a countable set and the nonzero elements of $\sigma(T)$ are eigenvalues of $T$ 
with finite multiplicity. Furthermore, any sequence of distinct 
elements in $\sigma(T)$ has $0$ as its only accumulation point.

Consider a sequence of self-adjoint compact operators $T_n:H\to H$, $n\geq 1$. Assume $\{T_n\}_{n=1}^\infty$ 
converges pointwise to $T$, which implies
\begin{align}
\lim\limits_{n\to \infty}\|T_nf-Tf\|_H=0,\quad\forall f\in H.
\end{align}
Furthermore, assume $\{T_n\}_{n=1}^\infty$ is collectively compact in the following sense.
\begin{Def}\label{def:collectively-compact}
A sequence of operators $\{T_n\}_{n=1}^\infty$ is called collectively compact if for any bounded subset $B\subset H$, the set
\begin{align}
W=\bigcup_{n=1}^\infty T_n(B)=\Big\{T_nf:\forall f\in B, ~\forall n\geq1\Big\}
\end{align}
is relatively compact in $H$.
\end{Def}

Let $0\neq\mu \in\sigma(T)$ be a nonzero eigenvalue of $T$.
Since $T$ is self-adjoint, the algebraic multiplicity of $\mu $ equals its geometric multiplicity, both being $m=\dim\ker(\mu -T)$.
Let $\Gamma$ be a circle in the complex plane centered at $\mu $ with sufficiently small 
radius which lies in $\rho(T)$ and encloses no other points of $\sigma(T)$.
The spectral projection operator associated with the operator $T$ and the eigenvalue $\mu $ is defined as
\begin{align}
E=E(\mu )=\frac{1}{2\pi\mathrm{i}}\int_\Gamma(z-T)^{-1}dz.
\end{align}
The operator $E$ is the projection onto the eigenspace of $T$ corresponding to $\mu $, that is, $\mathrm{Im}(E)=\ker(\mu -T)$.
For $n$ sufficiently large, $\Gamma\subset\rho(T_n)$ and the open disk bounded by $\Gamma$ 
contains exactly $m$ eigenvalues $\mu _{n, 1}, \ldots, \mu _{n, m}$ of $T_n$.
Based on \cite[Section 3]{key19}, the following two spectral approximation results can be derived.

\begin{lem}
\label{selfadjoint-operator-eigenvalue-order}
Assume $H$ is a Hilbert space, $T$ and $T_n$, $n\geq 1$ are 
self-adjoint compact operators from $H$ to $H$.
Furthermore, assume $\{T_n\}_{n=1}^\infty$ is collectively compact and
converges pointwise to $T$.
Let $\phi_1, \cdots, \phi_m$ be an orthonormal basis of $\mathrm{Im}(E(\mu ))$.
Then there exists a constant $C$ such that the following error estimate holds for any $\mu _{n, j}$, $1\leq j\leq m$, 
\begin{eqnarray}
|\mu -\mu _{n, j}|\leq
C\left\{ \sum\limits_{i, j=1}^m \left|\left((T-T_n)\phi_i, \phi_j\right)_H\right| 
+ \left\|(T-T_n)|_{\mathrm{Im}(E)}\right\|^2 \right\},
\end{eqnarray}
where $\mathrm{Im}(E)=\ker(\mu -T)$ and $(T-T_n)|_{\mathrm{Im}(E)}$ denotes the restrictions of the operator $T-T_n$ to the closed subspace $\mathrm{Im}(E)$.
\end{lem}

\begin{lem}
\label{operator-eigenvector-order}
The assumptions are the same as in Lemma \ref{selfadjoint-operator-eigenvalue-order}.
Let $\mu _n$ be an eigenvalue of $T_n$ such that $\lim\limits_{n\to\infty}\mu _n=\mu $. Suppose for each n that $\xi_n$ is a unit function satisfying $(\mu _n-T_n) \xi_n=0$.
Then there exists $\xi\in\mathrm{Im}(E(\mu ))$ and a constant $C$ such that
\begin{align}
\|\xi-\xi_n\|_H\leq C\|(T-T_n)|_{\mathrm{Im}(E)}\|, \quad\forall n\geq 1.
\end{align}
\end{lem} 

\section{Discretization for Maxwell eigenvalue problem}\label{Section_Problem}
In this section, based on the regular decomposition, we deduce the continuous 
mixed variational form for the Maxwell eigenvalue problem \eqref{maxwell-eigenvalue-problem}.  
Then the discretization by using continuous finite elements is designed for 
the mixed form of the Maxwell eigenvalue problem. 


It is well known that the weak form of the primal variable for 
the Maxwell eigenvalue problem can be defined as follows.  
\begin{problem}
\label{primal-problem}
Find $\lambda\in \mathbb{R}$ and $\boldsymbol{\xi}\in \mathcal{H}_0:=\mathbf{H}_0(\mathbf{curl};\Omega)\cap \mathbf{H}(\mathrm{div}^0;\Omega)$ such that $\boldsymbol{\xi}\neq \mathbf{0}$ and 
\begin{align}\label{Direct-variational-form}
(\nabla\times\boldsymbol{\xi}, \nabla\times\boldsymbol{\eta})=\lambda(\boldsymbol{\xi}, \boldsymbol{\eta}), 
\quad\forall\boldsymbol{\eta}\in \mathcal{H}_0.
\end{align}
\end{problem}
In (\ref{Direct-variational-form}), the Hilbert space $\mathcal{H}_0$ 
is the zero-divergence subspace of $X_N(\Omega)$, satisfying the spatial 
regularity results described in Lemma \ref{X_N-X_T-embedded-theorem}, 
and its norm and inner product can be chosen as the same as those 
for $\mathbf{H}(\mathbf{curl};\Omega)$.
Furthermore, due to the contractibility of the polyhedron $\Omega$, 
there exists a constant $\alpha>0$ depending on $\Omega$ such that the following Friedrichs 
inequality \cite[P. 627]{key2} holds on the space $\mathcal{H}_0$ 
\begin{align}
\label{eq:H_0-Friedrichs-inequality}
\|\nabla\times \boldsymbol{\eta}\|_{\mathbf{L}^{2}(\Omega)} \geq \alpha \|\boldsymbol{\eta}\|_{\mathbf{L}^{2}(\Omega)}, 
\quad\forall \boldsymbol{\eta}\in \mathcal{H}_0.
\end{align}
Inequality \eqref{eq:H_0-Friedrichs-inequality} implies that all eigenvalues of Problem \ref{primal-problem} are positive.
Moreover, the following Theorem \ref{maxwell-eigenvector-discrete-approximate} establishes the regularity of Maxwell eigenfunctions in $\mathcal{H}_0$.

\begin{thm}
\label{maxwell-eigenvector-discrete-approximate}
Let $\Omega$ be a bounded Lipschitz polyhedron, and let $\boldsymbol{\xi}\in \mathcal{H}_0$ be an eigenfunction of the Maxwell eigenvalue problem \eqref{maxwell-eigenvalue-problem} corresponding to a nonzero eigenvalue $\lambda\neq0$.
Then there is a constant $s>1/2$, which depends only on the geometric properties of $\Omega$, such that $\boldsymbol{\xi}\in \mathbf{H}_0^s(\mathbf{curl};\Omega)$.
\end{thm}

\begin{proof}
For any given eigenfunction $\boldsymbol{\xi}\in \mathcal{H}_0=\mathbf{H}_0(\mathbf{curl};\Omega)\cap \mathbf{H}(\mathrm{div}^0;\Omega)\subset X_N(\Omega)$, $\nabla\times\boldsymbol{\xi}$ satisfies $\nabla\times(\nabla\times\boldsymbol{\xi})=\lambda\boldsymbol{\xi}$, which implies $\nabla\times\boldsymbol{\xi}\in \mathbf{H}(\mathbf{curl};\Omega)\cap \mathbf{H}_0(\mathrm{div}^0;\Omega)\subset X_T(\Omega)$ from the de Rham complex \eqref{eq:de_rham_sequence_zero_bc}.
By Lemma \ref{X_N-X_T-embedded-theorem},
there exists $s>1/2$ such that both spaces $X_N(\Omega)$ and $X_T(\Omega)$ are continuously embedded in $\mathbf{H}^s(\Omega)$. Consequently, $\boldsymbol{\xi}\in \mathbf{H}^s(\Omega)$ and $\nabla\times\boldsymbol{\xi}\in \mathbf{H}^s(\Omega)$, hence $\boldsymbol{\xi}\in \mathbf{H}_0^s(\mathbf{curl};\Omega)$.
\end{proof}

From a computational point of view, 
it is challenging to enforce the divergence free constraint.
Therefore, by utilizing the regular decomposition \eqref{eq:H_0curl-order-s-decomposion} for the case $s=0$ 
\begin{align}
\label{eq:H_0curl-order-0-decomposion}
\mathbf{H}_0(\mathbf{curl};\Omega)=\mathbf{H}^1(\Omega)
\cap \mathbf{H}_0(\mathbf{curl};\Omega)+\nabla (\mathrm{H}_0^1(\Omega)),
\end{align}
we can derive the following mixed form for 
\eqref{Direct-variational-form}.
\begin{problem}\label{mixed-problem}
Find $\lambda\in \mathbb{R}$ and $(\mathbf{u},p)\in\mathbf{H}^1(\Omega)\cap \mathbf{H}_0(\mathbf{curl};\Omega)\times \mathrm{H}_0^1(\Omega)$ such that 
$\mathbf{u}+\nabla p\neq\mathbf{0}$ and 
\begin{eqnarray}\label{mix-variational-form}
\left\{
\begin{array}{rll}
(\nabla\times \mathbf{u}, \nabla\times \mathbf{v})&=\lambda(\mathbf{u}+\nabla p, \mathbf{v}), 
&\quad\forall \mathbf{v}\in \mathbf{H}^1(\Omega)\cap \mathbf{H}_0(\mathbf{curl};\Omega), \\
(\mathbf{u}+\nabla p, \nabla q)&=0, &\quad\forall q\in \mathrm{H}_0^1(\Omega).
\end{array}
\right.
\end{eqnarray}
\end{problem}
An important issue is to ensure that the spectra of Problems \ref{primal-problem} 
and \ref{mixed-problem} are identical. Actually, they are equivalent in the following sense.
\begin{thm}\label{thm:problem1-problem2-equivalence}
If $(\lambda,\boldsymbol{\xi})$ is an eigenpair of Problem \ref{primal-problem} with $\lambda>0$, 
then there exist $(\mathbf{u},p)\in\mathbf{H}^1(\Omega)\cap \mathbf{H}_0(\mathbf{curl};\Omega)\times \mathrm{H}_0^1(\Omega)$ such that $\boldsymbol{\xi}=\mathbf{u}+\nabla p$, and $(\lambda,\mathbf{u},p)$ 
is an eigenpair  of Problem \ref{mixed-problem}. 
Conversely, if $(\lambda,\mathbf{u},p)$ is an eigenpair of Problem \ref{mixed-problem}, 
then $(\lambda,\mathbf{u}+\nabla p)$ is also an eigenpair of Problem \ref{primal-problem}.
\end{thm}
\begin{proof}
First, we assume $(\lambda, \boldsymbol{\xi})$ is an eigenpair of Problem \ref{primal-problem}. 
From \eqref{eq:H_0curl-order-s-decomposion}, 
there exists $(\mathbf{u},p)\in\mathbf{H}^1(\Omega)\cap \mathbf{H}_0(\mathbf{curl};\Omega)\times \mathrm{H}_0^1(\Omega)$
such that $\boldsymbol\xi=\mathbf u+\nabla p\in \mathbf{H}(\mathrm{div}^0;\Omega)$. 
Besides, for any $\mathbf{v}\in \mathbf{H}^1(\Omega)\cap \mathbf{H}_0(\mathbf{curl};\Omega)$, 
there is a unique function $g\in \mathrm{H}_0^1(\Omega)$ such that $-\Delta g=\nabla\cdot\mathbf{v}$.
Then $\mathbf{v}+\nabla g\in \mathcal{H}_0 =  \mathbf{H}_0(\mathbf{curl};\Omega)\cap \mathbf{H}(\mathrm{div}^0;\Omega)$.
Taking $\boldsymbol{\eta} = \mathbf{v} + \nabla g$ in \eqref{Direct-variational-form}, 
and combining with $(\mathbf{u} + \nabla p, \nabla g) = 0$, 
the following equality holds 
\begin{align}\label{Equality_2}
(\nabla\times \mathbf{u}, \nabla\times \mathbf{v})
=(\nabla\times \boldsymbol\xi, \nabla\times \boldsymbol\eta)
=\lambda(\mathbf{u}+\nabla p, \mathbf{v}), 
\end{align}
which implies 
that $(\lambda,\mathbf{u},p)$ is an eigenpair for Problem \ref{mixed-problem}.

Conversely, for any $\boldsymbol{\eta} \in \mathcal{H}_0$, we may express $\boldsymbol{\eta} = \mathbf{v} + \nabla q$ with  $(\mathbf{v}, q) \in \mathbf{H}^1(\Omega)\cap \mathbf{H}_0(\mathbf{curl};\Omega) \times \mathrm{H}_0^1(\Omega)$.
From the second equation in \eqref{mix-variational-form}, it follows that $\mathbf{u} + \nabla p \in \mathcal{H}_0$. Adding the two equations in \eqref{mix-variational-form} yields
\begin{align}
(\nabla \times (\mathbf{u} + \nabla p), \nabla \times \boldsymbol{\eta}) = (\nabla \times \mathbf{u}, \nabla \times \mathbf{v}) = \lambda (\mathbf{u} + \nabla p, \boldsymbol{\eta}),\ \ \forall \boldsymbol{\eta}\in \mathcal{H}_0.
\end{align}
This shows that $(\lambda, \mathbf{u} + \nabla p)$ is an eigenpair of Problem \ref{primal-problem}.
\end{proof}



We now proceed to design the discretization 
for Problem \ref{mixed-problem} by using continuous finite elements. 
For this aim, a regular family meshes $\{\mathcal{T}_h\}$ is built for 
the closure of $\Omega$. In this paper, we are concerned with the tetrahedral meshes. 
As usual, $h := \max_{K\in \mathcal{T}_h}h_K$, where $h_K$ denotes the 
diameter for any tetrahedral $K\in \mathcal{T}_h$. 
We denote by $\mathbb{P}_{r}(K)$ the space of polynomials of degree at most $r$ on $K$ 
and by $\tilde{\mathbb{P}}_{r}(K)$ the subspace of homogeneous polynomials of degree $r$.

In the following analysis for the discretizations and their properties, 
the finite element spaces involved in the numerical schemes include
the Lagrange finite element space of order $r$ 
\begin{align}
\mathscr{L}_{h, 0}^r:=\Big\{q_h\in C(\bar{\Omega}):q_h|_K\in\mathbb{P}_{r}(K), ~\forall K \in \mathcal{T}_h\mathrm{~and~}q_h=0\mathrm{~on~}\partial\Omega\Big\}\subset \mathrm{H}_0^1(\Omega),
\end{align}
and the vector Lagrange finite element space of order $r$ 
\begin{align}
\mathscr{V}_{h, \tau0}^r:=\Big\{\mathbf{v}_h\in [C(\bar{\Omega})]^3\cap \mathbf{H}_0(\mathbf{curl};\Omega):\mathbf{v}_h|_K\in[\mathbb{P}_{r}(K)]^3, 
~\forall K \in \mathcal{T}_h\Big\}.
\end{align}
In the convergence analysis, the first kind N\'{e}d\'{e}lec space of order $r$ \cite{key1} 
\begin{align}
\mathscr{N}_{h, 0}^r:=\Big\{\mathbf{v}_h\in \mathbf{H}_0(\mathbf{curl};\Omega):\mathbf{v}_h|_K\in[\mathbb{P}_{r-1}(K)]^3 \oplus \mathbf{x}\times [\tilde{\mathbb{P}}_{r-1}(K)]^3, ~\forall K \in \mathcal{T}_h\Big\},
\end{align}
is also required.
Furthermore, on a bounded contractible Lipschitz domain, the following so-called discrete exact sequence property has been proved in \cite{key2}.

\begin{lem}(\cite[P.116]{key2})\label{discrete-exactline}
The gradients of functions in $\mathscr{L}_{h, 0}^r$ belong to the space $\mathscr{N}_{h, 0}^r$. 
Furthermore, when $\Omega$ is a bounded contractible Lipschitz domain in $\mathbb{R}^3$, 
the image space $\nabla \mathscr{L}_{h, 0}^r$ under the gradient operator is exactly the null 
space of the curl operator in $\mathscr{N}_{h, 0}^r$.
\end{lem}

These finite element spaces have different types of interpolation operators.
Considering
$\mathring{\mathrm{H}}^s(\Omega):=\{q\in\mathrm{H}^s(\Omega):q=0\mathrm{~on~}\partial\Omega\}$, 
the following Scott-Zhang interpolant \cite{key13}
\begin{align}
\Pi_h^r: \mathring{\mathrm{H}}^s(\Omega)\mapsto \mathscr{L}_{h, 0}^r.
\end{align}
is adopted for the continuous finite element space $\mathscr{L}_{h, 0}^r$. 
Correspondingly, the vector Scott-Zhang interpolant
\begin{align}
\boldsymbol\Pi_h^r:\mathbf{H}^s(\Omega)\cap \mathbf{H}_0(\mathbf{curl};\Omega)
\mapsto \mathscr{V}_{h, \tau0}^r,
\end{align}
can be used for the vector finite element space $\mathscr{V}_{h, \tau0}^r$. 
The operators $\Pi_h^r$ and $\boldsymbol\Pi_h^r$ are well-defined 
provided $s>1/2$.
In such a case, for any $0\leq m\leq s\leq r+1$, the following interpolation error estimates hold
\cite[Theorem 4.1]{key13}
\begin{align}
\sum_{K \in \mathcal{T}_h}h_K^{2(m-s)}\|q - \Pi_h^r q\|_{\mathrm{H}^m(K)}^2 
&\leq C \|q\|_{\mathrm{H}^s(\Omega)}^2, \label{eq:scalar-Scott-Zhang-estimate} \\
\sum_{K \in \mathcal{T}_h}h_K^{2(m-s)}\|\mathbf{v} - \boldsymbol\Pi_h^r \mathbf{v}\|_{\mathbf{H}^m(K)}^2 
&\leq C \|\mathbf{v}\|_{\mathbf{H}^s(\Omega)}^2. \label{eq:vector-Scott-Zhang-estimate}
\end{align}

In the analysis presented in Section \ref{Section_Operator}, a type of  
modified interpolation operators established in \cite{key44} will be employed. 
This type of interpolations are smoothed projections constructed by combining the canonical interpolation defined from the degrees of freedom with a smoothing operator.
The advantage of these interpolation operators is that they are not only bounded, 
but also satisfy the commutative diagram property. 
For more information, please refer to \cite{key44}.
\begin{lem}(\cite[Corollary 6.3]{key44})\label{lem:Commutative-Diagram}
Let $\Omega$ be a bounded polyhedron, then there exist linear smoothed projections
\begin{align}
\Pi_{\mathbf{grad}, h}^{r}: \mathrm{L}^2(\Omega)\mapsto \mathscr{L}_{h, 0}^r, \ \ \ 
\boldsymbol\Pi_{\mathbf{curl}, h}^{r}:
\mathbf{L}^2(\Omega)\mapsto \mathscr{N}_{h, 0}^r,
\end{align}
such that $\Pi_{\mathbf{grad}, h}^{r}$ is a uniformly bounded operator in $\mathcal{L}(\mathrm{L}^2(\Omega))$ and $\mathcal{L}(\mathrm{H}_0^1(\Omega))$, 
and $\boldsymbol\Pi_{\mathbf{curl}, h}^{r}$is a uniformly bounded operator in $\mathcal{L}(\mathbf{L}^2(\Omega))$ and $\mathcal{L}(\mathbf{H}_0(\mathbf{curl};\Omega))$.
These two projections satisfy the following commutative diagram property
\begin{align}
\nabla(\Pi_{\mathbf{grad}, h}^{r}q)=\boldsymbol\Pi_{\mathbf{curl}, h}^{r}(\nabla q), 
\quad\forall q\in\mathrm{H}_0^1(\Omega).
\end{align}
Moreover, for all $q\in \mathrm{L}^2(\Omega)$ and $\mathbf{v}\in \mathbf{L}^2(\Omega)$, 
it holds that $\Pi_{\mathbf{grad}, h}^{r}q\to q$ in $\mathrm{L}^2(\Omega)$ 
and $\boldsymbol\Pi_{\mathbf{curl}, h}^{r}\mathbf{v}\to \mathbf{v}$ in $\mathbf{L}^2(\Omega)$ as
$h\to 0$, respectively.
\end{lem}

Let the finite element spaces $\mathbf{X}_{h, \tau0}$ and $Y_{h, 0}$ 
denote subspaces of $\mathbf{H}^1(\Omega)\cap \mathbf{H}_0(\mathbf{curl};\Omega)$ 
and $\mathrm{H}_0^1(\Omega)$, respectively. 
Then the discrete variational formulation for Problem \ref{mixed-problem} is given as follows. 
\begin{problem}\label{discrete-mixed-problem}
Find $\lambda_h\in \mathbb{R}$ and $(\mathbf{u}_h,p_h)\in\mathbf{X}_{h, \tau0}\times Y_{h, 0}$ 
such that $\mathbf{u}_h+\nabla p_h\neq\mathbf{0}$ and 
\begin{eqnarray}\label{mix-discrete-variational-form}
\left\{
\begin{array}{rll}
(\nabla\times \mathbf{u}_h, \nabla\times \mathbf{v}_h)&=\lambda_h(\mathbf{u}_h+\nabla p_h, \mathbf{v}_h), 
&\quad\forall \mathbf{v}_h\in \mathbf{X}_{h, \tau0}, \\
(\mathbf{u}_h+\nabla p_h, \nabla q_h)&=0, &\quad\forall q_h\in Y_{h, 0}.
\end{array}
\right.\end{eqnarray}
\end{problem}
As shown in \eqref{mix-discrete-variational-form}, we are actually concerned with 
the so-called eigenfunction $\mathbf{0}\neq\boldsymbol{\xi}_h=\mathbf{u}_h+\nabla p_h$ 
corresponding to the eigenvalue $\lambda_h$.
Clearly, $\boldsymbol{\xi}_h$ belongs to the following space 
$\mathcal{H}_0^h\subset \mathbf{H}_0(\mathbf{curl};\Omega)$, which is naturally derived
from \eqref{mix-discrete-variational-form} and 
possesses a divergence-free property in the discrete sense relative to $\nabla Y_{h, 0}$
\begin{align}\label{def:discrete-div0-space}
\mathcal{H}_0^h:=\Big\{\mathbf{v}_h+\nabla q_h\in \mathbf{X}_{h, \tau0}+\nabla Y_{h, 0} : (\mathbf{v}_h 
+\nabla q_h, \nabla r_h)=0, ~\forall r_h\in Y_{h, 0}\Big\}.
\end{align}
In particular, in this paper, we choose $\mathbf{X}_{h, \tau0}=\mathscr{V}_{h, \tau0}^r$ and $Y_{h, 0}
=\mathscr{L}_{h, 0}^{r+1}$, where $r$ is any positive integer.
The key design of taking the polynomial degree of the scalar space one order higher than that of the vector space is a crucial feature of our discretization.
It ensures the discrete exact sequence property, which relies on the presence of rich basis functions in $Y_{h, 0}$ and plays an essential role in excluding spurious zero modes in discrete Problem \ref{discrete-mixed-problem}, as referenced in the following Theorem \ref{thm:mix-discrete-variational-form_nonzero-eigenvalue}.
It is also closely related to the derivation of discrete compactness for discrete space sequences, as detailed in Theorem \ref{thm:discrete-compactness-property}.

\begin{thm}\label{thm:mix-discrete-variational-form_nonzero-eigenvalue} 
All eigenvalues of Problem \ref{discrete-mixed-problem} are positive real numbers.
\end{thm}
\begin{proof}
It is evident that Problem \ref{primal-problem} does not admit any negative eigenvalues.
Suppose there exists $\mathbf{u}_h+\nabla p_h\in \mathbf{X}_{h, \tau0}+\nabla Y_{h, 0}$ and $\lambda_h=0$ 
satisfying \eqref{mix-discrete-variational-form}. Taking $\mathbf{v}_h=\mathbf{u}_h$
shows that $\nabla\times \mathbf{u}_h=\mathbf{0}$. Note that the discrete function 
$\mathbf{u}_h\in \mathbf{X}_{h, \tau0}=\mathscr{V}_{h, \tau0}^r\subset \mathscr{N}_{h, 0}^{r+1}$ is actually also within the first kind N\'{e}d\'{e}lec space. Therefore, according to the discrete exact sequence property 
in Lemma \ref{discrete-exactline},
there exists 
$r_h\in \mathscr{L}_{h, 0}^{r+1}=Y_{h, 0}$ such that $\mathbf{u}_h=\nabla r_h$.
Then, taking $q_h=r_h+p_h$ in \eqref{mix-discrete-variational-form} gives $|r_h+p_h|_{\mathrm{H}^1(\Omega)}=0$,
which leads to a contradiction that $\mathbf{u}_h+\nabla p_h=\nabla (r_h+p_h)=\mathbf{0}$.
Thus, Problem \ref{discrete-mixed-problem} has no zero eigenvalue.
\end{proof}

The following result shows that for every discrete function $\boldsymbol{\xi}_h\in \mathbf{X}_{h, \tau0}+\nabla Y_{h, 0}$, 
we can transform it into a function in the space $\mathcal{H}_0^h$ via gradient correction term.

\begin{thm}
\label{cor:X_h,t0-to-H_0^h}
For any $\boldsymbol{\xi}_h\in \mathbf{X}_{h, \tau0}+\nabla Y_{h, 0}$, 
there exists a unique $r_h\in Y_{h, 0}$ such that $\boldsymbol{\xi}_h+\nabla r_h\in \mathcal{H}_0^h$.
Furthermore, $r_h$ satisfies the following inequality
\begin{align}\label{eq:r_h-form-control}
\|\nabla r_h\|_{\mathbf{L}^2(\Omega)}
\leq\|\boldsymbol{\xi}-\boldsymbol{\xi}_h\|_{\mathbf{L}^2(\Omega)}, 
\quad\forall \boldsymbol{\xi}\in \mathbf{H}(\mathrm{div}^0;\Omega).
\end{align}
\end{thm}

\begin{proof}
According to the definition of the space $\mathcal{H}_0^h$ in \eqref{def:discrete-div0-space}, it suffices to show that for 
any $\boldsymbol{\xi}_h\in \mathbf{X}_{h, \tau0}+\nabla Y_{h, 0}$, there exists a unique $r_h\in Y_{h, 0}$ such that
\begin{align}
\label{eq:r_h-existence}
(\nabla r_h, \nabla q_h)=-(\boldsymbol{\xi}_h, \nabla q_h), \quad\forall q_h\in Y_{h, 0}.
\end{align}
Since $Y_{h, 0}$ is a Lagrange finite element space for $\mathrm{H}_0^1(\Omega)$, 
the symmetric bilinear form $(\nabla\cdot, \nabla\cdot)$ is elliptic and bounded on $Y_{h, 0}$, 
the existence of $r_h$ follows from Lax-Milgram theorem.
Combining $\boldsymbol{\xi}\in \mathbf{H}(\mathrm{div}^0;\Omega)$ 
with \eqref{eq:r_h-existence}, it follows that
\begin{align}
\label{eq:r_h-formcontrol}
(\nabla r_h, \nabla q_h)=(\boldsymbol{\xi}-\boldsymbol{\xi}_h, \nabla q_h), \quad\forall q_h\in Y_{h, 0}.
\end{align}
Taking $q_h=r_h$ in \eqref{eq:r_h-formcontrol} immediately yields \eqref{eq:r_h-form-control}.
\end{proof}

\section{Spectral convergence results for operators}\label{Section_Operator}
This section reformulates Problems \ref{mixed-problem} and \ref{discrete-mixed-problem} 
as spectral problems for the corresponding sequence of self-adjoint operators.  
In order to apply the spectral convergence theory of compact 
operators proposed in Section \ref{Section_Pre}, the aim of this section is to prove the 
resulting sequence of discrete operators, 
as the mesh size $h$ tends to $0^+$, possesses collective compactness 
in the sense of Definition \ref{def:collectively-compact} and converges pointwise 
to the continuous compact operator.

\subsection{Operators and their properties}
The construction of self-adjoint compact operators relies on the positive definite Maxwell problem.
First, let us define the operator $\mathcal{R}:\mathbf{L}^2(\Omega)\mapsto \mathbf{L}^2(\Omega)$, 
which maps any function $\mathbf{f}\in \mathbf{L}^2(\Omega)$ to $\mathcal{R}\mathbf{f}\in \mathcal{H}_0$ satisfying the following variational equation
\begin{align}\label{eq:def-operater-R}
(\nabla\times\mathcal{R}\mathbf{f}, \nabla\times\boldsymbol{\eta})+(\mathcal{R}\mathbf{f}, \boldsymbol{\eta})
=(\mathbf{f}, \boldsymbol{\eta}), \quad\forall\boldsymbol{\eta}\in \mathcal{H}_0.
\end{align}
This variational problem is clearly positive definite. According to
Lax-Milgram theorem, the solution $\mathcal{R}\mathbf{f}$ 
exists and is unique. Hence, the operator $\mathcal{R}$ is well-defined on $\mathbf{L}^2(\Omega)$, 
and taking $\boldsymbol{\eta}=\mathcal{R}\mathbf{f}$ yields the following inequality
\begin{align}
\label{eq:Rf-form-control}
\|\mathcal{R}\mathbf{f}\|_{\mathcal{H}_0}=\|\mathcal{R}\mathbf{f}\|_{\mathbf{H}(\mathbf{curl};\Omega)}
\leq\|\mathbf{f}\|_{\mathbf{L}^2(\Omega)}, \quad\forall \mathbf{f}\in \mathbf{L}^2(\Omega).
\end{align}

\begin{thm}
\label{thm:R-semidefinite-selfadjoint-tight}
The operator $\mathcal{R}:\mathbf{L}^2(\Omega)\mapsto \mathbf{L}^2(\Omega)$ 
is a compact self-adjoint positive semi-definite operator.
\end{thm}

\begin{proof}
For any bounded set $B$ in $\mathbf{L}^2(\Omega)$, the inequality \eqref{eq:Rf-form-control} 
implies that $\mathcal{R}(B)$ is a bounded set in $\mathcal{H}_0$. 
Since $\mathcal{H}_0$ is compactly embedded into $\mathbf{L}^2(\Omega)$, 
$\mathcal{R}(B)$ is relatively compact in $\mathbf{L}^2(\Omega)$. 
Thus, $\mathcal{R}$ is a compact operator from $\mathbf{L}^2(\Omega)$ to itself.
For any $\mathbf{f}, \mathbf{g}\in \mathbf{L}^2(\Omega)$, we have following equalities
\begin{align}
(\mathcal{R}\mathbf{f}, \mathbf{g})&=(\mathbf{g}, \mathcal{R}\mathbf{f})=(\nabla\times\mathcal{R}\mathbf{g},
\nabla\times \mathcal{R}\mathbf{f})+(\mathcal{R}\mathbf{g}, \mathcal{R}\mathbf{f}) \nonumber\\
&=(\nabla\times\mathcal{R}\mathbf{f}, \nabla\times \mathcal{R}\mathbf{g})+(\mathcal{R}\mathbf{f}, \mathcal{R}\mathbf{g})
=(\mathbf{f}, \mathcal{R}\mathbf{g}).
\end{align}
Therefore, the operator $\mathcal{R}$ is self-adjoint. Furthermore,
\begin{align}
(\mathcal{R}\mathbf{f}, \mathbf{f})&=(\mathbf{f}, \mathcal{R}\mathbf{f})=(\nabla\times\mathcal{R}\mathbf{f},
\nabla\times \mathcal{R}\mathbf{f})+(\mathcal{R}\mathbf{f}, \mathcal{R}\mathbf{f})\geq 0,\ \ \ 
\forall 0\neq \mathbf f\in \mathbf L^2(\Omega). 
\end{align}
Thus $\mathcal{R}$ is also positive semi-definite.
Observe that it is easy to verify that $\mathbf{f}=\nabla q$ satisfies $\mathcal{R}\mathbf{f}=\mathbf{0}$ for any $q\in \mathrm{H}_0^1(\Omega)$, hence $\mathcal{R}$ is not positive definite.
\end{proof}
The following theorem presents the spectral equivalence between 
the continuous mixed Problem \ref{mixed-problem} and the operator $\mathcal{R}$.
According to Theorem \ref{thm:problem1-problem2-equivalence}, 
it suffices to prove the same conclusion for Problem \ref{primal-problem}.
\begin{thm}\label{thm:operater-R-equivalence-variational}
All eigenvalues of the operator $\mathcal{R}$ lie within the interval $[0, 1)$.
Moreover, $(\lambda, \boldsymbol{\xi})$ being an eigenpair of Problem \ref{primal-problem} 
is equivalent to that $(1/(\lambda+1),\boldsymbol{\xi})$ is an eigenpair 
of the operator $\mathcal{R}$.  
\end{thm}
\begin{proof}
Since the operator $\mathcal{R}:\mathbf{L}^2(\Omega)\mapsto \mathbf{L}^2(\Omega)$ is a compact 
self-adjoint semi-positive definite operator, its eigenvalues are nonnegative real numbers 
by Riesz–Schauder theory.
Let $(\lambda, \boldsymbol{\xi})$ be an eigenpair of Problem \ref{primal-problem}, the following equalities hold according to \eqref{eq:def-operater-R} 
\begin{eqnarray}\label{Equality_3}
&&(\nabla\times \boldsymbol \xi, \nabla \times \boldsymbol \eta)+(\boldsymbol\xi,\boldsymbol\eta)=(\lambda+1)(\boldsymbol\xi,\boldsymbol\eta)\nonumber\\
&&=(\lambda+1)\left((\nabla\times \mathcal R\boldsymbol \xi, \nabla \times \boldsymbol \eta)+(\mathcal R\boldsymbol\xi,\boldsymbol\eta)\right),\ \ \ \forall\boldsymbol\eta\in \mathcal H_0. 
\end{eqnarray}
Then it is easy to deduce that $(\lambda+1)\mathcal R\boldsymbol\xi=\boldsymbol\xi$, 
which means that $(1/(\lambda+1),\boldsymbol\xi)$ is also an eigenpair of the operator $\mathcal R$. 

If  $(1/(\lambda+1),\boldsymbol\xi)$ is an eigenpair of the operator $\mathcal R$, 
from (\ref{Equality_3}), it is also easily seen that  $(\lambda, \boldsymbol{\xi})$ be an eigenpair of Problem \ref{primal-problem} and the proof is complete. 
\end{proof}

Analogous to the definition of the continuous operator $\mathcal{R}$, we can also define 
the discrete operator $\mathcal{R}_h:\mathbf{L}^2(\Omega)\mapsto \mathbf{L}^2(\Omega)$ mapping  
any function $\mathbf{f}\in \mathbf{L}^2(\Omega)$ to $\mathcal{R}_h\mathbf{f}\in \mathcal{H}_0^h$,  
which is divergence-free in the discrete sense and satisfies the following equation
\begin{align}\label{eq:def-operater-R_h}
(\nabla\times\mathcal{R}_h\mathbf{f}, \nabla\times\boldsymbol{\eta}_h)
+(\mathcal{R}_h\mathbf{f}, \boldsymbol{\eta}_h)=(\mathbf{f}, 
\boldsymbol{\eta}_h), \quad\forall\boldsymbol{\eta}_h\in \mathcal{H}_0^h.
\end{align}
Similar to the operator $\mathcal{R}$, the positive definiteness of this variational problem implies that the 
solution $\mathcal{R}_h\mathbf{f}$ exists and is unique. Hence, the operator $\mathcal{R}_h$ is well-defined 
on $\mathbf{L}^2(\Omega)$, and taking $\boldsymbol{\eta}_h=\mathcal{R}_h\mathbf{f}$ 
yields the following inequality:
\begin{align}\label{eq:Rhf-form-control}
\|\mathcal{R}_h\mathbf{f}\|_{\mathbf{H}(\mathbf{curl};\Omega)}\leq\|\mathbf{f}\|_{\mathbf{L}^2(\Omega)}, 
\quad\forall \mathbf{f}\in \mathbf{L}^2(\Omega).
\end{align}
In addition, the operator $\mathcal{R}_h$ has following property.
\begin{thm}
The operator $\mathcal{R}_h:\mathbf{L}^2(\Omega)\mapsto \mathbf{L}^2(\Omega)$ is 
a finite-rank self-adjoint positive semi-definite operator.
\end{thm}


The following theorem presents the spectral equivalence between 
the discrete mixed Problem \ref{discrete-mixed-problem} and the operator $\mathcal{R}_h$.
\begin{thm}\label{thm:operater-Rh-equivalence-variational}
The eigenpair $(\lambda_h,  \mathbf{u}_h, p_h)$ of Problem \ref{discrete-mixed-problem} 
is equivalent to that $(1/(\lambda_h+1),\boldsymbol{\xi}_h)$ being an eigenpair of the operator $\mathcal{R}_h$, 
where $\lambda_h>0$ and $\boldsymbol{\xi}_h=\mathbf{u}_h+\nabla p_h$.
\end{thm}
\begin{proof}
Let $(\lambda_h, \mathbf{u}_h, p_h)$ be an eigenpair of Problem \ref{discrete-mixed-problem}, and let $\mathbf{0}\neq \boldsymbol{\xi}_h=\mathbf{u}_h+\nabla p_h\in \mathcal{H}_0^h$.
Since any $\boldsymbol{\eta}_h\in \mathcal{H}_0^h$ can be decomposed as $\boldsymbol{\eta}_h=\mathbf{v}_h+\nabla q_h$, where $\mathbf{v}_h\in \mathbf{X}_{h, \tau0}$ and 
$q_h\in Y_{h, 0}$, it follows from \eqref{mix-discrete-variational-form} and \eqref{eq:def-operater-R_h} that
\begin{align}
&(\nabla\times \boldsymbol{\xi}_h, \nabla \times \boldsymbol{\eta}_h)+(\boldsymbol{\xi}_h,\boldsymbol{\eta}_h)
=\lambda_h(\boldsymbol\xi_h,\mathbf{v}_h)+(\boldsymbol{\xi}_h,\boldsymbol{\eta}_h)
=(\lambda_h+1)(\boldsymbol\xi_h,\boldsymbol\eta_h) \nonumber\\
=&(\lambda_h+1)\left((\nabla\times \mathcal R_h\boldsymbol \xi_h, \nabla \times \boldsymbol \eta_h)+(\mathcal R_h\boldsymbol\xi_h,\boldsymbol\eta_h)\right),\ \ \ \forall\boldsymbol\eta_h\in \mathcal H_0^h.  
\end{align}
Then it can be inferred that $(\lambda_h+1)\mathcal R_h\boldsymbol\xi_h=\boldsymbol\xi_h$, 
which means that $(1/(\lambda_h+1),\boldsymbol\xi_h)$ is also an eigenpair of the operator $\mathcal R_h$. 

Conversely, if $(1/(\lambda_h+1),\boldsymbol\xi_h)$ is an eigenpair of the operator $\mathcal R_h$, 
then for any $q_h\in Y_{h, 0}$, it follows from $\boldsymbol{\xi}_h=\mathbf{u}_h+\nabla p_h\in \mathcal{H}_0^h$ 
that $(\mathbf{u}_h+\nabla p_h, \nabla q_h)=0$.
Furthermore, for any $\mathbf{v}_h\in \mathbf{X}_{h, \tau0}$, by Theorem \ref{cor:X_h,t0-to-H_0^h}, 
there exists a unique $r_h\in Y_{h, 0}$ 
such that $\mathbf{v}_h+\nabla r_h=\boldsymbol{\eta}_h\in \mathcal{H}_0^h$. 
Combining this with \eqref{eq:def-operater-R_h} gives
\begin{align}
&(\nabla\times \mathbf{u}_h, \nabla\times \mathbf{v}_h)=(\nabla\times \boldsymbol{\xi}_h, \nabla\times \boldsymbol{\eta}_h)
=(\lambda_h+1)(\nabla\times \mathcal{R}_h\boldsymbol{\xi}_h, \nabla\times \boldsymbol{\eta}_h) \nonumber\\
&=(\lambda_h+1)(\boldsymbol{\xi}_h, \boldsymbol{\eta}_h)-(\lambda_h+1)(\mathcal{R}_h\boldsymbol{\xi}_h, \boldsymbol{\eta}_h)
=(\lambda_h+1)(\boldsymbol{\xi}_h, \boldsymbol{\eta}_h)-(\boldsymbol{\xi}_h, \boldsymbol{\eta}_h) \nonumber\\
&=\lambda_h(\boldsymbol{\xi}_h, \boldsymbol{\eta}_h)
=\lambda_h(\mathbf{u}_h+\nabla p_h, \mathbf{v}_h).
\end{align}
Therefore, $(\lambda_h, \mathbf{u}_h, p_h)$ is also an eigenpair of Problem \ref{discrete-mixed-problem}.
\end{proof}

Since the discrete space $\mathcal{H}_0^h$ is not a subspace of $\mathcal{H}_0$, on which the divergence-free condition is imposed, we cannot directly obtain the orthogonality of $\mathcal{R}\mathbf{f}-\mathcal{R}_{h}\mathbf{f}$ 
to $\mathcal{H}_0^h$ under the $\mathbf{H}(\mathbf{curl};\Omega)$ inner product for all $\mathbf{f}\in \mathbf{L}^2(\Omega)$.
However, if $\mathbf{f}\in \mathbf{H}(\mathrm{div}^0;\Omega)$, the following orthogonality theorem holds.
\begin{thm}
\label{thm:Rxi-Rhxi-orthogonal}
For any function $\boldsymbol{\xi}\in \mathbf{H}(\mathrm{div}^0;\Omega)$, $\mathcal{R}\boldsymbol{\xi}-\mathcal{R}_{h}\boldsymbol{\xi}$ 
is orthogonal to the sum space $\mathbf{X}_{h, \tau0}+\nabla Y_{h, 0}$ under the $\mathbf{H}(\mathbf{curl};\Omega)$ inner product, i.e.,
\begin{align}
\label{eq:eigenvector-Rxi-Rhxi-orthogonal}
(\mathcal{R}\boldsymbol{\xi}-\mathcal{R}_{h}\boldsymbol{\xi}, \mathbf{v}_h+\nabla q_h)_{\mathbf{H}(\mathbf{curl};\Omega)}=0, 
\quad\forall \mathbf{v}_h\in \mathbf{X}_{h, \tau0}, \ \forall q_h\in Y_{h, 0}.
\end{align}
\end{thm}

\begin{proof}
Given $\boldsymbol{\xi}\in \mathbf{H}(\mathrm{div}^0;\Omega)$.
For any $\mathbf{v}_h\in \mathbf{X}_{h, \tau0}$,
there is a unique function $q_{h, 1}\in \mathrm{H}_0^1(\Omega)$ such that $-\Delta q_{h, 1}=\nabla\cdot\mathbf{v}_h$.
Then $\mathbf{v}_h+\nabla q_{h, 1}\in \mathbf{H}_0(\mathbf{curl};\Omega)\cap \mathbf{H}(\mathrm{div}^0;\Omega)=\mathcal{H}_0$.
Furthermore, by Corollary \ref{cor:X_h,t0-to-H_0^h}, there exists a unique function $q_{h, 2}\in Y_{h, 0}$ such that $\mathbf{v}_h+\nabla q_{h, 2}\in \mathcal{H}_0^h$.
Taking the test functions as $\mathbf{v}_h+\nabla q_{h, 1}\in \mathcal{H}_0$ and $\mathbf{v}_h+\nabla q_{h, 2}\in \mathcal{H}_0^h$ in
\eqref{eq:def-operater-R} and \eqref{eq:def-operater-R_h}, respectively, we obtain
\begin{align}\label{single-eigenvalue-estimate_Rxi}
(\nabla\times\mathcal{R}\boldsymbol{\xi}, \nabla\times \mathbf{v}_h)+(\mathcal{R}\boldsymbol{\xi}, \mathbf{v}_h+\nabla q_{h, 1})=(\boldsymbol{\xi}, \mathbf{v}_h+\nabla q_{h, 1}), \\
\label{single-eigenvalue-estimate_Rhxi}
(\nabla\times\mathcal{R}_{h}\boldsymbol{\xi}, \nabla\times \mathbf{v}_h)+(\mathcal{R}_{h}\boldsymbol{\xi}, \mathbf{v}_h+\nabla q_{h, 2})=(\boldsymbol{\xi}, \mathbf{v}_h+\nabla q_{h, 2}).
\end{align}
Since $\boldsymbol{\xi}\in \mathbf{H}(\mathrm{div}^0;\Omega)$, $\mathcal{R}\boldsymbol{\xi}\in \mathcal{H}_0$, $\mathcal{R}_{h}\boldsymbol{\xi}\in \mathcal{H}_0^h$, $q_{h, 1}\in \mathrm{H}_0^1(\Omega)$ and $q_{h, 2}\in Y_{h, 0}$, it follows that
\begin{align}
\label{single-eigenvalue-estimate_qh1zero}
(\boldsymbol{\xi}, \nabla q_{h, 1})=(\mathcal{R}\boldsymbol{\xi}, \nabla q_{h, 1})=0, \quad
(\boldsymbol{\xi}, \nabla q_{h, 2})=(\mathcal{R}_{h}\boldsymbol{\xi}, \nabla q_{h, 2})=0.
\end{align}
Substituting \eqref{single-eigenvalue-estimate_qh1zero} into \eqref{single-eigenvalue-estimate_Rxi} and \eqref{single-eigenvalue-estimate_Rhxi} gives
\begin{align}
\label{eq:eigenvector-orthogonality}
(\nabla\times\mathcal{R}\boldsymbol{\xi}, \nabla\times \mathbf{v}_h)+(\mathcal{R}\boldsymbol{\xi}, \mathbf{v}_h)=(\boldsymbol{\xi}, \mathbf{v}_h)=
(\nabla\times\mathcal{R}_{h}\boldsymbol{\xi}, \nabla\times \mathbf{v}_h)+(\mathcal{R}_{h}\boldsymbol{\xi}, \mathbf{v}_h),
\end{align}
which shows that $\mathcal{R}\boldsymbol{\xi}-\mathcal{R}_{h}\boldsymbol{\xi}$ 
is orthogonal to $\mathbf{X}_{h, \tau0}$ under the $\mathbf{H}(\mathbf{curl};\Omega)$ inner product:
\begin{align}
\label{orthogonality_Rxi-Rhxi_vh}
(\nabla\times(\mathcal{R}\boldsymbol{\xi}-\mathcal{R}_{h}\boldsymbol{\xi}), \nabla\times \mathbf{v}_h)
+(\mathcal{R}\boldsymbol{\xi}-\mathcal{R}_{h}\boldsymbol{\xi}, \mathbf{v}_h)=0, \quad\forall \mathbf{v}_h\in \mathbf{X}_{h, \tau0}.
\end{align}
Moreover, since $\mathcal{R}\boldsymbol{\xi}\in \mathcal{H}_0$ 
and $\mathcal{R}_{h}\boldsymbol{\xi}\in \mathcal{H}_0^h$, 
$\mathcal{R}\boldsymbol{\xi}-\mathcal{R}_{h}\boldsymbol{\xi}$ is orthogonal to $\nabla Y_{h, 0}$:
\begin{align}\label{orthogonality_Rxi-Rhxi_gqh}
(\mathcal{R}\boldsymbol{\xi}-\mathcal{R}_{h}\boldsymbol{\xi}, \nabla q_h)_{\mathbf{H}(\mathbf{curl};\Omega)}
=(\mathcal{R}\boldsymbol{\xi}-\mathcal{R}_{h}\boldsymbol{\xi}, \nabla q_h)=0, \quad\forall q_h\in Y_{h, 0}.
\end{align}
Combining \eqref{orthogonality_Rxi-Rhxi_vh} and \eqref{orthogonality_Rxi-Rhxi_gqh} yields the orthogonality of 
$\mathcal{R}\boldsymbol{\xi}-\mathcal{R}_{h}\boldsymbol{\xi}$ to the sum space $\mathbf{X}_{h, \tau0}+\nabla Y_{h, 0}$ 
under the $\mathbf{H}(\mathbf{curl};\Omega)$ inner product \eqref{eq:eigenvector-Rxi-Rhxi-orthogonal}.
\end{proof}

\subsection{Collective compactness}
Let $\wedge=\{h_n\}_{n=1}^\infty$ be a strictly monotonically decreasing 
sequence of positive numbers converging to zero. 
In this subsection, we demonstrate the collective compactness of the operator sequence $\{\mathcal{R}_{h_n}\}_{n=1}^\infty$, which is a necessary condition for Lemmas \ref{selfadjoint-operator-eigenvalue-order} and \ref{operator-eigenvector-order}.
In fact, this is closely related to the discrete compactness property 
of the space sequence $\mathcal{H}_0^h$ into $\mathbf{L}^2(\Omega)$, 
which is analogous to the concept of discrete 
compactness for N\'{e}d\'{e}lec finite element spaces introduced 
by F. Kikuchi \cite{key43}.
\begin{Def}\label{def:discrete-compactness-property}
The space sequence $\{\mathcal{H}_0^h\}_{h\in\wedge}$ is said to possess 
the discrete compactness property with respect to the $\mathbf{H}(\mathbf{curl};\Omega)$ 
norm and embedding into $\mathbf{L}^2(\Omega)$ if, for any sequence of 
functions $\{\boldsymbol{\xi}_{h}\}_{h\in\wedge}$ satisfying the following two conditions: 
\begin{itemize}
\item [{\rm (i)}]
$\boldsymbol{\xi}_{h}\in \mathcal{H}_0^h, \forall h\in\wedge$,
\item[{\rm (ii)}] 
There exists a constant $C>0$ such that $\|\boldsymbol{\xi}_{h}\|_{\mathbf{H}(\mathbf{curl};\Omega)}<C$, $\forall h\in\wedge$, 
\end{itemize}
there exists a subsequence $\{\boldsymbol{\xi}_{h}\}_{h\in\wedge_s}, \wedge_s\subset\wedge$, 
and a function $\boldsymbol{\xi}_0\in \mathbf{L}^2(\Omega)$ such that
\begin{align}
\|\boldsymbol{\xi}_{h}-\boldsymbol{\xi}_0\|_{\mathbf{L}^2(\Omega)}\to0 
\quad\mathrm{as}\quad h\to0~\mathrm{in}\wedge_s.
\end{align}
\end{Def}

\begin{thm}\label{thm:discrete-compactness-property-to-collectively-compact}
If for any strictly monotonically decreasing sequence of positive 
numbers $\wedge=\{h_n\}_{n=1}^\infty$ converging to zero, 
the space sequence $\{\mathcal{H}_0^h\}_{h\in\wedge}$ possesses 
the discrete compactness property described in Definition \ref{def:discrete-compactness-property}, 
then for any such sequence $\wedge=\{h_n\}_{n=1}^\infty$, 
the operator sequence $\{\mathcal{R}_{h}\}_{h\in\wedge}$ is collectively compact.
\end{thm}

\begin{proof}
To prove that the operator sequence $\{\mathcal{R}_{h}\}_{h\in\wedge}$ 
is collectively compact, according to Definition \ref{def:collectively-compact}, 
it suffices to show that for any bounded subset $\mathbf{B}\subset\mathbf{L}^2(\Omega)$, the set
\begin{align}
\mathbf{W}=\bigcup_{h\in\wedge} \mathcal{R}_{h}(\mathbf{B})=\{\mathcal{R}_{h}\mathbf{f}:
\forall \mathbf{f}\in \mathbf{B}, ~\forall h\in\wedge\}
\end{align}
is relatively compact in $\mathbf{L}^2(\Omega)$.
Equivalently, for any fixed sequence $\{\boldsymbol{\xi}_j\}_{j=1}^\infty\subset \mathbf{W}$, 
we need to show that $\{\boldsymbol{\xi}_j\}_{j=1}^\infty$ has a 
convergent subsequence in $\mathbf{L}^2(\Omega)$.
Assume the constant $C$ is an upper bound for $\mathbf{B}$ in $\mathbf{L}^2(\Omega)$.
For each $\boldsymbol{\xi}_j\in \mathbf{W}$, 
there exists some $\mathbf{f}_j\in \mathbf{B}$ and index $h_{n_j}\in\wedge$ 
such that $\boldsymbol{\xi}_j=\mathcal{R}_{h_{n_j}}\mathbf{f}_j$, hence $\boldsymbol{\xi}_j\in \mathcal{H}_0^{h_{n_j}}$.
Now, consider two cases based on the number of distinct elements in the index set $\{h_{n_j}\}_{j=1}^\infty$.

If the index set $\{h_{n_j}\}_{j=1}^\infty$ contains only a finite number of distinct elements,
which implies $\{h_{n_j}\}_{j=1}^\infty$ has finite cardinality.
Then there must exist a positive integer $k$ such that countably many $\boldsymbol{\xi}_j$ satisfy $\boldsymbol{\xi}_j=\mathcal{R}_{h_{n_k}}\mathbf{f}_j\in \mathcal{H}_0^{h_{n_k}}$.
These $\boldsymbol{\xi}_j$ form a countable subset $\{\boldsymbol{\xi}_j\}_{j\in \wedge_k}$ of $\{\boldsymbol{\xi}_j\}_{j=1}^\infty$, where $\wedge_k\subset\wedge$.
Since $\mathcal{R}_{h_{n_k}}$ is a finite-rank operator from $\mathbf{L}^2(\Omega)$ to itself, and $\{\mathbf{f}_j\}_{j\in \wedge_k}\subset \mathbf{B}$ 
is a bounded sequence in $\mathbf{L}^2(\Omega)$ with upper bound $C$, it follows that $\{\boldsymbol{\xi}_j\}_{j\in \wedge_k}$ 
is sequentially compact in $\mathbf{L}^2(\Omega)$, meaning there exists a subsequence of $\{\boldsymbol{\xi}_j\}_{j\in \wedge_k}$ 
that converges in $\mathbf{L}^2(\Omega)$.

Otherwise, assume that the index set $\{h_{n_j}\}_{j=1}^\infty$ contains countably many distinct elements.
In this case, we can extract a subsequence from $\{\boldsymbol{\xi}_j\}_{j=1}^\infty$, still denoted as $\{\boldsymbol{\xi}_j\}_{j=1}^\infty$ 
for convenience, such that $\{n_j\}_{j=1}^\infty$ is strictly monotonically increasing with respect to $j$, 
and thus $\{h_{n_j}\}_{j=1}^\infty$ is strictly monotonically decreasing and converges to zero.
By the assumption of this theorem,
the space sequence $\{\mathcal{H}_0^{h_{n_j}}\}_{j=1}^\infty$ possesses the discrete compactness property described 
in Definition \ref{def:discrete-compactness-property}.
Furthermore,
inequality \eqref{eq:Rhf-form-control} gives
\begin{align}
\label{eq:discrete-compactness-property2}
\|\boldsymbol{\xi}_j\|_{\mathbf{H}(\mathbf{curl};\Omega)}=\|\mathcal{R}_{h_{n_j}}\mathbf{f}_j\|_{\mathbf{H}(\mathbf{curl};\Omega)}
\leq\|\mathbf{f}_j\|_{\mathbf{L}^2(\Omega)}\leq C, \quad\forall 1\leq j\leq \infty.
\end{align}
Combining the discrete compactness of $\{\mathcal{H}_0^{h_{n_j}}\}_{j=1}^\infty$
and the boundedness of $\{\boldsymbol{\xi}_j\}_{j=1}^\infty$ from \eqref{eq:discrete-compactness-property2}, 
it immediately follows that $\{\boldsymbol{\xi}_j\}_{j=1}^\infty$ has a subsequence that converges in $\mathbf{L}^2(\Omega)$.
\end{proof}

Theorem \ref{thm:discrete-compactness-property-to-collectively-compact} reduces the proof of the collective compactness 
of the operator sequence $\{\mathcal{R}_{h}\}_{h\in\wedge}$ to proving the discrete compactness of the space sequence $\{\mathcal{H}_0^h\}_{h\in\wedge}$,
which is not a trivial result on a general bounded Lipschitz domain.
Fortunately, it follows from Lemma \ref{X_N-X_T-embedded-theorem} that there exists a constant $s>1/2$ such that the space $X_N(\Omega)$ 
is continuously embedded into $\mathbf{H}^s(\Omega)$ when $\Omega$ is a bounded Lipschitz polyhedron, which is precisely the case considered in this paper.
This immediately leads to the following conclusion.

\begin{thm}\label{thm:discrete-compactness-property}
For polyhedron $\Omega$, the space sequence $\{\mathcal{H}_0^h\}_{h\in\wedge}$ possesses the discrete compactness property 
with respect to the $\mathbf{H}(\mathbf{curl};\Omega)$ norm and embedding into $\mathbf{L}^2(\Omega)$.
\end{thm}

\begin{proof}
According to Definition \ref{def:discrete-compactness-property}, to prove the discrete compactness of 
the space sequence $\{\mathcal{H}_0^h\}_{h\in\wedge}$, 
it suffices to prove the following statement: If
$\{\boldsymbol{\xi}_{h_n}\}_{n=1}^\infty$ 
satisfies that there exists a constant $C_1>0$ such that for all $n\geq1$, $\boldsymbol{\xi}_{h_n}\in \mathcal{H}_0^{h_n}$ 
and $\|\boldsymbol{\xi}_{h_n}\|_{\mathbf{H}(\mathbf{curl};\Omega)}<C_1$, then $\{\boldsymbol{\xi}_{h_n}\}_{n=1}^\infty$ 
must have a subsequence which converges in $\mathbf{L}^2(\Omega)$.

For each $\boldsymbol{\xi}_{h_n}\in \mathcal{H}_0^{h_n}$, the following variational problem
\begin{align}
(\nabla p_n, \nabla q)=(\boldsymbol{\xi}_{h_n}, \nabla q), \quad\forall q\in \mathrm{H}_0^1(\Omega)
\end{align}
has a unique solution $p_n\in \mathrm{H}_0^1(\Omega)$ by the Lax-Milgram theorem.
Taking $q=p_n$ yields $\|\nabla p_n\|_{\mathbf{L}^2(\Omega)}\leq\|\boldsymbol{\xi}_{h_n}\|_{\mathbf{L}^2(\Omega)}$.
Let $\mathbf{w}_n=\boldsymbol{\xi}_{h_n}-\nabla p_n\in \mathbf{H}_0(\mathbf{curl};\Omega)$.
Then $(\mathbf{w}_n, \nabla q)=0$ holds for any $q\in \mathrm{H}_0^1(\Omega)$, which means that $\nabla\cdot \mathbf{w}_n=0$, 
$\mathbf{w}_n\in \mathcal{H}_0\subset X_N(\Omega)$.
Since $\Omega$ is a polyhedron,
$X_N(\Omega)$ is continuously embedded into $\mathbf{H}^s(\Omega)$ for some constant $s>1/2$ according to Lemma \ref{X_N-X_T-embedded-theorem}.
Consequently, there exists a constant $C_2$ such that
\begin{align}\label{eq:discrete-compactness-property-w_n_bounded}
\|\mathbf{w}_n\|_{\mathbf{H}^s(\Omega)}
&\leq C_2\|\mathbf{w}_n\|_{\mathbf{H}(\mathbf{curl};\Omega)}= C_2\|\boldsymbol{\xi}_{h_n}-\nabla p_n\|_{\mathbf{H}(\mathbf{curl};\Omega)} \nonumber\\
&\leq C_2(\|\nabla\times \boldsymbol{\xi}_{h_n}\|_{\mathbf{L}^2(\Omega)}+\|\boldsymbol{\xi}_{h_n}\|_{\mathbf{L}^2(\Omega)}+\|\nabla p_n\|_{\mathbf{L}^2(\Omega)})\nonumber\\
&\leq C_2(2(\|\nabla\times \boldsymbol{\xi}_{h_n}\|_{\mathbf{L}^2(\Omega)}^2+4\|\boldsymbol{\xi}_{h_n}\|_{\mathbf{L}^2(\Omega)}^2))^\frac{1}{2} \nonumber\\
&\leq 2\sqrt{2}C_2\|\boldsymbol{\xi}_{h_n}\|_{\mathbf{H}(\mathbf{curl};\Omega)} 
< 2\sqrt{2}C_1C_2,
\end{align}
which indicates that $\{\mathbf{w}_{n}\}_{n=1}^\infty$ is a bounded sequence in $\mathbf{H}^s(\Omega)$.
Furthermore, due to the compact embedding of $\mathbf{H}^s(\Omega)$ into $\mathbf{L}^2(\Omega)$,
there is a subsequence $\{\mathbf{w}_{n_k}\}_{k=1}^\infty$ and some $\mathbf{w}_0\in\mathbf{L}^2(\Omega)$ such that
$\mathbf{w}_{n_k}\to\mathbf{w}_0$ in $\mathbf{L}^2(\Omega)$ as $k\to\infty$.
By the triangle inequality
\begin{align}
\|\boldsymbol{\xi}_{h_{n_k}}-\mathbf{w}_0\|_{\mathbf{L}^2(\Omega)}
\leq\|\boldsymbol{\xi}_{h_{n_k}}-\mathbf{w}_{n_k}\|_{\mathbf{L}^2(\Omega)}
+\|\mathbf{w}_{n_k}-\mathbf{w}_0\|_{\mathbf{L}^2(\Omega)},
\end{align}
to prove the convergence of $\{\boldsymbol{\xi}_{h_{n_k}}\}_{k=1}^\infty$ to $\mathbf{w}_0$ in $\mathbf{L}^2(\Omega)$, it remains to show that
\begin{align}
\label{thm:discrete-compactness-property-xi_n_k-w_n_k-1}
\|\boldsymbol{\xi}_{h_{n_k}}-\mathbf{w}_{n_k}\|_{\mathbf{L}^2(\Omega)}\to 0\quad\mathrm{as}\quad k \to \infty.
\end{align}

Consider the smoothed projection interpolation operators $\Pi_{\mathbf{grad}, h_{n_k}}^{r+1}$ and $\boldsymbol\Pi_{\mathbf{curl}, h_{n_k}}^{r+1}$ given in Lemma \ref{lem:Commutative-Diagram}.
Since $\boldsymbol{\xi}_{h_{n_k}}\in \mathcal{H}_0^{h_{n_k}}\subset \mathscr{N}_{h_{n_k}, 0}^{r+1}$, it is straightforward to deduce that $\boldsymbol\Pi_{\mathbf{curl}, h_{n_k}}^{r+1}\boldsymbol{\xi}_{h_{n_k}}=\boldsymbol{\xi}_{h_{n_k}}$.
Furthermore, $p_n\in \mathrm{H}_0^1(\Omega)$ yields
\begin{align}
\label{eq:Commutative-Diagram-p_n_k}
\boldsymbol\Pi_{\mathbf{curl}, h_{n_k}}^{r+1}\nabla p_{n_k}=\nabla(\Pi_{\mathbf{grad}, h_{n_k}}^{r+1}p_{n_k}),
\end{align}
where $\Pi_{\mathbf{grad}, h_{n_k}}^{r+1}p_{n_k}\in \mathscr{L}_{h_{n_k}, 0}^{r+1}=Y_{h_{n_k}, 0}$ and $\boldsymbol\Pi_{\mathbf{curl}, h_{n_k}}^{r+1}\nabla p_{n_k}\in \mathscr{N}_{h_{n_k}, 0}^{r+1}$.
Combining \eqref{eq:Commutative-Diagram-p_n_k} with $\boldsymbol{\xi}_{h_{n_k}}\in \mathcal{H}_0^{h_{n_k}}$ and $\nabla\cdot \mathbf{w}_{n_k}=0$, it follows that
\begin{align}
&\|\mathbf{w}_{n_k}-\boldsymbol{\xi}_{h_{n_k}}\|_{\mathbf{L}^2(\Omega)}^2=(\mathbf{w}_{n_k}-\boldsymbol{\xi}_{h_{n_k}}, \mathbf{w}_{n_k}-\boldsymbol{\xi}_{h_{n_k}}) \nonumber\\
=&(\mathbf{w}_{n_k}-\boldsymbol{\xi}_{h_{n_k}}, \mathbf{w}_{n_k}-\boldsymbol\Pi_{\mathbf{curl}, h_{n_k}}^{r+1}\mathbf{w}_{n_k})-(\mathbf{w}_{n_k}-\boldsymbol{\xi}_{h_{n_k}}, \boldsymbol\Pi_{\mathbf{curl}, h_{n_k}}^{r+1}\boldsymbol{\xi}_{h_{n_k}}-\boldsymbol\Pi_{\mathbf{curl}, h_{n_k}}^{r+1}\mathbf{w}_{n_k}) \nonumber\\
=&(\mathbf{w}_{n_k}-\boldsymbol{\xi}_{h_{n_k}}, \mathbf{w}_{n_k}-\boldsymbol\Pi_{\mathbf{curl}, h_{n_k}}^{r+1}\mathbf{w}_{n_k})-(\mathbf{w}_{n_k}-\boldsymbol{\xi}_{h_{n_k}}, \boldsymbol\Pi_{\mathbf{curl}, h_{n_k}}^{r+1}\nabla p_{n_k}) \nonumber\\
=&(\mathbf{w}_{n_k}-\boldsymbol{\xi}_{h_{n_k}}, \mathbf{w}_{n_k}-\boldsymbol\Pi_{\mathbf{curl}, h_{n_k}}^{r+1}\mathbf{w}_{n_k})-(\mathbf{w}_{n_k}-\boldsymbol{\xi}_{h_{n_k}}, \nabla(\Pi_{\mathbf{grad}, h_{n_k}}^{r+1}p_{n_k})) \nonumber\\
=&(\mathbf{w}_{n_k}-\boldsymbol{\xi}_{h_{n_k}}, \mathbf{w}_{n_k}-\boldsymbol\Pi_{\mathbf{curl}, h_{n_k}}^{r+1}\mathbf{w}_{n_k}) \nonumber\\
\leq& \|\mathbf{w}_{n_k}-\boldsymbol{\xi}_{h_{n_k}}\|_{\mathbf{L}^2(\Omega)}\|\mathbf{w}_{n_k}-\boldsymbol\Pi_{\mathbf{curl}, h_{n_k}}^{r+1}\mathbf{w}_{n_k}\|_{\mathbf{L}^2(\Omega)},
\end{align}
which implies
\begin{align}
\label{thm:discrete-compactness-property-xi_n_k-w_n_k-2}
\|\mathbf{w}_{n_k}-\boldsymbol{\xi}_{h_{n_k}}\|_{\mathbf{L}^2(\Omega)}
\leq \|\mathbf{w}_{n_k}-\boldsymbol\Pi_{\mathbf{curl}, h_{n_k}}^{r+1}\mathbf{w}_{n_k}\|_{\mathbf{L}^2(\Omega)}.
\end{align}
According to Lemma \ref{lem:Commutative-Diagram}, $\boldsymbol\Pi_{\mathbf{curl}, h_{n_k}}^{r+1}$ is uniformly bounded in $\mathcal{L}(\mathbf{L}^2(\Omega))$.
Thus from the projection property and error estimate \eqref{eq:vector-Scott-Zhang-estimate}, there exists a constant $C_3$ such that
\begin{align}
\label{thm:discrete-compactness-property-xi_n_k-w_n_k-3}
\|\mathbf{w}_{n_k}-\boldsymbol\Pi_{\mathbf{curl}, h_{n_k}}^{r+1}\mathbf{w}_{n_k}\|_{\mathbf{L}^2(\Omega)}
&=\inf_{\mathbf{v}_{h_{n_k}}\in \mathscr{N}_{h_{n_k}, 0}^{r+1}}\|\mathbf{w}_{n_k}-\mathbf{v}_{h_{n_k}}-\boldsymbol\Pi_{\mathbf{curl}, h_{n_k}}^{r+1}(\mathbf{w}_{n_k}-\mathbf{v}_{h_{n_k}})\|_{\mathbf{L}^2(\Omega)} \nonumber\\ 
&\leq C_3\inf_{\mathbf{v}_{h_{n_k}}\in \mathscr{N}_{h_{n_k}, 0}^{r+1}}\|\mathbf{w}_{n_k}-\mathbf{v}_{h_{n_k}}\|_{\mathbf{L}^2(\Omega)} \nonumber\\ 
&\leq C_3\|\mathbf{w}_{n_k}-\boldsymbol\Pi_{h_{n_k}}^r\mathbf{w}_{n_k}\|_{\mathbf{L}^2(\Omega)}
\leq C_3h_{n_k}^{s}\|\mathbf{w}_{n_k}\|_{\mathbf{H}^s(\Omega)},
\end{align}
where the vector Scott-Zhang interpolation $\boldsymbol\Pi_{h_{n_k}}^r\mathbf{w}_{n_k}\in \mathscr{V}_{h_{n_k}, \tau0}^r\subset \mathscr{N}_{h_{n_k},0}^{r+1}$.
Substituting
\eqref{eq:discrete-compactness-property-w_n_bounded} into \eqref{thm:discrete-compactness-property-xi_n_k-w_n_k-3} and combining it with \eqref{thm:discrete-compactness-property-xi_n_k-w_n_k-2} gives \eqref{thm:discrete-compactness-property-xi_n_k-w_n_k-1}.
Thus, we prove that there is a subsequence $\{\boldsymbol{\xi}_{h_{n_k}}\}_{k=1}^\infty$ of $\{\boldsymbol{\xi}_{h_n}\}_{n=1}^\infty$ converges to $\mathbf{w}_0$ in $\mathbf{L}^2(\Omega)$.
\end{proof}

\begin{remark}
The specific choice $Y_{h, 0}=\mathscr{L}_{h, 0}^{r+1}$, which is one order higher than $\mathbf{X}_{h, \tau0}=\mathscr{V}_{h, \tau0}^r$, plays an essential role in the proofs of Theorem \ref{thm:mix-discrete-variational-form_nonzero-eigenvalue} (elimination of spurious modes) and Theorem \ref{thm:discrete-compactness-property} (discrete compactness).
The necessity of this design is fundamentally rooted in the inclusion relationship $\mathscr{V}_{h, \tau0}^r\subset \mathscr{N}_{h, 0}^{r+1}$ and the discrete exact sequence property 
in Lemma \ref{discrete-exactline}.

\end{remark}

\subsection{Pointwise convergence}
To establish the spectral convergence of the operator sequence $\{\mathcal{R}_{h_n}\}_{n=1}^\infty$, it is also necessary to demonstrate the pointwise convergence of $\{\mathcal{R}_{h_n}\}_{n=1}^\infty$ to the operator $\mathcal{R}$ on $\mathbf{L}^2(\Omega)$.
The following property which involves the approximation of functions in $\mathcal{H}_0$ by sequences of discrete functions in $\{\mathcal{H}_0^h\}_{h\in\wedge}$
will be utilized in the subsequent proof.


\begin{lem}
\label{lem:H_0^h-approximate-H_0}
For any given $\boldsymbol{\eta}\in \mathcal{H}_0$, suppose $\boldsymbol{\eta}\in \mathbf{H}_0^t(\mathbf{curl};\Omega)$ with $t>1/2$, then there is a constant $C_t>0$ independent of $\boldsymbol{\eta}$, and there exists a sequence $\boldsymbol{\eta}_h\in \mathcal{H}_0^h$ such that
\begin{align}
\label{maxwell-eigenvector-discrete-approximate_Ht}
\|\boldsymbol{\eta}-\boldsymbol{\eta}_h\|_{\mathbf{H}(\mathbf{curl};\Omega)}
\leq C_t h^{\min\{t, r\}}\|\boldsymbol{\eta}\|_{\mathbf{H}^t(\mathbf{curl};\Omega)}, \quad\forall h\in\wedge.
\end{align}
If $t\leq1/2$, there also exists a sequence $\boldsymbol{\eta}_h\in \mathcal{H}_0^h$ such that
\begin{align}
\label{eq:H_0^h-approximate-H_0}
\|\boldsymbol{\eta}-\boldsymbol{\eta}_h\|_{\mathbf{H}(\mathbf{curl};\Omega)}\to0\quad\mathrm{as}\quad h\to0~\mathrm{in}\wedge.
\end{align}
\end{lem}

\begin{proof}

According to Theorem \ref{thm:Lipschitz-contractible-order_s+1-decomposion-zeroboundary}, there exist functions $\mathbf{u}=\mathsf{P}\boldsymbol{\eta}\in \mathbf{H}^{t+1}(\Omega)\cap \mathbf{H}_0(\mathbf{curl};\Omega)$ and $p=\mathsf{Q}\boldsymbol{\eta}\in \mathrm{H}^{t+1}(\Omega)\cap \mathrm{H}_0^1(\Omega)$ such that $\boldsymbol{\eta}=\mathbf{u}+\nabla p$. Additionally, there exist constants $C_{1, t}$ and $C_{2, t}$, both independent of $\boldsymbol{\eta}$, satisfying
\begin{align}
\|\mathbf{u}\|_{\mathbf{H}^{t+1}(\Omega)}\leq C_{1, t} \|\boldsymbol{\eta}\|_{\mathbf{H}^t(\mathbf{curl};\Omega)},  \ \ \ \ 
\|p\|_{\mathrm{H}^{t+1}(\Omega)}\leq C_{2, t} \|\boldsymbol{\eta}\|_{\mathbf{H}^t(\mathbf{curl};\Omega)}.
\end{align}
Consider the Scott-Zhang interpolation functions $\boldsymbol\Pi_h^r \mathbf{u}\in \mathscr{V}_{h, \tau0}^r=\mathbf{X}_{h, \tau0}$ and $\Pi_h^{r+1} p\in \mathscr{L}_{h, 0}^{r+1}=Y_{h, 0}$.
From the error estimates \eqref{eq:scalar-Scott-Zhang-estimate} and \eqref{eq:vector-Scott-Zhang-estimate}, there is a constant $C$ such that
\begin{align}\label{maxwell-eigenvector-discrete-approximate_SZestimate}
\|\mathbf{u} - \boldsymbol\Pi_h^r \mathbf{u}\|_{\mathbf{H}^1(\Omega)} &\leq Ch^{\min\{t, r\}} \|\mathbf{u}\|_{\mathbf{H}^{t+1}(\Omega)}\leq CC_{1, t}h^{\min\{t, r\}}\|\boldsymbol{\eta}\|_{\mathbf{H}^t(\mathbf{curl};\Omega)},  \nonumber\\
\|p - \Pi_h^{r+1} p\|_{\mathrm{H}^1(\Omega)} &\leq Ch^{\min\{t, r+1\}} \|p\|_{\mathrm{H}^{t+1}(\Omega)}\leq CC_{2, t}h^{\min\{t, r+1\}}\|\boldsymbol{\eta}\|_{\mathbf{H}^t(\mathbf{curl};\Omega)}.
\end{align}
By Theorem \ref{cor:X_h,t0-to-H_0^h}, there exists a unique $p_{h}\in Y_{h, 0}$ such that $\boldsymbol{\eta}_h=\boldsymbol\Pi_{h}^r \mathbf{u} + \nabla \Pi_{h}^{r+1} p+\nabla p_{h}\in \mathcal{H}_0^{{h}}$, and combining with $\boldsymbol{\eta}=\mathbf{u}+\nabla p\in \mathbf{H}(\mathrm{div}^0;\Omega)$ yields
\begin{align}
\label{eq:space_appro-nabla p_h—estimate}
&(\boldsymbol{\eta}-\boldsymbol{\eta}_h,\boldsymbol{\eta}-\boldsymbol{\eta}_h)_{\mathbf{H}(\mathbf{curl};\Omega)}
=(\boldsymbol{\eta}-\boldsymbol{\eta}_h,\mathbf{u}+\nabla p-\boldsymbol\Pi_{h}^r \mathbf{u} - \nabla \Pi_{h}^{r+1} p)_{\mathbf{H}(\mathbf{curl};\Omega)} \nonumber\\
\leq& \|\boldsymbol{\eta}-\boldsymbol{\eta}_h\|_{\mathbf{H}(\mathbf{curl};\Omega)}
\|\mathbf{u}+\nabla p-\boldsymbol\Pi_{h}^r \mathbf{u} - \nabla \Pi_{h}^{r+1} p\|_{\mathbf{H}(\mathbf{curl};\Omega)}.
\end{align}
Substituting the interpolation estimate \eqref{maxwell-eigenvector-discrete-approximate_SZestimate} into \eqref{eq:space_appro-nabla p_h—estimate}, it follows that
\begin{align}
&\|\boldsymbol{\eta}-\boldsymbol{\eta}_h\|_{\mathbf{H}(\mathbf{curl};\Omega)} 
\leq \|\mathbf{u}+\nabla p-\boldsymbol\Pi_{h}^r \mathbf{u} - \nabla \Pi_{h}^{r+1} p\|_{\mathbf{H}(\mathbf{curl};\Omega)} \nonumber\\
\leq& \|\mathbf{u}-\boldsymbol\Pi_{h}^r \mathbf{u}\|_{\mathbf{H}(\mathbf{curl};\Omega)}+\|\nabla p-\nabla \Pi_{h}^{r+1} p\|_{\mathbf{L}^2(\Omega)} \nonumber\\
\leq& \|\mathbf{u}-\boldsymbol\Pi_{h}^r \mathbf{u}\|_{\mathbf{H}^1(\Omega)}+\|p-\Pi_{h}^{r+1} p\|_{\mathrm{H}^1(\Omega)} \nonumber\\
\leq&  C(C_{1, t}h^{\min\{t, r\}}+C_{2, t}h^{\min\{t, r+1\}})\|\boldsymbol{\eta}\|_{\mathbf{H}^t(\mathbf{curl};\Omega)} \nonumber\\
\leq&  C_th^{\min\{t, r\}}\|\boldsymbol{\eta}\|_{\mathbf{H}^t(\mathbf{curl};\Omega)}, 
\end{align}
where $C_t=C(C_{1, t}+C_{2, t})$ is a constant depending only on $t$ and $\Omega$.

For the case $t\leq1/2$, one only needs to split $\boldsymbol{\eta}$ as $\boldsymbol{\eta}=\mathbf{u}+\nabla p$ by the regular decomposition \eqref{order-1-decomposion}, where $\mathbf{u}\in \mathbf{H}_0^1(\Omega)$ and $p\in \mathrm{H}_0^1(\Omega)$, and then utilize the density of $\mathrm{C}_0^\infty(\Omega)$ in $\mathrm{H}_0^1(\Omega)$ to perform similar derivations as above for smooth functions sufficiently close to $\mathbf{u}$ and $p$, thereby obtaining a sequence $\boldsymbol{\eta}_h\in \mathcal{H}_0^h$ that approximates $\boldsymbol{\eta}$.
\end{proof}

\begin{thm}
\label{thm:Rh-Point convergence-R}
For any given $\mathbf{f}\in \mathbf{L}^2(\Omega)$, the norm convergence
\begin{align}
\|\mathcal{R}\mathbf{f}-\mathcal{R}_{h_n}\mathbf{f}\|_{\mathbf{H}(\mathbf{curl};\Omega)}\to0\quad\mathrm{as}\quad n\to\infty
\end{align}
holds. 
This norm convergence implies the pointwise convergence of the operator sequence $\{\mathcal{R}_{h_n}\}_{n=1}^\infty$ to the operator $\mathcal{R}$ on $\mathbf{L}^2(\Omega)$.

\end{thm}

\begin{proof}
For any $\mathbf{f}\in \mathbf{L}^2(\Omega)$, the estimate \eqref{eq:Rhf-form-control} gives
\begin{align}
\|\mathcal{R}_{h_n}\mathbf{f}\|_{\mathbf{H}(\mathbf{curl};\Omega)}\leq\|\mathbf{f}\|_{\mathbf{L}^2(\Omega)}, \quad\forall n\geq1,
\end{align}
Therefore, $\{\mathcal{R}_{h_n}\mathbf{f}\}_{n=1}^\infty$ is a bounded sequence in $\mathbf{H}_0(\mathbf{curl};\Omega)$.
Since $\mathbf{H}_0(\mathbf{curl};\Omega)$, as a Hilbert space, is a reflexive Banach space, by the Banach-Alaoglu theorem,
there is a subsequence $\{\mathcal{R}_{h_{n_k}}\mathbf{f}\}_{k=1}^\infty$ of $\{\mathcal{R}_{h_n}\mathbf{f}\}_{n=1}^\infty$ converges weakly to $\mathbf{g}\in \mathbf{H}_0(\mathbf{curl};\Omega)$.
Furthermore, for any $q\in\mathrm{H}_0^1(\Omega)$,
From the weak convergence of $\{\mathcal{R}_{h_{n_k}}\mathbf{f}\}_{k=1}^\infty$ and the approximability of the Lagrange finite element space $\mathscr{L}_{h_{n_k}, 0}^{r+1}=Y_{h_{n_k}, 0}$ to $\mathrm{H}_0^1(\Omega)$, there is a sequence $q_{h_{n_k}}\in Y_{h_{n_k}, 0}$ such that
\begin{align}
(\mathbf{g},\nabla q)
&=\lim_{k\to\infty}(\mathcal{R}_{h_{n_k}}\mathbf{f},\nabla q)
=\lim_{k\to\infty}(\mathcal{R}_{h_{n_k}}\mathbf{f},\nabla q-\nabla q_{h_{n_k}}) \nonumber\\
&\leq \|\mathbf{f}\|_{\mathbf{L}^2(\Omega)} \lim_{k\to\infty} \|q - q_{h_{n_k}}\|_{\mathrm{H}^1(\Omega)}
=0,
\end{align}
which implies that $\mathbf{g}\in \mathbf{H}_0(\mathbf{curl};\Omega)\cap \mathbf{H}(\mathrm{div}^0;\Omega)=\mathcal{H}_0$.
Consider any $\boldsymbol{\eta}\in \mathcal{H}_0$, by Lemma \ref{lem:H_0^h-approximate-H_0}, there exists a sequence $\boldsymbol{\eta}_{h_{n_k}}\in \mathcal{H}_0^{h_{n_k}}$ such that
$\boldsymbol{\eta}_{h_{n_k}}\to\boldsymbol{\eta}$ in $\mathbf{H}_0(\mathbf{curl};\Omega)$ as $k\to\infty$.
Notice that for each $\boldsymbol{\eta}_{h_{n_k}}$, $\mathcal{R}_{h_{n_k}}\mathbf{f}$ satisfies
\begin{align}
\label{eqthm:Rh-Point convergence-R-1}
(\nabla\times\mathcal{R}_{h_{n_k}}\mathbf{f}, \nabla\times\boldsymbol{\eta}_{h_{n_k}})+(\mathcal{R}_{h_{n_k}}\mathbf{f}, \boldsymbol{\eta}_{h_{n_k}})=(\mathcal{R}_{h_{n_k}}\mathbf{f}, \boldsymbol{\eta}_{h_{n_k}})_{\mathbf{H}(\mathbf{curl};\Omega)}=(\mathbf{f}, \boldsymbol{\eta}_{h_{n_k}}),
\end{align}
where the norm convergence of $\{\boldsymbol{\eta}_{h_{n_k}}\}_{k=1}^\infty$ in $\mathbf{H}(\mathbf{curl};\Omega)$ implies
\begin{align}
\label{eqthm:Rh-Point convergence-R-11}
(\mathbf{f}, \boldsymbol{\eta}_{h_{n_k}})\to(\mathbf{f}, \boldsymbol{\eta})\quad\mathrm{as}\quad k\to\infty,
\end{align}
and the weak convergence of $\{\mathcal{R}_{h_{n_k}}\mathbf{f}\}_{k=1}^\infty$ to $\mathbf{g}$ in $\mathbf{H}_0(\mathbf{curl};\Omega)$ implies
\begin{align}
\label{eqthm:Rh-Point convergence-R-12}
&|(\mathcal{R}_{h_{n_k}}\mathbf{f}, \boldsymbol{\eta}_{h_{n_k}})_{\mathbf{H}(\mathbf{curl};\Omega)}-(\mathbf{g}, \boldsymbol{\eta})_{\mathbf{H}(\mathbf{curl};\Omega)}| \nonumber\\
=& |(\mathcal{R}_{h_{n_k}}\mathbf{f}, \boldsymbol{\eta}_{h_{n_k}}-\boldsymbol{\eta})_{\mathbf{H}(\mathbf{curl};\Omega)}
+(\mathcal{R}_{h_{n_k}}\mathbf{f}-\mathbf{g}, \boldsymbol{\eta})_{\mathbf{H}(\mathbf{curl};\Omega)}| \nonumber\\
\leq& |(\mathcal{R}_{h_{n_k}}\mathbf{f}, \boldsymbol{\eta}_{h_{n_k}}-\boldsymbol{\eta})_{\mathbf{H}(\mathbf{curl};\Omega)}|
+|(\mathcal{R}_{h_{n_k}}\mathbf{f}-\mathbf{g}, \boldsymbol{\eta})_{\mathbf{H}(\mathbf{curl};\Omega)}| \nonumber\\
\leq& \|\mathcal{R}_{h_{n_k}}\mathbf{f}\|_{\mathbf{H}(\mathbf{curl;}\Omega)}
\|\boldsymbol{\eta}-\boldsymbol{\eta}_{h_{n_k}}\|_{\mathbf{H}(\mathbf{curl};\Omega)}
+|(\mathcal{R}_{h_{n_k}}\mathbf{f}-\mathbf{g}, \boldsymbol{\eta})_{\mathbf{H}(\mathbf{curl};\Omega)}| \nonumber\\
\leq& \|\mathbf{f}\|_{\mathbf{L}^2(\Omega)}\|\boldsymbol{\eta}-\boldsymbol{\eta}_{h_{n_k}}\|_{\mathbf{H}(\mathbf{curl};\Omega)} 
+|(\mathcal{R}_{h_{n_k}}\mathbf{f}-\mathbf{g}, \boldsymbol{\eta})_{\mathbf{H}(\mathbf{curl};\Omega)}| 
\to ~0\quad\mathrm{as}\quad k\to\infty.
\end{align}
Let $k\to\infty$ in \eqref{eqthm:Rh-Point convergence-R-1}, it follows from \eqref{eqthm:Rh-Point convergence-R-11} and \eqref{eqthm:Rh-Point convergence-R-12} that
\begin{align}
(\mathbf{g}, \boldsymbol{\eta})_{\mathbf{H}(\mathbf{curl};\Omega)}
=(\mathbf{f}, \boldsymbol{\eta})
=(\mathcal{R}\mathbf{f}, \boldsymbol{\eta})_{\mathbf{H}(\mathbf{curl};\Omega)}.
\end{align}
This holds for any $\boldsymbol{\eta}\in \mathcal{H}_0$, which means that
$\{\mathcal{R}_{h_{n_k}}\mathbf{f}\}_{k=1}^\infty$ converges weakly to $\mathbf{g}=\mathcal{R}\mathbf{f}$ in $\mathbf{H}_0(\mathbf{curl};\Omega)$.
This conclusion holds for every weakly convergent subsequence of $\{\mathcal{R}_{h_n}\mathbf{f}\}_{n=1}^\infty$, so $\{\mathcal{R}_{h_n}\mathbf{f}\}_{n=1}^\infty$ also converges weakly to $\mathcal{R}\mathbf{f}$ in $\mathbf{H}_0(\mathbf{curl};\Omega)$.
Otherwise, there would exist some $\mathbf{f}_0\in \mathbf{H}_0(\mathbf{curl};\Omega)$, a positive number $\varepsilon >0$, and a countable subset $\wedge_\varepsilon \subset\wedge$ such that
\begin{align}
\label{eq:Rh-weak convergence-R-ifnot}
|(\mathcal{R}_{h_n}\mathbf{f}-\mathcal{R}\mathbf{f}, \boldsymbol{\eta}_0)_{\mathbf{H}(\mathbf{curl};\Omega)}|>\varepsilon , \quad\forall n\in\wedge_\varepsilon .
\end{align}
However, $\{\mathcal{R}_{h_n}\mathbf{f}\}_{n\in\wedge_\varepsilon }$ is also a bounded sequence in $\mathcal{H}_0$ and must have a subsequence that converges weakly to $\mathcal{R}\mathbf{f}$ in $\mathcal{H}_0$, which contradicts \eqref{eq:Rh-weak convergence-R-ifnot}.

Taking $\boldsymbol{\eta}=\mathcal{R}\mathbf{f}\in \mathcal{H}_0$ and $\boldsymbol{\eta}_{h_n}=\mathcal{R}_{h_n}\mathbf{f}\in \mathcal{H}_0^{h_n}$ in \eqref{eq:def-operater-R} and \eqref{eq:def-operater-R_h}, respectively, yields
\begin{align}
\label{eq:thm:Rh-Point convergence-R_1}
\|\mathcal{R}\mathbf{f}\|_{\mathbf{H}(\mathbf{curl};\Omega)}^2=(\mathbf{f}, \mathcal{R}\mathbf{f}), \quad
\|\mathcal{R}_{h_n}\mathbf{f}\|_{\mathbf{H}(\mathbf{curl};\Omega)}^2=(\mathbf{f}, \mathcal{R}_{h_n}\mathbf{f}).
\end{align}
The weak convergence from $\{\mathcal{R}_{h_n}\mathbf{f}\}_{n=1}^\infty$ to $\mathcal{R}\mathbf{f}$ indicates
\begin{eqnarray}
&&\lim_{n\to\infty}(\mathbf{f}, \mathcal{R}_{h_n}\mathbf{f})=(\mathbf{f}, \mathcal{R}\mathbf{f}),\label{eq:thm:Rh-Point convergence-R_2}\\
&&\lim_{n\to\infty}(\mathcal{R}\mathbf{f}, \mathcal{R}_{h_n}\mathbf{f})_{\mathbf{H}(\mathbf{curl};\Omega)}
=(\mathcal{R}\mathbf{f}, \mathcal{R}\mathbf{f})_{\mathbf{H}(\mathbf{curl};\Omega)}.\label{eq:thm:Rh-Point convergence-R_3}
\end{eqnarray}
Combining \eqref{eq:thm:Rh-Point convergence-R_1}, \eqref{eq:thm:Rh-Point convergence-R_2} and \eqref{eq:thm:Rh-Point convergence-R_3}, it follows that
\begin{align}
&\lim_{n\to\infty}\|\mathcal{R}_{h_n}\mathbf{f}-\mathcal{R}\mathbf{f}\|_{\mathbf{H}(\mathbf{curl};\Omega)}^2 
=\lim_{n\to\infty}(\mathcal{R}_{h_n}\mathbf{f}-\mathcal{R}\mathbf{f}, \mathcal{R}_{h_n}\mathbf{f}
-\mathcal{R}\mathbf{f}) _{\mathbf{H}(\mathbf{curl};\Omega)} \nonumber\\
=& \|\mathcal{R}\mathbf{f}\|_{\mathbf{H}(\mathbf{curl};\Omega)}^2 
+ \lim_{n\to\infty}\|\mathcal{R}_{h_n}\mathbf{f}\|_{\mathbf{H}(\mathbf{curl};\Omega)}^2 
-2\lim_{n\to\infty}(\mathcal{R}\mathbf{f}, \mathcal{R}_{h_n}\mathbf{f})_{\mathbf{H}(\mathbf{curl};\Omega)} \nonumber\\
=& 2\|\mathcal{R}\mathbf{f}\|_{\mathbf{H}(\mathbf{curl};\Omega)}^2 
-2(\mathcal{R}\mathbf{f}, \mathcal{R}\mathbf{f})_{\mathbf{H}(\mathbf{curl};\Omega)}=0,
\end{align}
which proves that $\{\mathcal{R}_{h_n}\mathbf{f}\}_{n=1}^\infty$ converges to $\mathcal{R}\mathbf{f}$ 
in the norm of $\mathbf{H}(\mathbf{curl};\Omega)$.
\end{proof}

\section{Error estimates for eigenpair approximations}\label{Section_Error_Estimates}
In Section \ref{Section_Operator}, we have respectively proven via Theorems \ref{thm:discrete-compactness-property} and \ref{thm:Rh-Point convergence-R} that the sequence of operators $\{\mathcal{R}_{h_n}\}_{n=1}^\infty$ is collectively compact and converges pointwise to the compact operator $\mathcal{R}$. Combined with the self-adjointness of each operator, we conclude that the discrete positive eigenvalues and corresponding eigenfunctions of $\{\mathcal{R}_{h_n}\}_{n=1}^\infty$ will approximate the positive eigenvalues and corresponding eigenfunctions of the operator $\mathcal{R}$.
In this section, we will analyze the convergence rate of the spectrum of the operator 
sequence $\{\mathcal{R}_{h_n}\}_{n=1}^\infty$ by utilizing Lemmas \ref{selfadjoint-operator-eigenvalue-order} and \ref{operator-eigenvector-order}.

\begin{thm}\label{thm:R-R_h-multi-eigenvalue-eigenvector-approximate_order}
Let $\Omega$ be a bounded contractible Lipschitz polyhedron in $\mathbb{R}^3$, 
and $\mu$ be an eigenvalue of the operator $\mathcal{R}$ with multiplicity $m$. 
Then for sufficiently small $h$, there exists an open interval centered at $\mu$ 
that contains exactly $m$ eigenvalues $\mu_{h, j},j=1, \cdots, m$ of the operator $\mathcal{R}_h$.
Let $\{\boldsymbol{\xi}_1, \ldots, \boldsymbol{\xi}_m\}$ be an $\mathbf{L}^2(\Omega)$-orthonormal basis of 
the eigenvector space $\ker(\mu-\mathcal{R})$ of $\mathcal{R}$ corresponding to $\mu$. 
Suppose each $\boldsymbol{\xi}_j\in \mathcal{H}_0$ satisfies $\boldsymbol{\xi}_j
\in \mathbf{H}_0^t(\mathbf{curl};\Omega)$ with some $t>1/2$, then there is 
a constant $C$ depending on $t$ and $\Omega$ such that
\begin{align}
\label{eq:R-R_h-multi-eigenvalue-approximate_order}
\left|\mu-\mu_{h, j}\right|
\leq Cm(m+1)\mu^2h^{2\min\{t, r\}}\max_{1\leq j\leq m}
\|\boldsymbol{\xi}_j\|_{\mathbf{H}^t(\mathbf{curl};\Omega)}^2, \quad\forall j=1, \cdots, m,
\end{align}
where the constant $C$ depends on 
the real number $t$ and the polyhedron $\Omega$.
\end{thm}

\begin{proof}
From Theorems \ref{thm:discrete-compactness-property} and \ref{thm:Rh-Point convergence-R},  
the sequence of adjoint operators $\{\mathcal{R}_{h_n}\}_{n=1}^\infty$ is collectively compact and converges 
pointwise to the adjoint compact operator $\mathcal{R}$.  
Then, for sufficiently small $h$, there exists an open interval 
centered at $\mu$ that contains exactly $m$ eigenvalues 
$\mu_{h, j},j=1, \cdots, m$ of the operator $\mathcal{R}_h$.

By Lemma \ref{selfadjoint-operator-eigenvalue-order}, there exists a constant $C_1$ such that for 
each $\mu_{h, j}$, $1\leq j\leq m$,
\begin{align}\label{multi-eigenvalue-estimate}
\left|\mu-\mu_{h, j}\right|\leq C_1
\left\{\sum\limits_{i, j=1}^m \left|\left((\mathcal{R}-\mathcal{R}_h)\boldsymbol{\xi}_i, \boldsymbol{\xi}_j\right)\right| 
+ \left\|(\mathcal{R}-\mathcal{R}_h)|_{\ker(\mu-\mathcal{R})}\right\|^2 \right\}, 
\end{align}
where $\ker(\mu-\mathcal{R})$ is the space spanned by the eigenfunctions 
$\boldsymbol{\xi}_1, \ldots, \boldsymbol{\xi}_m$.
For each $\boldsymbol{\xi}_j\in \mathcal{H}_0, 1\leq j\leq m$, Theorem \ref{thm:Rxi-Rhxi-orthogonal} implies that
\begin{align}
\label{eq:eigenvector-Rxi_j-Rhxi_j-orthogonal}
(\mathcal{R}\boldsymbol{\xi}_j-\mathcal{R}_{h}\boldsymbol{\xi}_j, \mathbf{v}_h
+\nabla q_h)_{\mathbf{H}(\mathbf{curl};\Omega)}=0, 
\quad\forall \mathbf{v}_h\in \mathbf{X}_{h, \tau0},\ \ 
\forall  q_{h}\in Y_{h, 0}.
\end{align}
Further, according to Lemma \ref{lem:H_0^h-approximate-H_0}, there exist a constant $C_{2, j}$ and a discrete function sequence $\boldsymbol{\xi}_{j,h}\in \mathcal{H}_0^h$ satisfying
\begin{align}
\label{eq:eigenvector-xi_j-xi_jh-curldistance}
\|\boldsymbol{\xi}_j-\boldsymbol{\xi}_{j,h}\|_{\mathbf{H}(\mathbf{curl};\Omega)}\leq C_{2, j}h^{\min\{t, r\}}\|\boldsymbol{\xi}_j\|_{\mathbf{H}^t(\mathbf{curl};\Omega)}.
\end{align}
Setting $\mathbf{v}_h+\nabla q_h=\mu\boldsymbol{\xi}_{j,h}-\mathcal{R}_{h}\boldsymbol{\xi}_j\in \mathcal{H}_0^h$ in
\eqref{eq:eigenvector-Rxi_j-Rhxi_j-orthogonal}, it follows that
\begin{align}
\label{eq:eigenvector-Rxi_j-Rhxi_j-distance}
\|\mathcal{R}\boldsymbol{\xi}_j-\mathcal{R}_{h}\boldsymbol{\xi}_j\|_{\mathbf{H}(\mathbf{curl};\Omega)}^2
=&(\mathcal{R}\boldsymbol{\xi}_j-\mathcal{R}_{h}\boldsymbol{\xi}_j, \mathcal{R}\boldsymbol{\xi}_j-\mathcal{R}_{h}\boldsymbol{\xi}_j)_{\mathbf{H}(\mathbf{curl};\Omega)} \nonumber\\
=&(\mathcal{R}\boldsymbol{\xi}_j-\mathcal{R}_{h}\boldsymbol{\xi}_j, \mathcal{R}\boldsymbol{\xi}_j
-\mathcal{R}_{h}\boldsymbol{\xi}_j-(\mu\boldsymbol{\xi}_{j,h}
-\mathcal{R}_{h}\boldsymbol{\xi}_j))_{\mathbf{H}(\mathbf{curl};\Omega)} \nonumber\\
=&(\mathcal{R}\boldsymbol{\xi}_j-\mathcal{R}_{h}\boldsymbol{\xi}_j, \mu\boldsymbol{\xi}_j-\mu\boldsymbol{\xi}_{j,h})_{\mathbf{H}(\mathbf{curl};\Omega)} \nonumber\\
\leq&\|\mathcal{R}\boldsymbol{\xi}_j-\mathcal{R}_{h}\boldsymbol{\xi}_j\|_{\mathbf{H}(\mathbf{curl};\Omega)}
\|\mu\boldsymbol{\xi}_j-\mu\boldsymbol{\xi}_{j,h}\|_{\mathbf{H}(\mathbf{curl};\Omega)}.
\end{align}
Substituting \eqref{eq:eigenvector-xi_j-xi_jh-curldistance} into \eqref{eq:eigenvector-Rxi_j-Rhxi_j-distance} yields
\begin{align}
\label{eq:eigenvector-Rxi_j-Rhxi_j-order}
\|\mathcal{R}\boldsymbol{\xi}_j-\mathcal{R}_{h}\boldsymbol{\xi}_j\|_{\mathbf{H}(\mathbf{curl};\Omega)}
\leq \mu\|\boldsymbol{\xi}_j-\boldsymbol{\xi}_{j,h}\|_{\mathbf{H}(\mathbf{curl};\Omega)}
\leq C_{2, j}\mu h^{\min\{t, r\}}\|\boldsymbol{\xi}_j\|_{\mathbf{H}^t(\mathbf{curl};\Omega)}.
\end{align}
Let $C_2=\max\limits_{1\leq j\leq m}C_{2, j}$.
On the one hand, taking the test functions as $\mathcal{R}_h\boldsymbol{\xi}_i\in \mathcal{H}_0^h$ 
and $\mathcal{R}_h\boldsymbol{\xi}_j\in \mathcal{H}_0^h$ in the orthogonality relation \eqref{eq:eigenvector-Rxi_j-Rhxi_j-orthogonal} respectively, we obtain
\begin{align}
(\mathcal{R}\boldsymbol{\xi}_i-\mathcal{R}_h\boldsymbol{\xi}_i, \mathcal{R}_h\boldsymbol{\xi}_j)_{\mathbf{H}(\mathbf{curl};\Omega)}
=(\mathcal{R}\boldsymbol{\xi}_j-\mathcal{R}_h\boldsymbol{\xi}_j, \mathcal{R}_h\boldsymbol{\xi}_i)_{\mathbf{H}(\mathbf{curl};\Omega)}=0.
\end{align}
Combining this with \eqref{eq:eigenvector-Rxi_j-Rhxi_j-order} 
and self-adjointness of $\mathcal R$ and $\mathcal R_h$, it follows that
\begin{align}
&|((\mathcal{R}-\mathcal{R}_h)\boldsymbol{\xi}_i, \boldsymbol{\xi}_j)|=|(\boldsymbol{\xi}_j, \mathcal{R}\boldsymbol{\xi}_i)-(\boldsymbol{\xi}_j, \mathcal{R}_h\boldsymbol{\xi}_i)| \nonumber\\
=&|(\mathcal{R}\boldsymbol{\xi}_j, \mathcal{R}\boldsymbol{\xi}_i)_{\mathbf{H}(\mathbf{curl};\Omega)}-(\mathcal{R}_h\boldsymbol{\xi}_j, \mathcal{R}_h\boldsymbol{\xi}_i)_{\mathbf{H}(\mathbf{curl};\Omega)}| \nonumber\\
=&|(\mathcal{R}\boldsymbol{\xi}_j-\mathcal{R}_h\boldsymbol{\xi}_j, \mathcal{R}\boldsymbol{\xi}_i)_{\mathbf{H}(\mathbf{curl};\Omega)}+(\mathcal{R}_h\boldsymbol{\xi}_j, \mathcal{R}\boldsymbol{\xi}_i-\mathcal{R}_h\boldsymbol{\xi}_i)_{\mathbf{H}(\mathbf{curl};\Omega)}| \nonumber\\
=&|(\mathcal{R}\boldsymbol{\xi}_j-\mathcal{R}_h\boldsymbol{\xi}_j, \mathcal{R}\boldsymbol{\xi}_i)_{\mathbf{H}(\mathbf{curl};\Omega)}| \nonumber\\
=&|(\mathcal{R}\boldsymbol{\xi}_j-\mathcal{R}_h\boldsymbol{\xi}_j, \mathcal{R}\boldsymbol{\xi}_i-\mathcal{R}_h\boldsymbol{\xi}_i)_{\mathbf{H}(\mathbf{curl};\Omega)}
+(\mathcal{R}\boldsymbol{\xi}_j-\mathcal{R}_h\boldsymbol{\xi}_j, \mathcal{R}_h\boldsymbol{\xi}_i)_{\mathbf{H}(\mathbf{curl};\Omega)}| \nonumber\\
=&|(\mathcal{R}\boldsymbol{\xi}_j-\mathcal{R}_h\boldsymbol{\xi}_j, \mathcal{R}\boldsymbol{\xi}_i-\mathcal{R}_h\boldsymbol{\xi}_i)_{\mathbf{H}(\mathbf{curl};\Omega)}| \nonumber\\
\leq&\|\mathcal{R}\boldsymbol{\xi}_j-\mathcal{R}_h\boldsymbol{\xi}_j\|_{\mathbf{H}(\mathbf{curl};\Omega)}
\|\mathcal{R}\boldsymbol{\xi}_i-\mathcal{R}_h\boldsymbol{\xi}_i\|_{\mathbf{H}(\mathbf{curl};\Omega)} \nonumber\\
\leq& C_{2, j}C_{2, i}\mu^2 h^{2\min\{t, r\}}\|\boldsymbol{\xi}_j\|_{\mathbf{H}^t(\mathbf{curl};\Omega)}\|\boldsymbol{\xi}_i\|_{\mathbf{H}^t(\mathbf{curl};\Omega)} \nonumber\\
\leq& C_2^2\mu^2 h^{2\min\{t, r\}}\max_{1\leq j\leq m}\|\boldsymbol{\xi}_j\|_{\mathbf{H}^t(\mathbf{curl};\Omega)}^2.
\end{align}
Thus, the first term
of \eqref{multi-eigenvalue-estimate} satisfies
\begin{align}
\label{eq:multi-eigenvalue-estimate1}
\sum\limits_{i, j=1}^m \left|\left((\mathcal{R}-\mathcal{R}_h)\boldsymbol{\xi}_i, \boldsymbol{\xi}_j\right)\right|\leq m^2C_2^2\mu^2 h^{2\min\{t, r\}}\max_{1\leq j\leq m}\|\boldsymbol{\xi}_j\|_{\mathbf{H}^t(\mathbf{curl};\Omega)}^2.
\end{align}
On the other hand, any function $\boldsymbol{\xi}\in\ker(\mu-\mathcal{R})$ 
can be expressed as $\boldsymbol{\xi}=\sum_{i=1}^mc_i\boldsymbol{\xi}_i$ 
where $c_i\in\mathbb{R}, 1\leq i\leq m$.
By the property of orthonormal bases, $\|\boldsymbol{\xi}\|_{\mathbf{L}^2(\Omega)}=1$ 
is equivalent to $\sum_{i=1}^mc_i^2=1$. 
Therefore, according to \eqref{eq:eigenvector-Rxi_j-Rhxi_j-order} and the Cauchy-Schwarz inequality, the second term 
of \eqref{multi-eigenvalue-estimate} satisfies
\begin{align}\label{eq:multi-eigenvalue-estimate2}
&\|(\mathcal{R}-\mathcal{R}_h)|_{\ker(\mu-\mathcal{R})}\|^2
=\sup_{\boldsymbol{\xi}\in\ker(\mu-\mathcal{R}), \|\boldsymbol{\xi}\|_{\mathbf{L}^2(\Omega)}=1}
\|(\mathcal{R}-\mathcal{R}_h)\boldsymbol{\xi}\|_{\mathbf{L}^2(\Omega)}^2 \nonumber\\
=&\sup_{\sum_{i=1}^mc_i^2=1}\bigg\|\sum_{i=1}^mc_i(\mathcal{R}-\mathcal{R}_h)\boldsymbol{\xi}_i\bigg\|_{\mathbf{L}^2(\Omega)}^2
\leq\sup_{\sum_{i=1}^mc_i^2=1} m\sum_{i=1}^mc_i^2\|(\mathcal{R}-\mathcal{R}_h)\boldsymbol{\xi}_i\|_{\mathbf{L}^2(\Omega)}^2 \nonumber\\
\leq&\sup_{\sum_{i=1}^mc_i^2=1} m\sum_{i=1}^mc_i^2C_{2, i}^2\mu^2 h^{2\min\{t, r\}}\|\boldsymbol{\xi}_i\|_{\mathbf{H}^t(\mathbf{curl};\Omega)}^2 \nonumber\\
\leq& mC_2^2\mu^2 h^{2\min\{t, r\}}\max_{1\leq j\leq m}\|\boldsymbol{\xi}_j\|_{\mathbf{H}^t(\mathbf{curl};\Omega)}^2.
\end{align}
Substituting \eqref{eq:multi-eigenvalue-estimate1} and \eqref{eq:multi-eigenvalue-estimate2} into \eqref{multi-eigenvalue-estimate}, the following error estimate holds 
\begin{align}
|\mu-\mu_{h, j}|\leq Cm(m+1)\mu^2h^{2\min\{t, r\}}\max_{1\leq j\leq m}\|\boldsymbol{\xi}_j\|_{\mathbf{H}^t(\mathbf{curl};\Omega)}^2, \quad\forall j=1, \cdots, m, 
\end{align}
where the constant $C=C_1C_2^2$ depends on $t$ and $\Omega$.
\end{proof}

\begin{cor}
According to the spectral equivalence in Theorems \ref{thm:operater-R-equivalence-variational} and \ref{thm:operater-Rh-equivalence-variational},
by taking $\mu=1/(\lambda+1)$ and $\mu_{h,j}=1/(\lambda_{h,j}+1)$ in \eqref{eq:R-R_h-multi-eigenvalue-approximate_order},
the following error estimate for the eigenvalues of the Maxwell system holds
\begin{align}
\label{eq:maxwell-multi-eigenvalue-approximate_order}
\left|\lambda-\lambda_{h, j}\right|
&=\left|\frac{1}{\mu}-\frac{1}{\mu_{h, j}}\right|
\leq \frac{2}{\mu^2}Cm(m+1)\mu^2h^{2\min\{t, r\}}\max_{1\leq j\leq m}\|\boldsymbol{\xi}_j\|_{\mathbf{H}^t(\mathbf{curl};\Omega)}^2 \nonumber\\
&\leq 2Cm(m+1)h^{2\min\{t, r\}}\max_{1\leq j\leq m}\|\boldsymbol{\xi}_j\|_{\mathbf{H}^t(\mathbf{curl};\Omega)}^2, \quad\forall j=1, \cdots, m,
\end{align}
where the constant $C$ depends on $t$ and $\Omega$.
\end{cor}

Theorem \ref{thm:R-R_h-multi-eigenvalue-eigenvector-approximate_order} provides 
the convergence rates for the positive discrete eigenvalues $\mu_h$ of the operator sequence $\{\mathcal{R}_h\}_{h\in\wedge}$ to the eigenvalue $\mu$ of the operator $\mathcal{R}$. Specifically, the convergence order of $\mu_h$ to $\mu$ with respect to the mesh size $h$ is $2\min\{t, r\}$, where $t$ denotes the regularity of the eigenfunctions in $\ker(\mu-\mathcal{R})$.
Consequently, the following theorem 
essentially characterizes the gap between the direct sum of the eigenvector spaces 
of all $\mu_h$ and the eigenspace $\ker(\mu-\mathcal{R})$, 
which provides the corresponding convergence order for the eigenfunction approximations. 
\begin{thm}\label{R-R_h-eigenvector-approximate_order}
Under the conditions of Theorem \ref{thm:R-R_h-multi-eigenvalue-eigenvector-approximate_order},
let $\mu_h$ be an eigenvalue of the
operator $\mathcal{R}_h$ such that $\lim\limits_{h\to0^+}\mu_h=\mu$, and let $\boldsymbol{\xi}_h\in \mathcal{H}_0^h$ satisfy 
\begin{align}
\mathcal{R}_h\boldsymbol{\xi}_h=\mu_h\boldsymbol{\xi}_h, \quad\|\boldsymbol{\xi}_h\|_{\mathbf{L}^2(\Omega)}=1, \quad\forall h\in \wedge.
\end{align}
Then there exist a constant $C$ depending on $t$ and $\Omega$ and an eigenfunction $\boldsymbol{\xi}\in\ker(\mu-\mathcal{R})$ of $\mathcal{R}$ corresponding to $\mu$ such that
\begin{align}
\|\boldsymbol{\xi}-\boldsymbol{\xi}_h\|_{\mathbf{H}(\mathbf{curl};\Omega)}\leq C\sqrt{m}h^{\min\{t, r\}}\max_{1\leq j\leq m}\|\boldsymbol{\xi}_j\|_{\mathbf{H}^t(\mathbf{curl};\Omega)}.
\end{align}
\end{thm}

\begin{proof}
By Lemma \ref{operator-eigenvector-order}, there is a constant $C_1$ and a function $\boldsymbol{\xi}\in\ker(\lambda-\mathcal{R})$ such that
\begin{align}
\label{eq:R-R_h-eigenvector-approximate}
\|\boldsymbol{\xi}-\boldsymbol{\xi}_h\|_{\mathbf{L}^2(\Omega)}\leq C_1\|(\mathcal{R}-\mathcal{R}_h)|_{\ker(\lambda-\mathcal{R})}\|, \quad\forall h\in \wedge.
\end{align}
In the proof of Theorem \ref{thm:R-R_h-multi-eigenvalue-eigenvector-approximate_order}, we have already derived the estimate \eqref{eq:multi-eigenvalue-estimate2} that there exists a constant $C_2$ such that
\begin{align}
\|(\mathcal{R}-\mathcal{R}_h)|_{\ker(\lambda-\mathcal{R})}\|^2 \leq mC_2^2\mu^2 h^{2\min\{t, r\}}\max_{1\leq j\leq m}\|\boldsymbol{\xi}_j\|_{\mathbf{H}^t(\mathbf{curl};\Omega)}^2.
\end{align}
Substituting this into \eqref{eq:R-R_h-eigenvector-approximate} yields the following $\mathbf{L}^2(\Omega)$-norm estimate for $\boldsymbol{\xi}-\boldsymbol{\xi}_h$
\begin{align}
\label{xi-xi_h-eigenvector-L2-approximate_order}
\|\boldsymbol{\xi}-\boldsymbol{\xi}_h\|_{\mathbf{L}^2(\Omega)} \leq \sqrt{m}C_1C_2\mu h^{\min\{t, r\}}\max_{1\leq j\leq m}\|\boldsymbol{\xi}_j\|_{\mathbf{H}^t(\mathbf{curl};\Omega)}.
\end{align}
The $\mathbf{H}(\mathbf{curl};\Omega)$-norm error
$\|\boldsymbol{\xi}-\boldsymbol{\xi}_h\|_{\mathbf{H}(\mathbf{curl};\Omega)}$ can be decomposed into the sum of following three terms
\begin{align}\label{R-R_h-eigenvector-approximate_order-I123}
&\|\boldsymbol{\xi}-\boldsymbol{\xi}_h\|_{\mathbf{H}(\mathbf{curl};\Omega)} 
=\left\|\frac{1}{\mu}\mathcal{R}\boldsymbol{\xi}-\frac{1}{\mu_h}\mathcal{R}_h\boldsymbol{\xi}_h\right\|_{\mathbf{H}(\mathbf{curl};\Omega)} \nonumber\\
=&\left\|\frac{1}{\mu}\mathcal{R}\boldsymbol{\xi}-\frac{1}{\mu_h}\mathcal{R}\boldsymbol{\xi} +\frac{1}{\mu_h}\mathcal{R}\boldsymbol{\xi} -\frac{1}{\mu_h}\mathcal{R}_h\boldsymbol{\xi} +\frac{1}{\mu_h}\mathcal{R}_h\boldsymbol{\xi} -\frac{1}{\mu_h}\mathcal{R}_h\boldsymbol{\xi}_h\right\|_{\mathbf{H}(\mathbf{curl};\Omega)} \nonumber\\
\leq& \left|\frac{1}{\mu}-\frac{1}{\mu_h}\right|\|\mathcal{R}\boldsymbol{\xi}\|_{\mathbf{H}(\mathbf{curl};\Omega)} +\frac{1}{\mu_h}\|\mathcal{R}\boldsymbol{\xi} -\mathcal{R}_h\boldsymbol{\xi}\|_{\mathbf{H}(\mathbf{curl};\Omega)}\nonumber\\
&+\frac{1}{\mu_h}\|\mathcal{R}_h\boldsymbol{\xi}-\mathcal{R}_h\boldsymbol{\xi}_h\|_{\mathbf{H}(\mathbf{curl};\Omega)}.
\end{align}
These three terms represent the eigenvalue approximation error, the operator approximation error, and the discrete solution operator error, respectively.

For sufficiently small $h$,
taking $f=\boldsymbol{\xi}$ in \eqref{eq:Rf-form-control} and combining it with \eqref{xi-xi_h-eigenvector-L2-approximate_order} yields
\begin{align}
\label{eq:eta_h-L2-estimate}
\|\mathcal{R}\boldsymbol{\xi}\|_{\mathbf{H}(\mathbf{curl};\Omega)}
\leq\|\boldsymbol{\xi}\|_{\mathbf{L}^2(\Omega)}
\leq \|\boldsymbol{\xi}-\boldsymbol{\xi}_h\|_{\mathbf{L}^2(\Omega)}+\|\boldsymbol{\xi}_h\|_{\mathbf{L}^2(\Omega)}
=O(h^{\min\{t, r\}})+1\leq 2.
\end{align}
By Theorem \ref{thm:R-R_h-multi-eigenvalue-eigenvector-approximate_order}, there exists a constant $C_3$ such that
\begin{align}
\left|\frac{1}{\mu}-\frac{1}{\mu_h}\right|=\frac{|\mu-\mu_h|}{\mu\mu_h}\leq \frac{2}{\mu^2}C_3m(m+1)\mu^2h^{2\min\{t, r\}}\max_{1\leq j\leq m}\|\boldsymbol{\xi}_j\|_{\mathbf{H}^t(\mathbf{curl};\Omega)}^2.
\end{align}
Combining this with \eqref{eq:eta_h-L2-estimate},
the first term of \eqref{R-R_h-eigenvector-approximate_order-I123} satisfies
\begin{align}
\label{R-R_h-eigenvector-approximate_order-I1}
\left|\frac{1}{\mu}-\frac{1}{\mu_h}\right|\|\mathcal{R}\boldsymbol{\xi}\|_{\mathbf{H}(\mathbf{curl};\Omega)}
\leq 4C_3m(m+1)h^{2\min\{t, r\}}\max_{1\leq j\leq m}\|\boldsymbol{\xi}_j\|_{\mathbf{H}^t(\mathbf{curl};\Omega)}^2.
\end{align}
Furthermore,
since $\{\boldsymbol{\xi}_1, \ldots, \boldsymbol{\xi}_m\}$ is an $\mathbf{L}^2(\Omega)$-orthonormal basis of $\ker(\mu-\mathcal{R})$, 
$\boldsymbol{\xi}$ can be expressed as $\boldsymbol{\xi}=\sum_{i=1}^mc_{h, i}\boldsymbol{\xi}_i$ where $c_{h, i}\in\mathbb{R}, 1\leq i\leq m$ satisfy 
$\sum_{i=1}^m c_{h, i}^2=\|\boldsymbol{\xi}\|_{\mathbf{L}^2(\Omega)}^2\leq4$
according to \eqref{eq:eta_h-L2-estimate}.
In the proof of Theorem \ref{thm:R-R_h-multi-eigenvalue-eigenvector-approximate_order}, we have derived the estimate 
\eqref{eq:eigenvector-Rxi_j-Rhxi_j-order} that for each $\boldsymbol{\xi}_i$, there is a constant $C_{4, i}$ (corresponding to the constant $C_{2,j}$ in \eqref{eq:eigenvector-Rxi_j-Rhxi_j-order}) such that 
\begin{align}\label{eq:eigenvector-Rxi_i-Rhxi_i-order}
\|\mathcal{R}\boldsymbol{\xi}_i-\mathcal{R}_{h}\boldsymbol{\xi}_i\|_{\mathbf{H}(\mathbf{curl};\Omega)} 
\leq C_{4, i}\mu h^{\min\{t, r\}}\|\boldsymbol{\xi}_i\|_{\mathbf{H}^t(\mathbf{curl};\Omega)}.
\end{align}
Let $C_4=\max\limits_{1\leq i\leq m}C_{4, i}$,
then \eqref{eq:eigenvector-Rxi_i-Rhxi_i-order} yields that
\begin{align}
&\|\mathcal{R}\boldsymbol{\xi}-\mathcal{R}_{h}\boldsymbol{\xi}\|_{\mathbf{H}(\mathbf{curl};\Omega)}^2 
=\bigg\|\sum_{i=1}^mc_{h, i}(\mathcal{R}-\mathcal{R}_h)\boldsymbol{\xi}_i\bigg\|_{\mathbf{H}(\mathbf{curl};\Omega)}^2\nonumber\\
&\leq m\sum_{i=1}^mc_{h, i}^2\|(\mathcal{R}-\mathcal{R}_h)\boldsymbol{\xi}_i\|_{\mathbf{H}(\mathbf{curl};\Omega)}^2\leq m\sum_{i=1}^mc_{h, i}^2C_{4, i}^2\mu^2 h^{2\min\{t, r\}}\|\boldsymbol{\xi}_i\|_{\mathbf{H}^t(\mathbf{curl};\Omega)}^2\nonumber\\
&\leq 4mC_{4}^2\mu^2 h^{2\min\{t, r\}}\max_{1\leq j\leq m}\|\boldsymbol{\xi}_j\|_{\mathbf{H}^t(\mathbf{curl};\Omega)}^2.
\quad\quad\quad 
\end{align}
Therefore, the second term of \eqref{R-R_h-eigenvector-approximate_order-I123} satisfies
\begin{align}
\label{R-R_h-eigenvector-approximate_order-I2}
\frac{1}{\mu_h}\|\mathcal{R}\boldsymbol{\xi}-\mathcal{R}_h\boldsymbol{\xi}\|_{\mathbf{H}(\mathbf{curl};\Omega)}
&\leq \frac{2}{\mu}\cdot2\sqrt{m}C_4\mu h^{\min\{t, r\}}\max_{1\leq j\leq m}\|\boldsymbol{\xi}_j\|_{\mathbf{H}^t(\mathbf{curl};\Omega)} \nonumber\\
&=4\sqrt{m}C_4h^{\min\{t, r\}}\max_{1\leq j\leq m}\|\boldsymbol{\xi}_j\|_{\mathbf{H}^t(\mathbf{curl};\Omega)}.
\end{align}
Moreover, taking $f=\boldsymbol{\xi}-\boldsymbol{\xi}_h$ in \eqref{eq:Rhf-form-control} gives
\begin{align}
\|\mathcal{R}_h\boldsymbol{\xi}-\mathcal{R}_h\boldsymbol{\xi}_h\|_{\mathbf{H}(\mathbf{curl};\Omega)} 
\leq\|\boldsymbol{\xi}-\boldsymbol{\xi}_h\|_{\mathbf{L}^2(\Omega)}.
\end{align}
Combining this with the $\mathbf{L}^2(\Omega)$-estimate \eqref{xi-xi_h-eigenvector-L2-approximate_order}, 
the third term of \eqref{R-R_h-eigenvector-approximate_order-I123} satisfies
\begin{align}
\label{R-R_h-eigenvector-approximate_order-I3}
\frac{1}{\mu_h}\|\mathcal{R}_h\boldsymbol{\xi}-\mathcal{R}_h\boldsymbol{\xi}_h\|_{\mathbf{H}(\mathbf{curl};\Omega)}
&\leq \frac{2}{\mu}\sqrt{m}C_1C_2\mu h^{\min\{t, r\}}\max_{1\leq j\leq m}\|\boldsymbol{\xi}_j\|_{\mathbf{H}^t(\mathbf{curl};\Omega)} \nonumber\\
&= 2\sqrt{m}C_1C_2 h^{\min\{t, r\}}\max_{1\leq j\leq m}\|\boldsymbol{\xi}_j\|_{\mathbf{H}^t(\mathbf{curl};\Omega)}.
\end{align}
Substituting \eqref{R-R_h-eigenvector-approximate_order-I1}, \eqref{R-R_h-eigenvector-approximate_order-I2} 
and \eqref{R-R_h-eigenvector-approximate_order-I3} into \eqref{R-R_h-eigenvector-approximate_order-I123}, 
since the first term is a higher order  
than those of the other two terms for suﬀiciently small $h$, it follows that
\begin{align}
\|\boldsymbol{\xi}-\boldsymbol{\xi}_h\|_{\mathbf{H}(\mathbf{curl};\Omega)}
\leq C\sqrt{m}h^{\min\{t, r\}}\max_{1\leq j\leq m}\|\boldsymbol{\xi}_j\|_{\mathbf{H}^t(\mathbf{curl};\Omega)}, 
\end{align}
where the constant $C=4(C_4+C_1C_2)$ depends on $t$ and $\Omega$.
\end{proof}

\section{Numerical results}\label{Section_Numerical}
This section presents numerical examples on three polyhedrons $\Omega$ 
to validate the proposed numerical method and the derived theoretical error estimates.  
For this aim, we begin by creating an initial mesh $\mathcal{T}_H$ consisting of regularly  
distributed tetrahedral elements over $\Omega$.
Then, the initial mesh $\mathcal{T}_H$ is uniformly refined several 
times to obtain fine meshes $\mathcal{T}_h$. 
On these meshes, through solving the discrete mixed Problem \ref{discrete-mixed-problem} 
to obtain the approximate eigenvalues 
and eigenfunctions for the Maxwell eigenvalue problem \eqref{maxwell-eigenvalue-problem}, 
where the spaces are selected as $\mathbf{X}_{h, \tau0}=\mathscr{V}_{h, \tau0}^r$ 
and $Y_{h, 0}=\mathscr{L}_{h, 0}^{r+1}$. In all examples here, we will test the two cases of 
$r=1$ and $r=2$. 
The resulting generalized algebraic eigenvalue problem $Ax=\lambda Bx$ has been solved using the
Krylov-Schur method from SLEPc \cite{key21} 
and the General Conjugate Gradient Eigensolver (GCGE)  \cite{key47,key48,key49}.
Furthermore, the convergence criterion is set to be $\|Ax-\lambda Bx\|_2\leq1.0 \times 10^{-10}$, 
where $\|\cdot\|_2$ denotes the $\mathrm{L}^2$-norm for the vectors.

The numerical examples are carried out on LSSC-IV in the State Key Laboratory 
of Scientific and Engineering Computing, Chinese Academy of Sciences.
Each computing node has two 18-core Intel Xeon Gold 6140 processors at 2.3 GHz 
and 192 GB memory. 
The matrices are produced by the package of open parallel finite element method(OpenPFEM). 
OpenPFEM provides the parallel finite element discretization 
for the partial differential equations and corresponding eigenvalue problems. 
The source code can be downloaded from  \href{https://gitlab.com/xiegroup}{https://gitlab.com/xiegroup}.

\subsection{Unit cube domain}\label{Section_Example_Cube}
In the first example, we consider the Maxwell eigenvalue problems defined on the unit cube. 
\begin{figure}[htbp]
\centering
\includegraphics[width=0.4\textwidth]{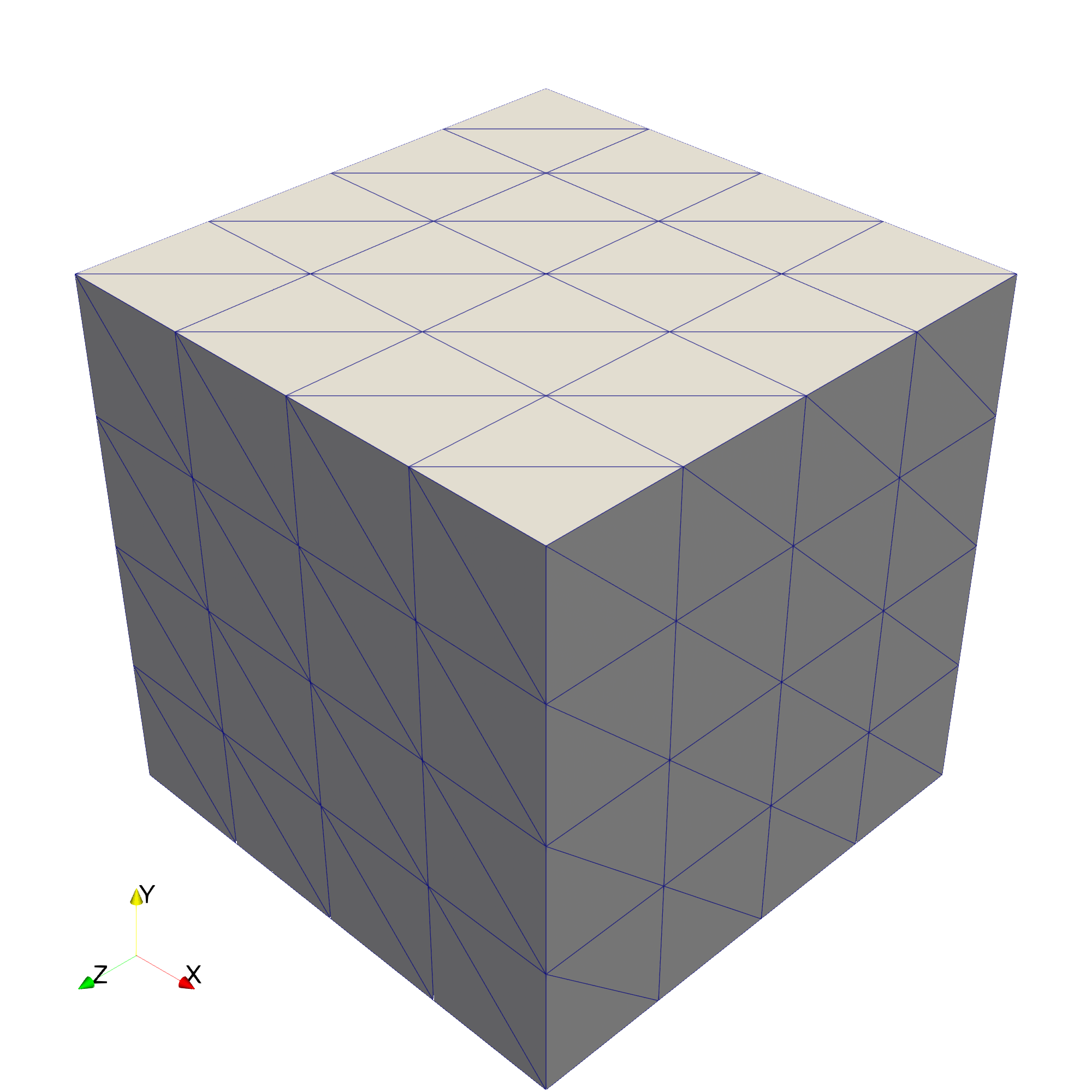}
\caption{The unit cube domain with the mesh size $h=0.25$}
\label{fig:Unit cube domain}
\end{figure}
The unit cube $\Omega = (0,1)^3$ is selected as the computational 
domain for the first numerical example. 
A regular structured mesh consisting of $384$ tetrahedral elements 
with a diameter of $h=0.25$ which is shown in Figure \ref{fig:Unit cube domain}
serves as a coarse partition. 

On the unit cube $\Omega = (0,1)^3$, the exact eigenvalues of 
the Maxwell eigenvalue problem \eqref{maxwell-eigenvalue-problem} can be analytically expressed as 
\begin{align}
\lambda_{m, n, \ell}=(m^2+n^2+\ell^2)\pi^2,
\end{align}
where $m, n, \ell \in \mathbb{N}$ and at least two indices are nonzero.
The first 11 eigenvalues are $(2, 2, 2, 3, 3, 5, 5, 5, 5, 5, 5)\pi^2$, respectively.
For each triplet $(m, n, \ell)$, the set of linearly independent eigenfunctions corresponding to $\lambda_{m, n, \ell}$ is
\begin{eqnarray}\label{eq:Unit-cubes-exact-eigenvector}
\left[
\begin{array}{l}
A\cos(m\pi x)\sin(n\pi y)\sin(\ell\pi z)  \\
B\cos(n\pi y)\sin(m\pi x)\sin(\ell\pi z)  \\
C\cos(\ell\pi z)\sin(m\pi x)\sin(n\pi y)
\end{array}
\right],
\end{eqnarray}
where the coefficients $(A, B, C)$ satisfying $Am+Bn+C\ell=0$. 
Specifically, the three linearly independent eigenfunctions corresponding to the eigenvalue $2\pi^2$ are
\begin{eqnarray}
\label{eq:Unit-cubes-2pi2-eigenvectors}
\left\{
\begin{array}{l}
{\boldsymbol{\xi}}_1(x, y, z)=[0, 0, \sin(\pi x)\sin(\pi y)]^\top,  \\
{\boldsymbol{\xi}}_2(x, y, z)=[0, \sin(\pi x)\sin(\pi z), 0]^\top,  \\
{\boldsymbol{\xi}}_3(x, y, z)=[\sin(\pi y)\sin(\pi z), 0, 0]^\top.
\end{array}
\right.
\end{eqnarray}

The method for computing the distance between the discrete eigenspace and the exact one is as follows. 
Taking the first eigenfunction approximation $\boldsymbol{\xi}_h$ 
as an example,  
we first perform its $\mathbf{L}^2$ projection onto the space spanned by the three eigenfunctions $\boldsymbol{\xi}_1$, 
$\boldsymbol{\xi}_2$ and $\boldsymbol{\xi}_3$ in \eqref{eq:Unit-cubes-2pi2-eigenvectors}.
The projection is denoted as 
\begin{align}
\mathbf{P}\boldsymbol{\xi}_h:=a_1\boldsymbol{\xi}_1+a_2\boldsymbol{\xi}_2+a_3\boldsymbol{\xi}_3,   
\end{align}
satisfying the following Galerkin condition 
\begin{align}\label{projection-inner product}
(\mathbf{P}\boldsymbol{\xi}_h,\boldsymbol{\xi}_i)=(\boldsymbol{\xi}_h,\boldsymbol{\xi}_i),\ \ \ ~\forall i=1,2,3,
\end{align}
from which the values of $a_i,i=1,2,3$ can be solved.
Besides, it holds that
\begin{align}\label{Hcurl-projection-inner product}
&(\mathbf{P}\boldsymbol{\xi}_h,\boldsymbol{\xi}_i)_{\mathbf{H}(\mathbf{curl};\Omega)}
=(\mathbf{P}\boldsymbol{\xi}_h,\nabla\times\nabla\times\boldsymbol{\xi}_i+\boldsymbol{\xi}_i)
=(2\pi^2+1)(\mathbf{P}\boldsymbol{\xi}_h,\boldsymbol{\xi}_i) \nonumber\\
=&(2\pi^2+1)(\boldsymbol{\xi}_h,\boldsymbol{\xi}_i)
=(\boldsymbol{\xi}_h,\nabla\times\nabla\times\boldsymbol{\xi}_i+\boldsymbol{\xi}_i)
=(\boldsymbol{\xi}_h,\boldsymbol{\xi}_i)_{\mathbf{H}(\mathbf{curl};\Omega)}
,\ \ \ ~\forall i=1,2,3,
\end{align}
thus $\mathbf{P}\boldsymbol{\xi}_h$ is also the $\mathbf{H}(\mathbf{curl};\Omega)$ projection of $\boldsymbol{\xi}_h$ onto the eigenfunction space.
The relative $\mathbf{L}^2(\Omega)$ error and $\mathbf{H}(\mathbf{curl};\Omega)$ error of  $\boldsymbol{\xi}_h$ is obtained by computing
\begin{align}
\label{eq:L2-eigvec-error}
\frac{\|\boldsymbol{\xi}_h-\mathbf{P}\boldsymbol{\xi}_h\|_{\mathbf{L}^2(\Omega)}}
{\|\boldsymbol{\xi}_h\|_{\mathbf{L}^2(\Omega)}} 
\quad \mathrm{and} \quad
\frac{\|\boldsymbol{\xi}_h-\mathbf{P}\boldsymbol{\xi}_h\|_{\mathbf{H}(\mathbf{curl};\Omega)}}
{\|\boldsymbol{\xi}_h\|_{\mathbf{H}(\mathbf{curl};\Omega)}}.
\end{align}

\begin{figure}[ht]
\centering
\includegraphics[width=0.49\textwidth]{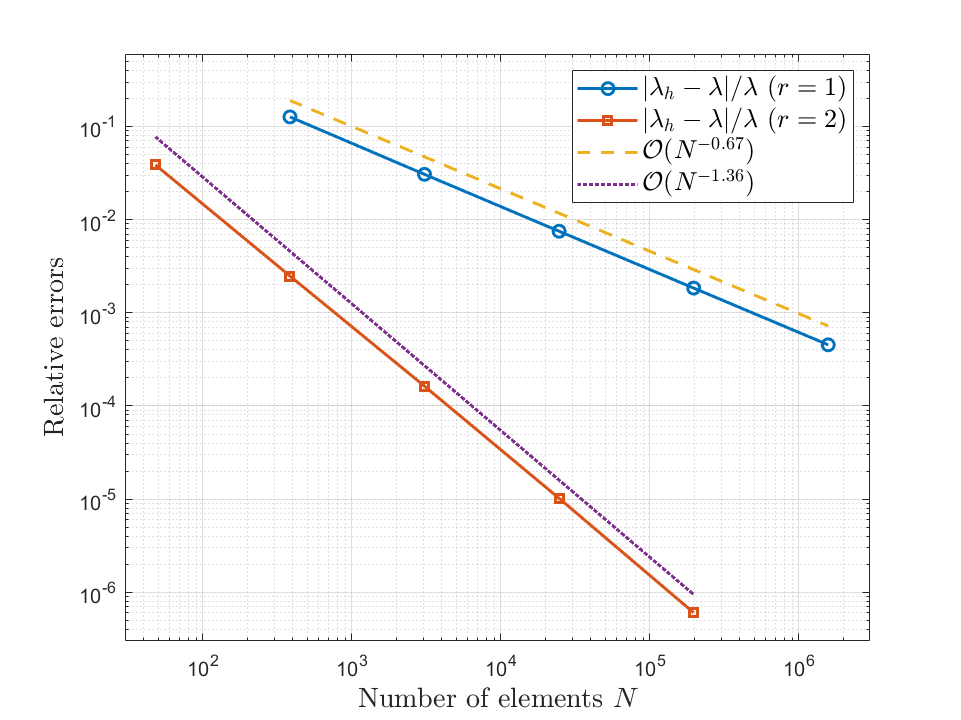}
\includegraphics[width=0.49\textwidth]{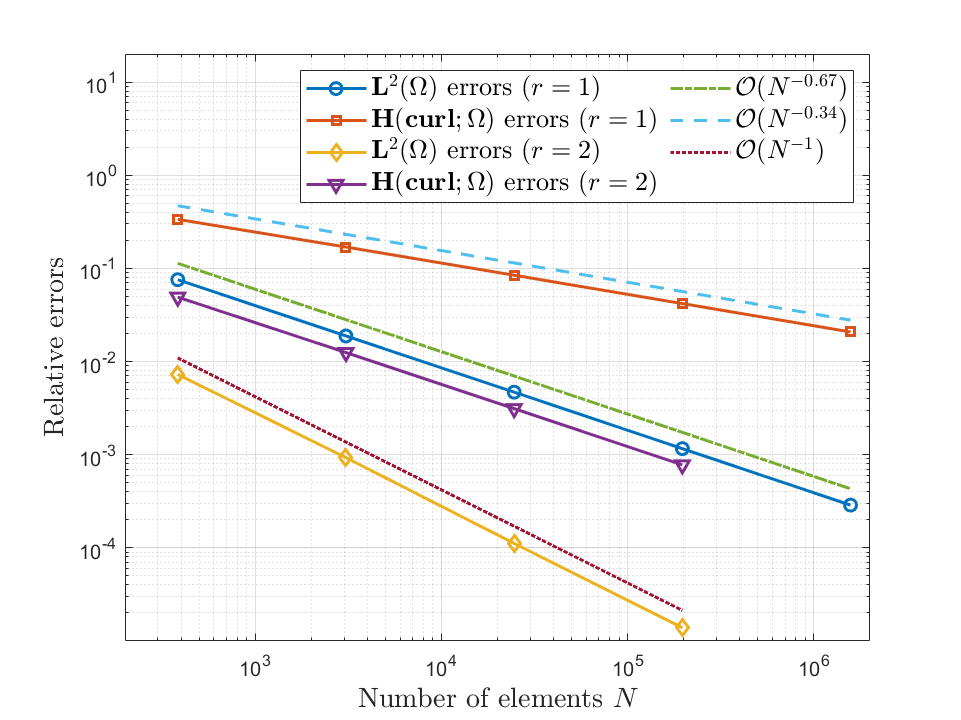}
\caption{Errors for the smallest positive eigenvalue and associated eigenfunction for the Maxwell eigenvalue 
problem on the unit cube domain}\label{fig:cube 1}
\end{figure}
\begin{figure}[ht]
\centering
\includegraphics[width=0.49\textwidth]{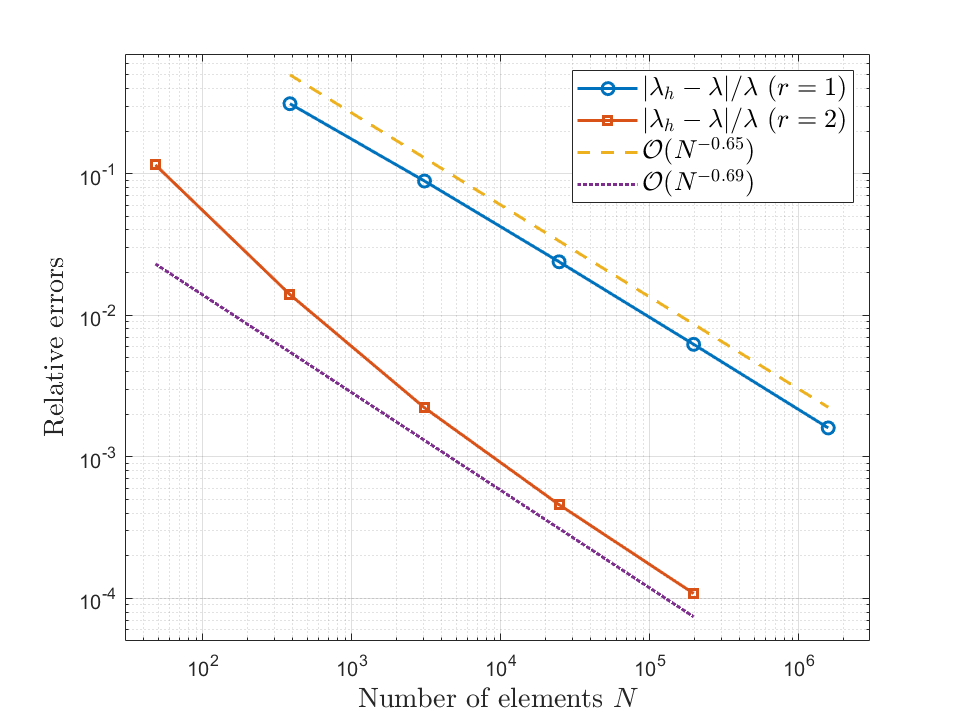}
\includegraphics[width=0.49\textwidth]{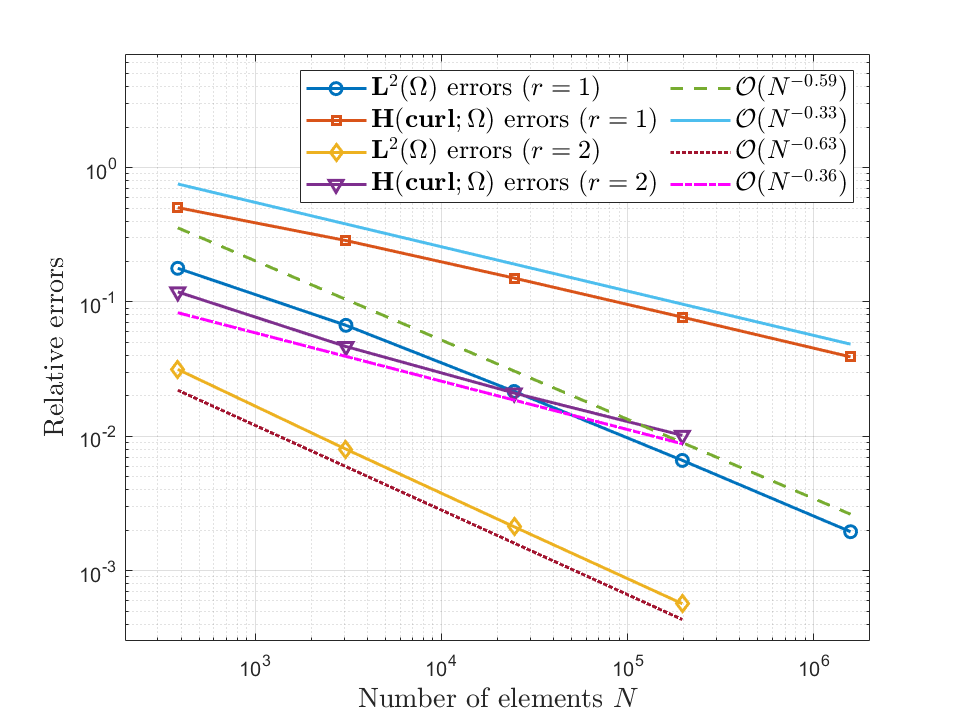}
\caption{Errors for the smallest positive eigenvalue and associated eigenfunction for the Maxwell eigenvalue 
problem on the unit cube domain, when $\mathbf{X}_{h, \tau0}$ is chosen as $\mathscr{V}_{h, \tau0}^r\cap \mathbf{H}_0^1(\Omega)$}
\label{fig:cube H01}
\end{figure}
Figure \ref{fig:cube 1} displays two log–log plots of the relative errors for 
the smallest eigenvalue approximation $\lambda_h$, and the relative $\mathbf{L}^2(\Omega)$ 
errors and $\mathbf{H}(\mathbf{curl};\Omega)$ 
errors \eqref{eq:L2-eigvec-error} 
for the corresponding eigenfunction approximation $\boldsymbol{\xi}_h$, 
versus the number $N$ of elements of the meshes, with $r=1$ and $r=2$.


Since the Maxwell eigenfunctions on the unit cube are smooth enough, 
the eigenpair approximations should have the full convergence order 
presented in Theorems \ref{thm:R-R_h-multi-eigenvalue-eigenvector-approximate_order} 
and \ref{R-R_h-eigenvector-approximate_order}. 
As can be seen from Figure \ref{fig:cube 1}, the full convergence orders for the eigenvalue errors and eigenfunction errors in the sense of $\mathbf{H}(\mathbf{curl};\Omega)$-norm are consistent with the theoretical results.
An interesting phenomenon observable in Figure \ref{fig:cube 1} is that, 
regardless of whether $r=1$ or $r=2$, the convergence rate of 
the discrete eigenfunction in the $\mathbf{L}^2(\Omega)$-norm has one order higher than that
in the $\mathbf{H}(\mathbf{curl};\Omega)$-norm,
which suggests that a duality argument (Aubin-Nitsche trick) might be applicable.
This superconvergence result currently lacks theoretical investigation and is our future work.

Meanwhile, we also conducted supplementary numerical experiments by selecting $\mathbf{X}_{h, \tau0}$ 
as $\mathscr{V}_{h, \tau0}^r\cap \mathbf{H}_0^1(\Omega)$, and presented two corresponding 
log–log plots in Figure \ref{fig:cube H01}. 
As mentioned in Section \ref{Section_Intro}, such a discrete scheme, which is constrained by 
the regularity of the potential function, fails to achieve full convergence order when $r=2$.

\subsection{Thick L-shaped domain}
In this example, the computational domain is the so-called thick L-shaped domain 
$\Omega={(-1, 1)^2\times(0, 1)}\setminus[0, 1]^3$, 
a non-convex Lipschitz polyhedron with an edge having an opening angle of $3\pi/2$. 
A regular structured mesh consisting of $1152$ tetrahedral elements 
with a diameter of $h=0.25$ which is shown in Figure \ref{fig:Thick L-shaped domain} 
serves as a coarse partition. 
\begin{figure}[ht]
\centering
\includegraphics[width=6cm,height=6cm]{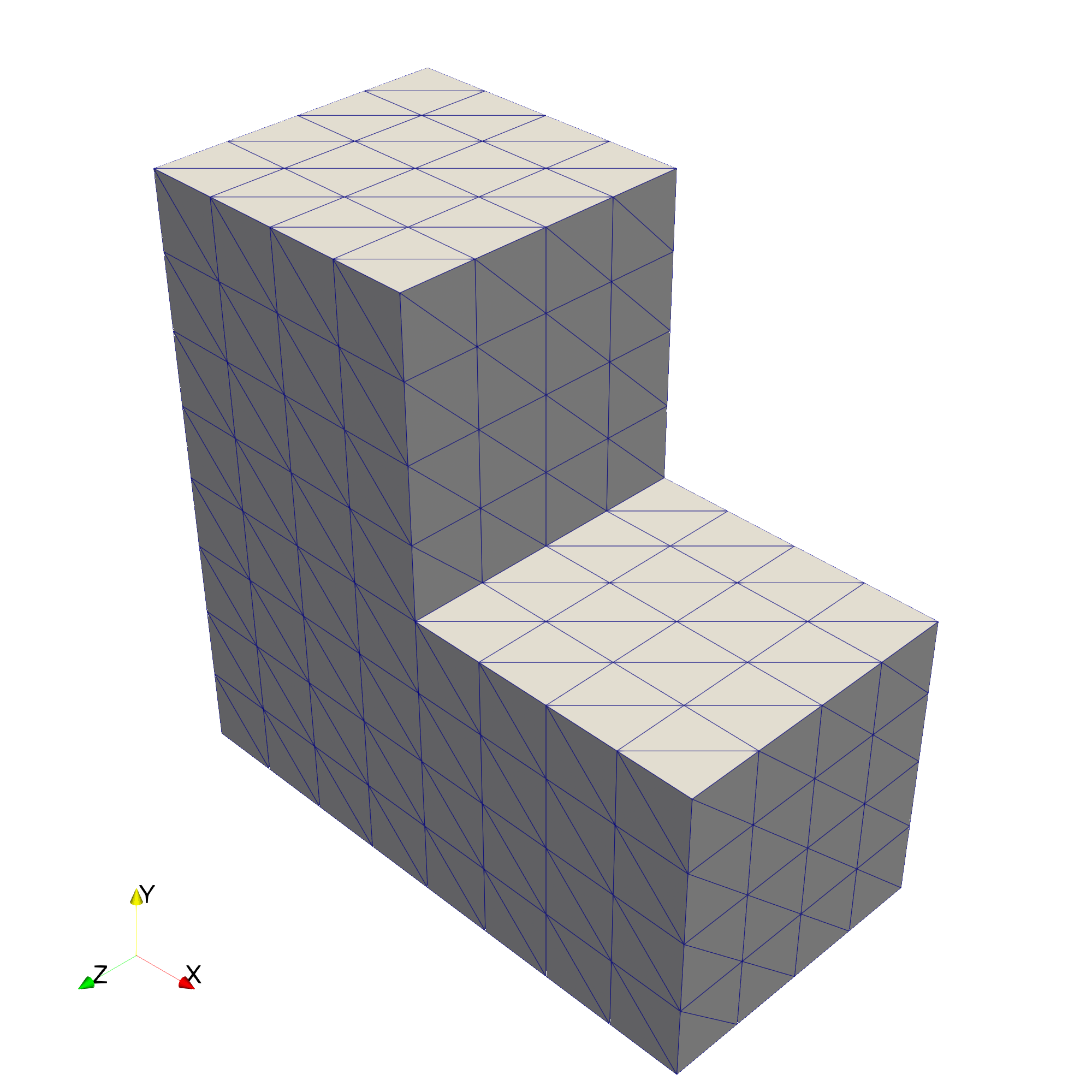}
\caption{The Thick L-shaped domain with the mesh size $h=0.25$}\label{fig:Thick L-shaped domain}
\end{figure}

Maxwell eigenfunctions on the thick L-shaped domain $\Omega$ may not belong 
to $\mathbf{H}^1(\Omega)$. However, according to Theorem \ref{maxwell-eigenvector-discrete-approximate}, 
there exists $s>1/2$ such that all Maxwell eigenfunctions 
belong to $\mathbf{H}^s(\Omega)$. In fact, due to the $3\pi/2$ 
opening angle at the non-convex edge of the thick L-shaped domain $\Omega$, 
the upper regularity bound $s$ for spaces $X_N(\Omega)$ 
and $X_T(\Omega)$ can be arbitrarily close to $2/3$ \cite{key18}.

If we denote the 2D L-shaped domain as $\Omega_2={(-1, 1)^2}\setminus[0, 1]^2$, 
then the eigenvalues of the Maxwell problem on the 
thick L-shaped domain $\Omega$ can be expressed either as the sum 
of a nonzero Maxwell eigenvalue on $\Omega_2$ and a Dirichlet 
eigenvalue on $[0, 1]$ (which takes the form $n^2\pi^2$ for all positive integers $n$), 
or as the sum of a Laplace eigenvalue on $\Omega_2$ and any Neumann eigenvalue 
on $[0, 1]$ (namely, $n^2\pi^2$ for all positive integers $n$). Specifically, 
the first three high-precision Laplace eigenvalues on $\Omega_2$ 
are $9.63972384$, $15.19725193$ and $19.73920880$, 
while the first five high-precision nonzero Maxwell eigenvalues 
on $\Omega_2$ are $1.47562182$, $3.53403137$, $9.86960440$, $9.86960440$, 
and $11.38947940$ \cite{dauge_benchmark}.

Figure \ref{fig:L 1} presents a log-log plot of the relative errors 
between the smallest eigenvalue approximation and the corresponding 
high-precision eigenvalue $9.63972384$ 
versus the number $N$ of elements of the meshes, with $r=1$ and $r=2$.
\begin{figure}[ht]
\centering
\includegraphics[width=0.6\textwidth]{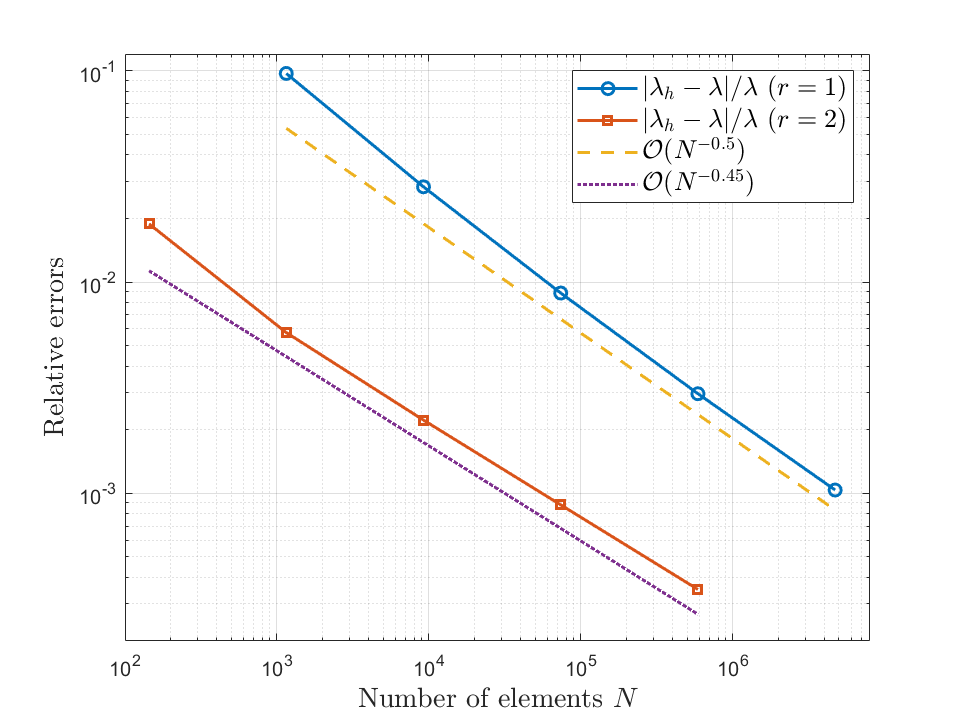}
\caption{Errors of the smallest eigenvalue approximation for the Maxwell 
eigenvalue problem on the thick L-shaped domain}\label{fig:L 1}
\end{figure}

Unlike the case where the computing domain is the unit cube, 
Figure \ref{fig:L 1} shows that the eigenvalue approximation has only a low 
slope of approximately $-0.45$ when $r=2$.
This indicates that the errors for the smallest eigenvalue 
computed on the sequence of uniformly refined meshes has following convergence order  
\begin{align}
|\lambda-\lambda_h| \approx \mathcal{O}\left(N^{-0.45}\right) = \mathcal{O}(h^{2t})
\end{align}
with $t \approx 2/3$, which aligns with the mentioned fact in Section 
\ref{Section_Pre} that the upper regularity bound for 
the space $\mathcal{H}_0$ on the thick L-shaped domain $\Omega$ 
can be arbitrarily close to $2/3$.

As a special case, we point out that on the thick L-shaped domain $\Omega$, 
the eigenfunctions corresponding to the 6th, 7th and 8th smallest eigenvalues 
$2\pi^2$ are smooth analytic functions.
In this scenario, the convergence order of the eigenvalue approximations returns 
to the optimal $\mathcal{O}(h^{2r})$, as shown in Figure \ref{fig:L 6}.
\begin{figure}[ht]
\centering
\includegraphics[width=0.6\textwidth]{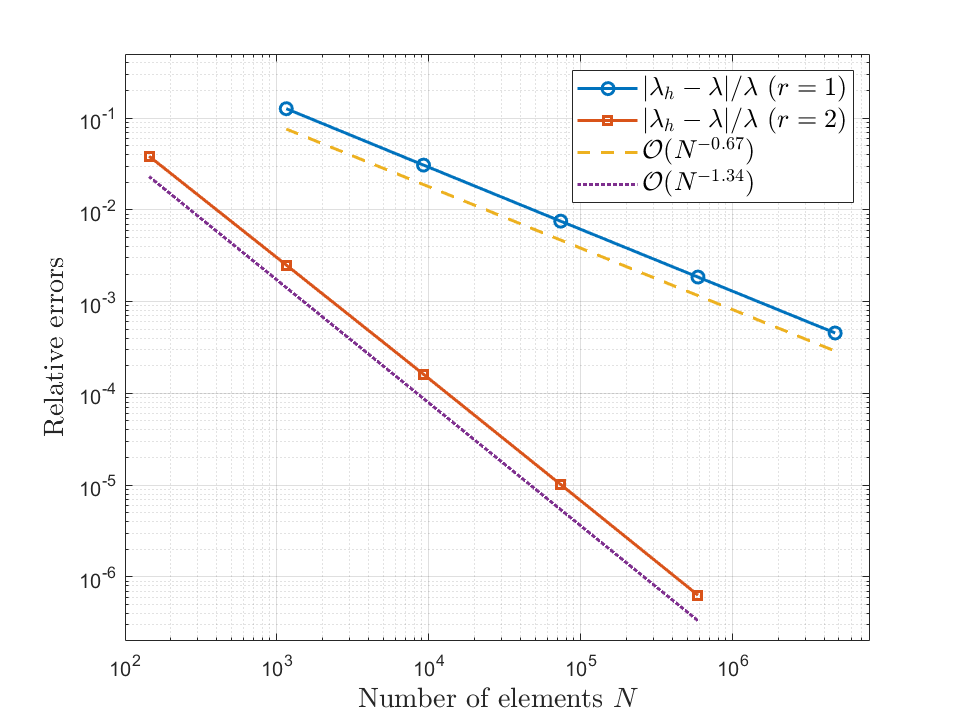}
\caption{Errors for the 6-th smallest eigenvalue of the Maxwell 
eigenvalue problem on the thick L-shaped domain}\label{fig:L 6}
\end{figure}

\subsection{Fichera corner}
We choose the so-called Fichera corner $\Omega=(-1, 1)^3\setminus[0, 1]^3$ 
as the domain for the final numerical example,  
which has one concave corner and three re-entrant edges.
The Fichera corner is the prototype of a domain where edge and corner singularities interact.
It is named after Gaetano Fichera who provided an approximation of $0.45418$ for the first corner singular 
exponent of the Fichera corner in 1993 \cite{dauge_benchmark}.
A regular structured mesh consisting of $2688$ tetrahedral elements with a diameter of $h=0.25$, 
which is shown in Figure \ref{fig:Fichera corner}, serves as a coarse partition in this example.
\begin{figure}[ht]
\centering
\includegraphics[width=0.4\textwidth]{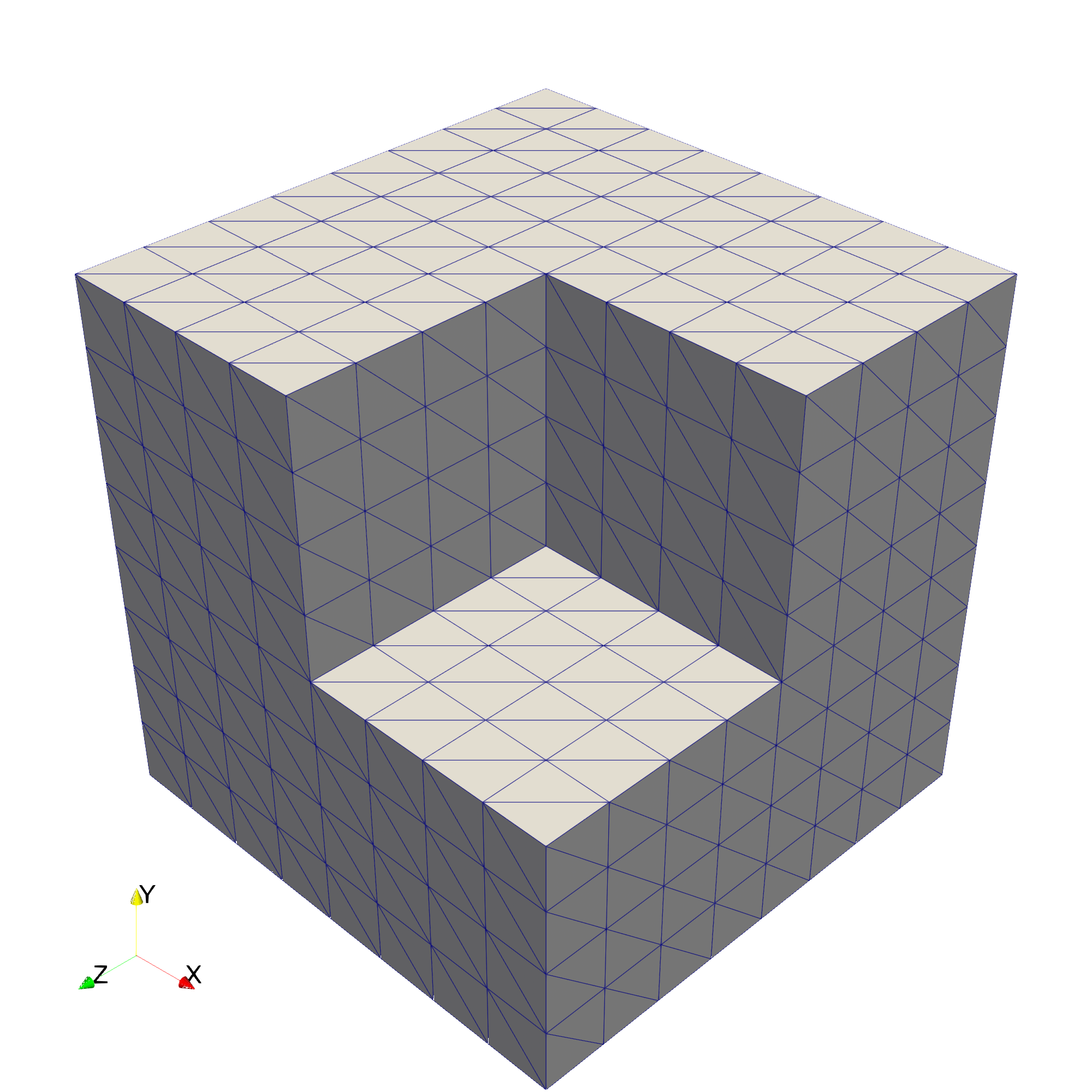}
\caption{The Fichera corner with the mesh size $h=0.25$}\label{fig:Fichera corner}
\end{figure}

As the exact Maxwell eigenvalues on the Fichera corner $\Omega$ are unknown, 
we use the computational results from Monique Dauge's website as a benchmark \cite{fichera_benchmark,dauge_benchmark}. 
These results, posted by Marc Durufl\'{e} in 2006, 
are listed in Table \ref{tab:Fichera corner benchmark}.
\begin{table}[ht]
\centering
\caption{Benchmark for the first eight Maxwell eigenvalues on the Fichera corner,
along with the number of hopefully reliable digits, which are marked in blue.}
\label{tab:Fichera corner benchmark}
\footnotesize
\begin{tabular}{|c|c|}
\hline
Eigenvalues & Number of hopefully reliable digits \\
\hline
\textcolor{blue}{3.219}87401386 & 4 \\
\textcolor{blue}{5.88041}891178 & 6 \\
\textcolor{blue}{5.88041}891780 & 6 \\
\textcolor{blue}{10.68}54921311 & 4 \\
\textcolor{blue}{10.693}7829409 & 5 \\
\textcolor{blue}{10.693}7829737 & 5 \\
\textcolor{blue}{12.3165}204656 & 6 \\
\textcolor{blue}{12.3165}204669 & 6 \\
\hline
\end{tabular}
\end{table}

Tables \ref{tab:fichera-eigenvalue-P1P2} and \ref{tab:fichera-eigenvalue-P2P3} present 
the first eight eigenvalue approximations computed on the Fichera corner $\Omega$ 
for the cases of $r=1$ and $r=2$, respectively, across a sequence of uniformly refined 
meshes with different mesh sizes $h$.
The numerical results clearly indicate that on the finest mesh with mesh size $h=1/{32}$, 
the initial few digits of the approximated eigenvalues agree with the reference ones provided 
in Table \ref{tab:Fichera corner benchmark}.
\begin{table}[ht]
\centering
\caption{The first eight smallest eigenvalue approximation for the case of finite element spaces $\mathbf{X}_{h, \tau0}=\mathscr{V}_{h, \tau0}^1$ and $Y_{h, 0}=\mathscr{L}_{h, 0}^2$ on uniformly refined meshes 
with different mesh sizes $h$ on the Fichera corner.}\label{tab:fichera-eigenvalue-P1P2}
\footnotesize
\begin{tabular}{|cccccc|}
\hline 
$h$ & $\frac{1}{2}$ & $\frac{1}{4}$ & $\frac{1}{8}$ & $\frac{1}{16}$ & $\frac{1}{32}$ \\
$N$ & $336$ & $2688$ & $21504$ & $172032$ & $1376256$ \\
\hline
$\lambda_{1 h}$ & $5.64529995$ & $3.85659662$ & $3.38194201$ & $3.26129321$ & $\textcolor{blue}{3.2}3077689$ \\
$\lambda_{2 h}$ & $7.29620100$ & $6.24130581$ & $5.97607321$ & $5.90624137$ & $\textcolor{blue}{5.88}758575$ \\
$\lambda_{3 h}$ & $7.62144639$ & $6.34392910$ & $6.00277609$ & $5.91264091$ & $\textcolor{blue}{5.88}910509$ \\
$\lambda_{4 h}$ & $14.68192246$ & $11.70775578$ & $10.97828389$ & $10.77940895$ & $\textcolor{blue}{10}.71822904$ \\
$\lambda_{5 h}$ & $14.82866959$ & $11.74242843$ & $10.98644759$ & $10.78082670$ & $\textcolor{blue}{10}.72108140$ \\
$\lambda_{6 h}$ & $16.33691886$ & $12.02907743$ & $11.04776056$ & $10.79202622$ & $\textcolor{blue}{10}.72195707$ \\
$\lambda_{7 h}$ & $17.37060702$ & $13.85458768$ & $12.73951707$ & $12.43209312$ & $\textcolor{blue}{12.3}4867492$ \\
$\lambda_{8 h}$ & $18.19300319$ & $14.00207790$ & $12.78642002$ & $12.44380619$ & $\textcolor{blue}{12.3}5142510$ \\
\hline
\end{tabular}
\end{table}
\begin{table}[ht]
\centering
\caption{The first eight smallest eigenvalue approximation for the case of 
finite element spaces $\mathbf{X}_{h, \tau0}=\mathscr{V}_{h, \tau0}^2$ 
and $Y_{h, 0}=\mathscr{L}_{h, 0}^3$ on uniformly refined meshes with 
different mesh sizes $h$ on the Fichera corner.}\label{tab:fichera-eigenvalue-P2P3}
\footnotesize
\begin{tabular}{|cccccc|}
\hline 
$h$ & $\frac{1}{2}$ & $\frac{1}{4}$ & $\frac{1}{8}$ & $\frac{1}{16}$ & $\frac{1}{32}$ \\
$N$ & $336$ & $2688$ & $21504$ & $172032$ & $1376256$ \\
\hline
$\lambda_{1 h}$ & $3.22302985$ & $3.21452284$ & $3.21762012$ & $3.21903789$ & $\textcolor{blue}{3.219}56986$ \\
$\lambda_{2 h}$ & $5.91425978$ & $5.88621242$ & $5.88222728$ & $5.88111123$ & $\textcolor{blue}{5.880}69236$ \\
$\lambda_{3 h}$ & $5.91747984$ & $5.88627636$ & $5.88224600$ & $5.88111399$ & $\textcolor{blue}{5.880}69345$ \\
$\lambda_{4 h}$ & $10.87682102$ & $10.73312211$ & $10.70276087$ & $10.69230970$ & $\textcolor{blue}{10.68}819876$\\
$\lambda_{5 h}$ & $10.87787382$ & $10.74241142$ & $10.71158169$ & $10.70077574$ & $\textcolor{blue}{10.69}654641$\\
$\lambda_{6 h}$ & $10.91711684$ & $10.74324978$ & $10.71168815$ & $10.70080437$ & $\textcolor{blue}{10.69}654992$\\
$\lambda_{7 h}$ & $12.54313405$ & $12.34451116$ & $12.32305284$ & $12.31881040$ & $\textcolor{blue}{12.31}739501$\\
$\lambda_{8 h}$ & $12.55191724$ & $12.34495599$ & $12.32328059$ & $12.31883317$ & $\textcolor{blue}{12.31}740448$\\
\hline
\end{tabular}
\end{table}

The first Maxwell eigenvalue $\lambda_1$ is simple, and the corresponding 
eigenfunction has the most singular part possible, 
manifested as the Fichera exponent observable at the corner.
Taking $\lambda_1\approx3.2199$ as reference, Figure \ref{fig:fichera 1} presents 
relative errors for the smallest eigenvalue approximations 
for the cases of $r=1$ and $r=2$.
\begin{figure}[ht]
\centering
\includegraphics[width=0.6\textwidth]{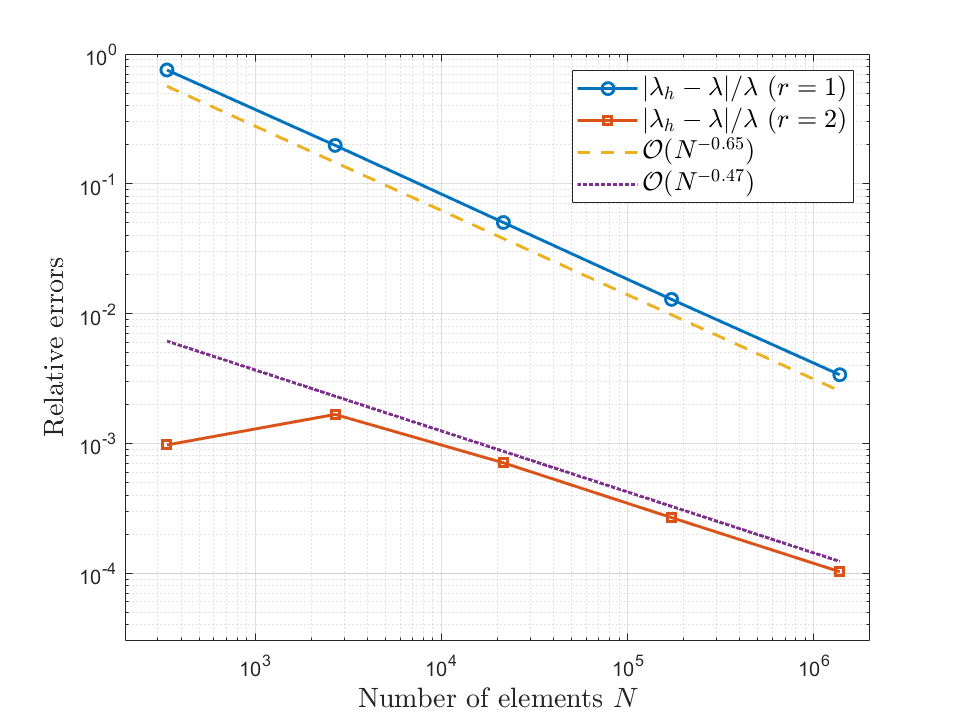}
\caption{Errors for the smallest eigenvalue approximations of the Maxwell eigenvalue problem 
on the Fichera corner.}\label{fig:fichera 1}
\end{figure}

Figure \ref{fig:fichera 1} shows the slope of the error curve corresponding to the case of $r=1$ 
is approximately $-2/3$, while the slope for $r=2$ is approximately $-0.47$.
The approximate $\mathcal{O}(h^2)$ convergence in eigenvalues at $r=1$ is attributed to the fact that the discrete eigenvalues have not yet sufficiently approached the exact one.
Due to computational platform limitations, the results on 
fine enough quasi-uniform meshes have not yet been obtained.
Besides, the slope of the error curve corresponding to the case of $r=2$ indicates 
that the smallest eigenvalue approximations have the following convergence order 
\begin{align*}
|\lambda-\lambda_h| \approx \mathcal{O}(N^{-0.47}) = \mathcal{O}(h^{1.41}).
\end{align*}

\section{Conclusions}
With the novel regular decomposition technique, 
we decompose the space $\mathbf{H}_0^s(\mathbf{curl}; \Omega)$ 
into the sum of a vector potential space and the gradient of a scalar space, both possessing higher regularity. 
Based on this regular decomposition, a type of numerical method 
using continuous Lagrange finite elements for solving Maxwell eigenvalue problems is proposed and analyzed in this paper. 

This new regular decomposition here enables the use of standard continuous finite element spaces, circumventing the implementation complexities associated with high-order edge elements or specialized meshes.
Furthermore, we establish a rigorous convergence analysis, demonstrating that the proposed scheme yields optimal convergence rates for eigenpair approximations and eliminates spurious zero modes.
Numerical experiments on various benchmark domains have validated the efficacy of the proposed method, confirming the theoretical predictions.
Future work will focus on developing efficient eigensolvers specifically tailored to the algebraic eigenvalue problems arising from the discretization scheme in this paper.



\printbibliography

\end{document}